\documentclass[11pt]{article}
\newcommand{\mathpart}[1]{#1}
     \usepackage[margin=1in]{geometry}
      \usepackage[T1]{fontenc}
     \usepackage{lmodern}
    \usepackage{amsmath,amssymb,amsthm,mathtools}
      \numberwithin{equation}{section}
    \usepackage{mathrsfs,microtype,booktabs}
      \clubpenalty%
   =10000
      \widowpenalty%
     =10000
       \usepackage[hidelinks]{hyperref}
     \hypersetup{pdfauthor={Jingchen Hu and Li Sheng},pdftitle={Interior C2 Estimates and a Liouville Theorem for the Complex Monge--Ampere Equation under a Quadratic Lower Bound},pdfsubject={Quadratic lower growth and Bedford--Taylor solutions}}
    \allowdisplaybreaks[2]
     \newtheorem{theorem}{Theorem}[section]
     \newtheorem{proposition}[theorem]{Proposition}
  \newtheorem{lemma}[theorem]{Lemma}
  \newtheorem{corollary}[theorem]{Corollary}
\makeatletter
\def\th@boldremark{\th@remark\thm@headfont{\bfseries}}
\makeatother
\theoremstyle{boldremark}
  \newtheorem{remark}[theorem]{Remark}
   \newcommand{\C}{\mathbb    C}
   \newcommand{\R}{\mathbb R}
     \newcommand{\Id}{I}
  \newcommand{\tr}{\operatorname{tr}}
      \newcommand{\supp}{\operatorname{supp}}
   \newcommand{\sphereavg}{\mathcal S}
 \newcommand{\ballavg}{\mathcal    B}
     \newcommand{\cl}[1]{\overline{#1}}
        \title{Interior $C^{2}$ Estimates and a Liouville Theorem\\
  for   the Complex Monge--Amp\`ere
    Equation\\
     under a Quadratic Lower
 Bound}
     \author{Jingchen Hu \and Li Sheng}
      \date{November 19, 2026}
    
\begin{document}
   \maketitle
     \begin{abstract}
   Let $\mathpart{n\ge2}$ and       let   $\mathpart{u}$ be
      a  locally   bounded plurisubharmonic Bedford--Taylor solution of
 the    complex  Monge--Amp\`ere    equation
    $\mathpart{\det(u_{i\bar    j})=f}$.
     Assume    that in a  sufficiently
    large ball centered    at  the   origin
   \[
     u(z)\ge
  \delta\sum_{i=1}^{n}|x_i|^2,\qquad
       x_i=\operatorname{Re}z_{i},\quad
     \delta>0,
   \]
   and    $\mathpart{u(0)\leq  H}$.    If $\mathpart{f}$
  is Dini-continuous,   bounded above and bounded away from    zero,
      and   satisfies   $\Delta_{\mathbb  {R}}(f^{1/n})\ge-K$  in
     distributions,
    we  derive an interior bound for the real $C^{2}$ norm on the unit    ball $B_{1}(0)$. When
      $\mathpart{f\equiv1}$,    every    entire solution satisfying
       $u(z)\ge\delta\sum_{i=1}^{n}|\operatorname{Re}z_i|^2$
   for all    $z\in\C^{n}$
  and    some   $\mathpart{\delta>0}$ is a real quadratic
     polynomial.
      \end{abstract}
      \section{Introduction}
  \label{sec:introduction}
Interior derivative estimates are basic tools in partial differential equations
and play an essential role in proofs of Liouville theorems.
    The   harmonic
  equation provides       a
   simple  model.  For a positive harmonic function
   $\mathpart{h}$ on    $B_{2R}\subset\R^{d}$,
    the  logarithmic    gradient     estimate gives
      \[
     \sup_{B_R}|D\log
    h|\le\frac{C_{d}}{R};
   \]
   see Cheng and Yau \cite{ChengYau}    for      its Riemannian
   generalization.
    For a       nonnegative
   entire harmonic function  $\mathpart{h}$, apply    the  estimate to
     $\mathpart{h+\varepsilon}$ with $\mathpart{\varepsilon>0}$.   Letting
   $\mathpart{R\to\infty}$ at    any  fixed
    point   gives    $\mathpart{D\log(h+\varepsilon)=0}$,  so $\mathpart{h}$
   is     constant.
    See  also
     Yau    \cite{YauHarmonic}.

   For nonlinear  equations, the
     essential issue    in proving a
   Liouville  theorem
  is to obtain interior    second-derivative estimates
   that remain controlled    after
     rescaling.

   Then, for  concave uniformly  elliptic
     equations,  the       interior
    $\mathpart{C^{2,\alpha}}$
     theory    of  Evans \cite{Evans} and Krylov \cite{Krylov}
    provides a general route      from estimates
   to  rigidity.
%
\iffalse
     If     $v_{R}(x)=R^{-2}v(Rx)$   has  a $C^{2,{\alpha}}(B_{1})$
%
     bound     independent
%
   of $\mathpart{{R}}$, the    change
%
   of  variables gives
%
   %
      \[
%
 %
   |D^2v(x)-D^{2}v(y)|\le
   CR^{-\alpha}|x-y|^\alpha,
%
     \qquad
%
    x,y\in   B_R.
%
%
  %
%
  \]
%
  Letting $\mathpart{R\to\infty}$ for   fixed $\mathpart{x,y}$ makes the
%
   Hessian    constant.
\fi
   %
%
   %

   %
      \subsection{Real Monge--Amp\`ere and Hessian equations}
    \label{sec:real-history}
   %
 %  -------------------------------------------------------------------------
%
   %

  %
%
   For the real Monge--Amp\`ere
      equation $\det   D^{2}v=f>0$,    Pogorelov's
    interior  estimate  \cite{PogorelovEstimate} controls the
   Hessian with     a
  weight    measuring   separation
    from affine    boundary data. In the
    constant-density
   case, a standard  form    for a smooth strictly convex
    solution with $\mathpart{v<0}$ in    a  bounded smooth   convex    domain and $\mathpart{v=0}$
  on        its boundary    is
     \[
      (-v)\|D^{2}v\|\le
     C,
     \qquad \det
   D^2v=1,
    \]
     where $\mathpart{{C}}$  depends on
   the    dimension and bounds for  $\mathpart{v}$    and $\mathpart{Dv}$; see
    Trudinger   and   Wang
      \cite[Lemma~3.6]{TrudingerWangSurvey}.
  The    section  geometry and the    affine  boundary
  normalization  are part
   of  this estimate.   Pogorelov's singular    examples in
  dimensions $\mathpart{{d}\ge3}$
    show why one    cannot
   replace them  by      an unrestricted estimate based
    only on   the amplitude of  a convex solution
   on    an arbitrary ball
   \cite[Section~3]{TrudingerWangSurvey}.

     The       associated
     entire-solution theorem    is    the
    J\"orgens--Calabi--Pogorelov   theorem:
   every smooth   convex solution  of
    $\det D^{2}v=1$ on
  $\R^{d}$    is a quadratic  polynomial, without  a growth
     assumption.     J\"orgens
     \cite{Jorgens}      proved the     two-dimensional  case,
    Calabi \cite{Calabi} treated   dimensions
     at most  five,    and
   Pogorelov \cite{PogorelovRigidity}   proved    the
     general   case.
   The role of normalized    sections and interior
      estimates in    this
      rigidity
   theorem   is reviewed    in
      \cite[Theorem~4.4]{TrudingerWangSurvey}.

      The $\mathpart{k}$-Hessian    equations
     interpolate    between   the Laplace and
    Monge--Amp\`ere
     equations:
     \[
      \sigma_k(D^{2}v)%
      =f,
   \qquad
   \sigma_k(\lambda)=\sum_{1\le
 i_1<\cdots<i_k\le  d}
       \lambda_{i_1}\cdots\lambda_{i_k},
    \qquad
   1\le    k\le d.
    \]
    The  elliptic  branch
   is    the admissible   cone
     $\Gamma_{k}=\{\lambda:\sigma_{j}(\lambda)>0,\ 1\le
  j\le    k\}$.
     A smooth     function is called $\mathpart{k}$-admissible,
   or    $\mathpart{k}$-convex,  when the
   eigenvalues  of    its real Hessian
     lie    in $\Gamma_{k}$.
 Caffarelli,     Nirenberg,
     and Spruck   \cite{CNSHessian} developed     the
 classical Dirichlet theory for
  these      and   more general concave
  functions    of the Hessian eigenvalues.
   Chou   and     Wang
      \cite[Theorem~4.1]{ChouWang}
     obtained      weighted    interior   Hessian
    estimates using       an admissible comparison function above
 the  solution.
   Sheng, Urbas, and   Wang
     \cite{ShengUrbasWang} extended    the   weighted
    interior-estimate     method to
   classes      of  curvature and curvature-quotient
   equations.     For
   $\mathpart{2\le  {k}\le d-1}$,    and $\mathpart{(k+1)}$-convex solutions  of
   $\sigma_{k}(D^2v)={f}(x,v,Dv)>0$
      with zero   boundary values, Li,
     Ren,  and  Wang
   \cite{LiRenWang}    proved
  the stronger    weight-one estimate
     $\mathpart{(-v)\Delta    v\le
   C}$.

      These estimates  lead to    Liouville theorems under
   positive      quadratic
    lower
    growth,
     \[
      v(x)\ge    a|x|^{2}-b\quad\text{for
   large    }|x|,
   \qquad
   a>0.
     \]
   Bao,  Chen, Guan, and    Ji
      \cite[Theorem~3.2]{BaoChenGuanJi}      proved
     quadraticity  for    smooth convex
      entire   solutions   of
 $\sigma_k(D^{2}v)=1$     with
    this growth. Li,  Ren,   and    Wang
    \cite{LiRenWang} obtained      the
 corresponding   result   under
  $\mathpart{(k+1)}$-convexity, for   $\mathpart{2\le
  k\le    d-1}$.
   For    the larger
  class of $\mathpart{k}$-admissible       solutions,
   Wang and     Bao \cite{WangBao}
 assumed two-sided positive   quadratic
   growth.   Chu and  Dinew \cite{ChuDinew} developed a
    common   framework
      for several  Hessian operators;  for    $\sigma_{k}$,
  their  result    applies
  under positive  quadratic
  lower growth  and
   $\sigma_{k+1}(D^{2}v)\ge-{A}$.     Zhang
    \cite{ZhangHessian}  established
    weighted  estimates and  rigidity    for $\mathpart{3\le k\le
     d-1}$ under   semiconvexity
   $D^2v\ge-K\Id_{d}$
   and positive   quadratic     lower growth.
   Tu \cite{TuHessian}  obtained    weighted
      estimates and rigidity from
     positive   quadratic   lower
  growth     for    the    $\mathpart{({d}-1)}$-Hessian equation
    without
   semiconvexity, for   $\mathpart{d\ge3}$.
      In a recent   preprint,
 B.~Wang \cite{BinWang} extends   this
  lower-growth     rigidity  to $\mathpart{{k}}$-admissible solutions
  with    $\mathpart{2k>d}$,
   for   $\mathpart{2\le
 k\le   d-1}$.

  The    quadratic Hessian  equation
 $\sigma_{2}(D^2v)=f$ has    a   more extensive
 interior    theory. Warren   and Yuan \cite{WarrenYuan}
 proved   interior
    Hessian   estimates
     for    $\mathpart{f=1}$ in $\R^{3}$,   without   convexity.
  Qiu \cite{QiuHessian} treated   positive
    smooth    variable   right-hand
    sides     in dimension
  three. Guan    and  Qiu \cite{GuanQiu} proved
     estimates    in all dimensions for  convex  solutions, and
   more   generally
   under   $\Gamma_{2}$-admissibility   and
      $\sigma_{3}(D^2v)\ge-A$.
  For       $\mathpart{f\equiv1}$,
     Shankar and Yuan   proved    estimates for semiconvex solutions
     \cite{ShankarYuanEstimate},   and
   later treated   dimension  four  without
     convexity      \cite{ShankarYuanFour};
  their       higher-dimensional estimates
    in   the    latter
  paper   retain    an      additional eigenvalue condition.

      For    the    corresponding
   Liouville problem,   Chang and    Yuan
   \cite{ChangYuan}   proved quadraticity     for entire
    solutions     of
    $\sigma_2(D^{2}{v})=1$
    satisfying
  \[
      D^{2}v\ge\left(\delta-\sqrt{\frac{2}{d(d-1)}}\right)\Id_d,
  \qquad
    \delta>0.
   \]
     In   particular, their theorem includes
  convex       solutions without
    growth assumptions.   Shankar    and
      Yuan \cite{ShankarYuan}      later
   proved the    same   conclusion  under arbitrary fixed
   semiconvexity
     $D^{2}v\ge-K\Id_{d}$,
   again without growth assumptions.
  Under      positive quadratic
    lower   growth, Chen  and Xiang
     \cite[Theorems~1.3--1.4]{ChenXiang}
  proved   rigidity   for
      $\Gamma_{2}$-admissible     entire
      solutions in    $\R^{3}$, and   in  all
      dimensions  with
  the additional    assumption $\sigma_{3}(D^2v)\ge-A$.

    Several recent   preprints    advance
      this theory. Li    and      Wu
     \cite{LiWu}
  obtain    interior   Hessian estimates for
   $\sigma_{2}(D^2v)=1$     in
    every       dimension without convexity.
      For positive H\"older continuous
  right-hand   sides,    Zhou and   Zhu
  \cite{ZhouZhuRegularity}
      obtain    interior $\mathpart{C^{2,{\alpha}}}$ regularity
   for      convex viscosity  solutions,  while
    Chen,  Zhou, and Zhu
    \cite{ChenZhouZhu}  treat the
   full ${\Gamma}_{2}$-admissible      branch.
  For higher   Hessian
     equations, Zhou    and   Zhu
    \cite[Theorem~1.1]{ZhouZhuRigidity}
  prove in    a recent preprint that
  every $C^{2}$ convex entire solution of
     $\sigma_{k}(D^2v)=c_0>0$     is
  quadratic for   $\mathpart{3\le   k\le d-1}$, without a
 growth    assumption.
 Convexity or another structural    hypothesis  remains
       significant:
   Warren      \cite{WarrenNonpolynomial}
  constructed nonpolynomial       entire
  admissible solutions when $\mathpart{k\ge2}$ and
    $\mathpart{d\ge2k-1}$.
      \subsection{Complex equations and earlier Liouville theorems}
       \label{sec:complex-history}
%
 %
      %   -------------------------------------------------------------------------
%
      %

     For  the  complex
 Monge--Amp\`ere equation
   $\mathpart{\det(u_{i\bar
   j})=f}$,  the   prescribed right-hand  side    $\mathpart{f}$ is   called
   the   \emph{density}.     Bedford
 and   Taylor    \cite{BTDirichlet,BT}
    developed    the weak
  theory for locally   bounded plurisubharmonic
     potentials.      Caffarelli, Kohn, Nirenberg,
   and    Spruck \cite{CKNS}
  established   smooth Dirichlet   solvability    on
  smooth strictly
      pseudoconvex    domains    with
 smooth positive density  and smooth
   boundary values. On bounded       real
      convex domains,    B\l   ocki
     \cite{BlockiConvex}    proved  interior
 smoothness for the zero-boundary
     Dirichlet   problem under suitable
  positive-density   assumptions.
   These        boundary
    and geometric  hypotheses are essential    to the scope
    of       the
  cited  estimates.

    For   general   plurisubharmonic solutions,
  smooth positive     density
    alone does not imply interior ${C}^{2}$    regularity;
 complex analogues
    of Pogorelov's         examples
     demonstrate  this   obstruction
  \cite{BD}.   He
    \cite{HeRegularity}    obtained   interior    Laplacian
  estimates    from $L^{p}$
  control  of   the Laplacian for    ${p}>n^{2}$  and suitable
  positive    data.      B\l ocki and Dinew \cite{BD}
     proved    local    regularity
   for     $\mathpart{W^{2,p}}$ solutions in complex
   dimension    $\mathpart{n\ge2}$, with
        $\mathpart{p>n(n-1)}$
   and   positive  $\mathpart{{C}^{1,1}}$
 density. Once     second
   derivatives are  controlled,  Dinew, Zhang,
  and       Zhang \cite{DinewZhangZhang} obtain
  $\mathpart{C^{2,\alpha}}$     regularity
   from      a $\mathpart{C^{1,1}}$ potential
  and positive H\"older    density.
  Wang   \cite{Wang} obtains full    real
   second-derivative      H\"older
   estimates   from    a   bound for the complex Laplacian and
   corresponding
  positive    H\"older
  data.

   Interior  estimates and   Liouville theorems also
    interact    in    the
    complex Hessian   equations.   Hou,   Ma,
   and    Wu \cite{HouMaWu}
     proved a second-order estimate controlled      by the square
     of the gradient    on compact
    K\"ahler manifolds.   Dinew     and
   Ko\l    odziej \cite{DinewKolodziejLocal}
    proved  local Laplacian estimates under
      $\mathpart{W^{2,{p}}}$  control   with
 $\mathpart{p>n(m-1)}$   and  suitable
     positive    data.
    Their      Liouville
  theorem \cite{DinewKolodziejLiouville}    states  that
 bounded  entire
     maximal $\mathpart{m}$-subharmonic functions    with bounded
    gradient  are    constant;  it  supplies    the
    blow-up argument    for  global
   gradient estimates. This is   a
      homogeneous  complex-Hessian result,
  as    distinct from
   quadratic  rigidity  for  an     equation with positive
  constant   right-hand  side.    Sz\'ekelyhidi
       \cite{SzekelyhidiHermitian}
  established estimates   for    a broad class of concave
    elliptic    equations
     on compact    Hermitian
  manifolds   using  a $\mathcal   {C}$-subsolution.

 For the constant-density complex
   Monge--Amp\`ere    equation,
  Riebesehl and Schulz \cite{RiebesehlSchulz}
    proved constancy    of the
   complex Hessian  under    uniform
 equivalence of the associated    K\"ahler
     metric with the Euclidean    metric.  Constancy of    the
    complex       Hessian
     does not   by itself make  the
   potential quadratic, since  adding    the
     real  part    of a holomorphic    function leaves $\mathpart{u_{i\bar
  j}}$    unchanged.
     Wang       \cite{WangLiouville}
  proved quadratic rigidity for    entire
       Bedford--Taylor
  solutions asymptotic to a prescribed positive
  quadratic      solution: for the
  standard   model,
   $u(z)=|z|^{2}+o(|z|^2)$    implies
 that $u-|z|^{2}$ is  affine.

  Li       and Sheng \cite[Theorem~1.2]{LiSheng} proved
     constancy  of    the
      complex Hessian     under completeness and   a
  positive quadratic  lower
    growth     bound. Their  Corollary~1.3
    gives   quadraticity of the potential
 under completeness
%   and two-sided    quadratic   growth. Bing-Long   Chen
  and two-sided    quadratic   growth. Chen
  \cite{BLChen}   obtained  a  Liouville theorem
 under   completeness    and
   a weaker lower growth condition involving
     powers     of $\mathpart{n}$ totally
    real   linear
    functions,   with   exponent $\mathpart{p>2n/(n+1)}$.
  In   a recent preprint, Chen,   Hu, and Sheng
  \cite{CHSPogorelov}
  obtained weighted
   interior second-derivative    estimates   under a
    matrix condition    relating $\mathpart{(u_{ij})}$ and
    $\mathpart{(u_{i\bar      j})}$.
     Their    estimates imply
    quadraticity for smooth    real convex entire
    solutions, without    growth
 or completeness assumptions;   convexity
   outside    a  compact set
   also   suffices.

   Related estimates  for    the homogeneous    complex
       Monge--Amp\`ere
   equation    concern
    metric  lower   bounds \cite{HuMetric}, preservation
 of   convexity   \cite{HuConvex}, and   the rank
    of   the complex      Hessian
    \cite{HuRank},   under   the
      geometric and   boundary   assumptions  of   those
    problems.  For
    determinant one in  complex dimension two,  Chen, Hu,
   and    Sheng     \cite{CHSPower}  obtain
  strict   convexity of
      $\mathpart{-\sqrt{-u}}$    on
  suitable convex domains with $\mathpart{u=0}$ on  the boundary.

      \subsection{The weak equation and the main results}
       \label{sec:main-results}
%
%
   % -------------------------------------------------------------------------
  
    We  prove  an  interior    estimate
  for  the   full  real Hessian of a
   plurisubharmonic   solution  with
     a quadratic lower   bound  in  the    real
     coordinates $x_{i}=\operatorname{Re}z_{i}$    on a sufficiently  large    ball.  The    density
      is       positive and Dini-continuous
  with     $\mathpart{\Delta_\R(f^{1/n})\ge-K}$ in
    distributions.       Theorem~\ref{thm:real-c2}
   gives  the precise    dependence on the value of  the
      solution   at the   center.
 Theorem~\ref{thm:main} proves that every   entire
   solution      of
   \begin{equation}\label{eq:entire_unit_density_monge_ampere}
  \det(u_{i\bar
  j})=1\qquad\text{on      }\C^{n}
   \end{equation}
   with    $u(z)\ge\delta\sum_{i=1}^{n}
     |\operatorname{Re}z_{i}|^2$ for    some
   $\mathpart{{\delta}>0}$      is
     a real quadratic   polynomial.
    This  gives, in   particular,
  the two-sided   growth conjecture of
 G\'abor    Sz\'ekelyhidi
      \cite[Problem~1.4]{AIM2018}.  The    theorem
  applies    to   locally bounded
   Bedford--Taylor     solutions  and    requires
  neither  convexity
  nor completeness. Chu and   Liu
     \cite[Theorem~1.2]{ChuLiu}
  proved quadraticity for smooth entire solutions
  with     $C^{-1}(1+|z|^2)\le
  u(z)\le C(1+|z|^{2})$. In our October preprint
     \cite[Theorem~1.1]{HuShengWeak},
    we       treated   locally bounded
   Bedford--Taylor solutions
     under   the    paired bounds
   \[
   \frac{c r}{\ell^\alpha}\le    u\le
    C    r\ell^{1-\alpha},\qquad
  r=e+|z|^{2},\quad \ell=1+\log
    r,\quad     0\le\alpha\le1.
      \]
    For $\mathpart{{\alpha}>0}$ the       earlier lower
   bound grows more    slowly  than $|z|^{2}$
     at  infinity.  Here  we
    assume a   quadratic   lower  bound in the real
    coordinates    and
   derive an  upper  bound in   Lemma~\ref{lem:lower-upper}.
    Wang's
     asymptotic  theorem
   \cite{WangLiouville} assumes    that
   a    quadratic model is  prescribed
  at
    infinity.

 For a    smooth  plurisubharmonic solution, the
   complex   Hessian is
  nonnegative  and
   its    determinant equals $\mathpart{f}$. A bound   for     its  trace
      therefore   bounds    each
  eigenvalue  from above; if $\mathpart{f}$ has     a positive
    lower   bound, their
  product also   bounds each eigenvalue from    below.
     These   bounds  make  the     linearized
     equation uniformly     elliptic.
   Our argument    obtains the initial   trace bound  by
      controlling        the
   difference    between    a    potential
 and its Green-function  average.
    The   contribution
   from the distant boundary tends   to    zero as the
     outer radius       tends   to
     infinity.

  Throughout the    paper the complex dimension is
 $\mathpart{n\ge2}$. Write   $z_j=x_{j}+\sqrt{-1}y_j$ and
   \[
    \partial_{z_j}=\tfrac12(\partial_{x_j}-\sqrt{-1}\partial_{y_j}),\qquad
    \partial_{\bar
       z_j}=\tfrac{1}2(\partial_{x_j}+\sqrt{-1}\partial_{y_j}).
       \]
   For a twice continuously         differentiable real function
     $\mathpart{\varphi}$,     its
     \emph{complex
 Hessian}    is
      \[
     D^2_{\C}\varphi=(\varphi_{i\bar
    j})_{i,j=1}^n,
  \qquad  \varphi_{i\bar
  j}=\partial_{z_{i}}\partial_{\bar    z_j}\varphi.
    \]
   We also    write $\varphi_i=\partial_{z_{i}}\varphi$,    $\varphi_{\bar
   j}=\partial_{\bar z_{j}}\varphi$
    and      $\varphi_{ij}=\partial_{z_i}\partial_{z_{j}}\varphi$.
      The real    Hessian
     $D^{2}_{\R}\varphi$
    uses   the coordinate order
       $(x_{1},\ldots,x_n,y_1,%
     \ldots,y_n)$.
  Here    $\mathpart{\tr}$ denotes    the   sum of  the
   diagonal      entries.  Since
       $\varphi_{j\bar
      j}=\tfrac{1}4(\varphi_{x_jx_j}+\varphi_{y_jy_j})$
   by the    derivative conventions    above, the real
   Laplacian   satisfies
      \[
       \Delta_\R\varphi=\tr
      D^2_\R\varphi
    =\sum_{j=1}^n(\varphi_{x_jx_{j}}+\varphi_{y_jy_j})
    =4\tr
     D^2_\C\varphi.
   \]
 Thus      the  trace
    of  the  complex Hessian, $\tr D^{2}_\C\varphi$,
     is one quarter   of the real    Laplacian,
   $\mathpart{\Delta_\R\varphi}$.

    For  a   Hermitian
   matrix    $\mathpart{{H}}$,  we   use
 $H[v]=\sum_{i,j}H_{i\bar
    j}v_{i}\bar  v_j$.
  The    notation $\mathpart{H\ge0}$    means  that this    quantity
  is nonnegative    for every
     $\mathpart{v}$; $\mathpart{{H}>0}$ means that it is     positive
      for    every   nonzero $\mathpart{v}$.
   Thus    $\mathpart{H\le  {H}'}$ means
    $\mathpart{H'-H\ge0}$. Real   symmetric    matrix  inequalities
      are    defined using
      $v^{T}Hv$, where $\mathpart{T}$      denotes   transpose.
    We    write $\Id_{d}$  for the $\mathpart{d\times d}$   identity  matrix
 and   $\mathpart{{H}^*}$ for
     conjugate
    transpose.

      We specify    the weak meaning
     of the   equation and    state the full
   $C^{2}$    estimate,
    the pseudo-estimates underlying it,   and the   rigidity
      theorem.
    The definitions below fix the
  normalization    of
    the Monge--Amp\`ere measure   and the averaging
     operator, so   that  the  constants
   in these statements    can be read
    without       further
        conventions.

      A   function    is plurisubharmonic
    if it is    upper  semicontinuous    and
    subharmonic  on every complex line. For
  $u\in    C^{2}$, this is equivalent
     to $D^{2}_{\C}u\ge0$; strict      plurisubharmonicity
 means   $D^{2}_{\C}u>0$.

    To state the   equation
  without assuming  second      derivatives, set
   \[
      \omega_{0}=\sqrt{-1}\partial\overline{\partial}|z|^2,\qquad
      \rho({z})=1+|z|^2.
        \]
 For
   smooth    $\mathpart{{u}}$,
    \[
     \sqrt{-1}\partial\overline{\partial}u=\sqrt{-1}\sum_{i,j}u_{i\bar
     j}\,dz_i\wedge   d\bar   z_j,
   \qquad (\sqrt{-1}\partial\overline{\partial}u)^n=\det
    D^{2}_{\C}u\,\omega_0^n.
     \]
      The  measure $\omega_0^{n}$ is   a fixed positive
     multiple   of    ordinary
     volume
     on  $\mathpart{\R^{2n}}$.

   For locally    bounded plurisubharmonic
      $\mathpart{u}$, Bedford--Taylor    theory
     defines  the   measure
   $(\sqrt{-1}\partial\overline{\partial}u)^{n}$    by    decreasing   smooth
    plurisubharmonic approximation on    smaller
  open sets.    The    approximating
     measures converge weakly
     to a  measure independent of the   approximation
     \cite{BT}.
    The   weak       equation
  $(\sqrt{-1}\partial\overline{\partial}u)^n=f\,\omega_0^{n}$ is     an
     equality   of measures:   $\mathpart{f}$   is  the
   density     of the
   Monge--Amp\`ere  measure relative to the  reference
       volume
    measure        $\omega_{0}^n$.
  We   identify $\C^{n}$  with  $\mathpart{\R^{2n}}$
      and  write
 $B_{t}=\{h\in\C^{n}:|h|<t\}$   for
   the Euclidean   ball   centered at zero.

  %    -------------------------------------------------------------------------
%
% Subtract the real    part
%
 % of the holomorphic quadratic  polynomial. Its
%
    % complex Hessian vanishes,  and    the lower bound
%
% gains the    missing
  %    imaginary-coordinate
%
 %     squares.
%
  \paragraph{Reduction to full  quadratic growth.}
     Write $z_i=x_i+\sqrt{-1}y_{i}$,   $\mathpart{\delta>0}$,
  and      $\mathpart{{m}=\delta/2}$.
 For   a potential
   $\mathpart{u}$ satisfying    $u(z)\ge\delta\sum_{i=1}^n|x_{i}|^2$,
   set
   \begin{equation}%
    \label{eq:pluriharmonic_quadratic_growth_normalization}
         q(z)=m\operatorname{Re}\sum_{i=1}^n
   z_{i}^2,\qquad v=u-q+m.
   \end{equation}
 Since   $q=m\sum_i(x_i^{2}-y_i^2)$,    we
   have    $u-q\ge   m|z|^{2}$.
 Adding  $\mathpart{m}$  gives
    $\mathpart{v\ge    m\rho}$.   Pluriharmonicity of   $\mathpart{{q}}$ gives
     $\mathpart{\sqrt{-1}\partial\overline{\partial}v=\sqrt{-1}\partial\overline{\partial}u}$
 as  currents, so   $\mathpart{u}$ and   $\mathpart{v}$    have      the same
    Bedford--Taylor       measure.
 Let $\mathpart{F>0}$ bound the
    density and assume   $\mathpart{{u}(0)\le    H}$. Apply
     Lemma~\ref{lem:lower-upper}    in
      Appendix~\ref{app:height} to      $\mathpart{u-q}$,
 whose    value at the   origin
      is  $\mathpart{u(0)}$. Set
     \begin{equation}\label{eq:quadratic_height_bound_constants}
    \begin{gathered}
%
%
%    m=\delta/2,\qquad   A=A(n,m,F),\qquad
   m=\delta/2,\qquad   b_{\mathrm h}=b_{\mathrm h}(n,m,F),\qquad
 L_{\mathrm h}=L_{\mathrm  {h}}(n,m,F,H),\\
      M=M_*(n,m,F,H)+m
     .
       \end{gathered}
  \end{equation}
    For $L\ge L_{\mathrm  {h}}$ and    $\mathpart{u}$    defined     on
%
   %
%
%   $\mathpart{B_{2AL}}$, the  lemma    gives
  $\mathpart{B_{2b_{\mathrm h}L}}$, the  lemma    gives
   $\mathpart{u-q\le     M_*\rho}$
  on $\mathpart{B_{2L}}$.    Since   $\mathpart{\rho\ge1}$,
   \begin{equation}\label{eq:normalized_two_sided_growth_on_ball}
    \mathpart{m\rho}\le
   v\le M\rho\qquad\text{on    }B_{2L}.
     \end{equation}
    Only  the   measurable density bound   $\mathpart{0\le f\le F}$    is used
     in   this    step.

  The threshold    radius in the finite-ball estimate must be
    allowed to       depend
   on    $\mathpart{{H}}$. Indeed,  fix    $\mathpart{R>0}$.
 The   potentials
  $mR^2+N|z_{1}|^2+N^{-1}|z_2|^2$
      on $B_{R}\subset\C^{2}$    have    determinant one
     and lie above    $m|z|^{2}$.  They   all   have the same
      center value   $H={m}R^{2}$,
   while their real Hessians at the origin have
      an      eigenvalue  $\mathpart{2{N}}$.

     We    first state      the   estimate
  for actual    second    derivatives. A
    \emph{Dini     modulus} is
  a   continuous nondecreasing    function
     $\mathpart{\omega:[0,1]\to[0,\infty)}$     with
 $\mathpart{\omega(0)=0}$   and
   $\int_{0}^1\omega(s)\,ds/{s}<\infty$.
   Its tail  will be denoted   by
%
% The integral accumulates the prescribed modulus of continuity over smaller radii.
      \begin{equation}%
        \label{eq:dini_modulus_tail}
   \mathcal
   {D}_\omega(r)=\omega(r)+\int_0^r\frac{\omega(s)}s\,ds.
      \end{equation}
    The    smallness of     this tail supplies
   the scale   at   which the
  perturbation
 argument    begins.

  \begin{theorem}[$C^{2}$
 interior    estimate]\label{thm:real-c2}
 Fix $\mathpart{n\ge2}$, $\mathpart{{\delta}>0}$,  $\mathpart{H\ge0}$,
 $0<c_{0}\le c_1$, $\mathpart{{K}\ge0}$, and a
    Dini modulus    $\mathpart{{\omega}}$.
 %
%
%    Take $A,L_{\mathrm    {h}}$ from
   Take $b_{\mathrm h},L_{\mathrm    {h}}$ from
     \eqref{eq:quadratic_height_bound_constants}
  with  $F=c_{1}$.
     There are   constants
    $L_0\ge\max\{16,L_{\mathrm      {h}}\}$ and $\mathpart{C<\infty}$,
      depending only  on    these
  data, with the following property.
    Let $L\ge   L_{0}$, and  let $\mathpart{u}$
   be locally bounded  and    plurisubharmonic
   on
%
%    $\mathpart{B_{2AL}}$,
   $\mathpart{B_{2b_{\mathrm h}L}}$,
      satisfying
     \[
     (\sqrt{-1}\partial\overline{%
    \partial}u)^n=f\,\omega_{0}^n,
    \qquad
%  u(z)\ge\delta\sum_{i=1}^n|x_i|^2\quad(z\in    B_{2AL}),
 u(z)\ge\delta\sum_{i=1}^n|x_i|^2\quad(z\in    B_{2b_{\mathrm h}L}),
  \qquad
  u(0)\le   H,
     \]
  with   the equation understood in     the
   Bedford--Taylor  sense.
%
%  Suppose    $\mathpart{f\in C(B_{2AL})}$
 Suppose    $\mathpart{f\in C(B_{2b_{\mathrm h}L})}$
   satisfies
  \[
    c_0\le   {f}\le
   c_{1},
     \qquad
     \Delta_\R(f^{1/n})\ge-K
     \quad\text{in   distributions
%  on     }B_{2AL},
 on     }B_{2b_{\mathrm h}L},
   \]
  and
       \[
   \mathpart{|f(x)-f(y)|}\le\omega(|x-y|)
    \quad\text{whenever
%
%
%   }x,y\in    B_{2AL},\ |x-y|\le1.
  }x,y\in    B_{2b_{\mathrm h}L},\ |x-y|\le1.
      \]
      Then
   \[
    u\in   C^2(B_{1}),\qquad
    \|u\|_{C^2(B_1)}\le       C.
     \]
    Here $x_{i}=\operatorname{Re}z_i$,    and   the    derivatives
    are the full real      derivatives.
    The constants    are   independent of   $\mathpart{L}$ and depend on
  the  prescribed
   modulus $\mathpart{\omega}$, including
   its small-scale      tails.
       \end{theorem}

       \begin{corollary}[Interior
   smoothness]\label{cor:interior-smoothness}
   Under the hypotheses
      of  Theorem~\ref{thm:real-c2}, if
  ${f}\in     C^\infty(B_{1})$,  then
      $u\in     C^\infty(B_{1})$.
   In particular,     this   conclusion holds  when
   $\mathpart{f\equiv1}$.
     \end{corollary}

     In particular,    the  estimate includes
  $\mathpart{f\equiv1}$, with    $\mathpart{{K}=0}$ and
   $\mathpart{\omega=0}$. For    positive
   Dini-continuous densities, a   sufficiently
  large finite  domain
      therefore gives     an   actual ${C}^{2}$ bound.
   The next   theorem supplies
  the      quantitative    averaging estimate used
    to   prove this result. It    remains valid when the
   density   vanishes.
   We     first introduce the
    averaging     operators.

 For a  locally  bounded Borel    measurable
   function    $\mathpart{\varphi}$,   define    its
    uniform    sphere
    and ball   averages,    centered       at $\mathpart{{z}}$, by
      \begin{align}
    \sphereavg_t\varphi(z)
  &=\frac1{|\partial
   B_t|}\int_{\partial    B_{t}}\varphi(z+h)\,d\sigma(h),
  \label{eq:sphere_average_definition}%
    \\
   \ballavg_s%
  \varphi(z)
    &=\frac1{|B_s|}\int_{B_s}\varphi(z+h)\,dh.
   \label{eq:ball_average_definition}
     \end{align}
  Here $\mathpart{d\sigma}$ is surface  measure,
  $\mathpart{dh}$ is  ordinary volume   measure,
     and    the denominators are  the
     corresponding      total  measures.

     The    operator used
     in   the    estimate  averages $\ballavg_{s}\varphi$   from
 \eqref{eq:ball_average_definition}    with
  respect    to $\mathpart{{s}}$:
  \begin{equation}\label{eq:green_average_from_ball_averages}
        \mathcal{G}_t\varphi(z)
      =\frac2{t^{2}}\int_0^t
   s\,\ballavg_{s}\varphi(z)\,ds.
  \end{equation}
   Its  weights   are nonnegative
  and have    total mass   one, since
       $2t^{-2}\int_{0}^t
    s\,ds=1$.
    We  call $\mathcal{G}_{t}$
    the \emph{Green-function    averaging    operator};
  its representation by the Green    function of
   a ball is   given in
    Section~\ref{sec:weighted-balls}.
   For a   solution $\mathpart{w}$, the    quantity
   we      will
   estimate    is
    \begin{equation}\label{eq:green_average_defect}
      Q_t=\mathcal{G}_{t}
  w-w.
     \end{equation}
   A  plurisubharmonic function is
    subharmonic   in the   $\mathpart{2n}$       real
    coordinates.   The
 submean  inequality gives $\ballavg_{s} w(z)\ge    w(z)$,
 so \eqref{eq:green_average_from_ball_averages}  gives
    $Q_{t}\ge0$. Since $\mathpart{q}$ in
      \eqref{eq:pluriharmonic_quadratic_growth_normalization}    is     harmonic,   its
    mean       on every   ball
     equals its center    value. Thus
     $\mathcal    G_t    v-v=\mathcal G_{t}u-u$.
    We  now    estimate
    the   averaging  error $Q_{t}=\mathcal    G_tw-w$
    for    a nonnegative continuous
    density $\mathpart{{f}}$.  The constants use
  an    upper    bound for    $\mathpart{f}$ and
  a lower   bound for    the real   Laplacian  of
      its   nonnegative $\mathpart{n}$th    root $\mathpart{f^{1/n}}$, but
     no       derivative bound for    the    solution.
  The constant-density
  equation    used in   Theorem~\ref{thm:main} is
   the     case
     $\mathpart{f=1}$.

  \begin{theorem}[Pseudo-$C^{2}$   interior
      estimate]
    \label{thm:finite-ball}
   Fix $\mathpart{{n}\ge2}$, $\mathpart{\delta>0}$,  $\mathpart{H\ge0}$,
  $\mathpart{f_{\max}>0}$     and $\mathpart{{K}\ge0}$.
   Use  \eqref{eq:quadratic_height_bound_constants}   with
 $\mathpart{{F}=f_{\max}}$ and    take
   $L\ge\max\{16,L_{\mathrm     {h}}\}$. Let $\mathpart{w}$
    be a real,    locally    bounded
 plurisubharmonic
%      function on $\mathpart{B_{2A{L}}}$. Suppose
     function on $\mathpart{B_{2b_{\mathrm h}{L}}}$. Suppose
     \[
    (\sqrt{-1}\partial\overline{\partial}w)^n=f\,\omega_0^n,
     \qquad
%
 %
%
%     w(z)\ge\delta\sum_{i=1}^n|x_i|^{2}\quad(z\in     B_{2AL}),
    w(z)\ge\delta\sum_{i=1}^n|x_i|^{2}\quad(z\in     B_{2b_{\mathrm h}L}),
    \qquad
     w(0)\le       H,
     \]
   where     $x_i=\operatorname{Re}z_{i}$
    and the equation  holds     in the
  Bedford--Taylor      sense.
 Assume       that
     $\mathpart{f\in
%
%    C(B_{2AL})}$   and
   C(B_{2b_{\mathrm h}L})}$   and
    \[
   \mathpart{0}\le f\le
      f_{\max},\qquad
   \Delta_\R(f^{1/n})\ge-K\quad\text{in distributions
 %
%
%    on   }B_{2AL}.
   on   }B_{2b_{\mathrm h}L}.
      \]
 With $\mathpart{M}$    from
   \eqref{eq:quadratic_height_bound_constants},    define
       \begin{equation}%
     \label{eq:pseudo_c2_estimate_constants}
   \begin{gathered}
        \mu=\frac{m^{n-1}}{f_{\max}(4n)^{n-1}},\qquad
       \Lambda=\frac{2(n+1)}{\mu},\qquad
       t_0=\min\left\{1,\sqrt{\frac{m\mu}{2}}\right\},\\
       \varepsilon=\frac{m}{16M+m},\qquad
    C=\frac{4M^{2}}{m}.
      \end{gathered}
    \end{equation}
    For $Q_{t}=\mathcal    G_t w-w$,
    the   pointwise estimate
%
% The final term tends to zero as the outer radius increases.
    \begin{equation}%
     \label{eq:finite_ball_green_defect_bound}
     0\le
     Q_{t}(z)
  \le\frac{\Lambda
    t^2}{2(n+1)}
  +\frac{MK}{2m(n+1)}t^2%
   \rho(z)
  +C
 L^{-\varepsilon}\rho(z)
       \end{equation}
 holds  for
  $\mathpart{|z|<\sqrt    {L}}$  and $0<t\le t_{0}$.
    \end{theorem}
 We     call Theorem~\ref{thm:finite-ball}
    a \emph{pseudo-${C}^{2}$  interior
  estimate} because, for
      smooth $\mathpart{{w}}$,  the expansion proved in
   \eqref{eq:green_average_laplacian_expansion}
       gives
    \[
  \frac{Q_{t}(z)}%
  {t^2}\longrightarrow
  \frac{1}{2(n+1)}\tr
   D^2_{\C}w(z)
      \qquad\text{as
    }t\downarrow0.
    \]
  The first    two  terms on
  the right    of    \eqref{eq:finite_ball_green_defect_bound}  have
    order ${t}^{2}$ at each fixed
  point. The remainder  $\mathpart{C    L^{-\varepsilon}\rho}$
    is independent of    $\mathpart{t}$,   so shrinking
  the  averaging radius at   a fixed
 $\mathpart{L}$ does   not   directly give a
    Hessian  bound. Under    the positivity
     and Dini      assumptions
    of       Theorem~\ref{thm:real-c2},    however,
     Section~\ref{sec:real-c2}
     turns a    sufficiently small    remainder
  into  a full    real ${C}^{2}$
    estimate. For  an    entire    solution
  with   the same constants     on
  expanding balls,  we  first let $\mathpart{L\to\infty}$
    with   $\mathpart{z}$   and   $\mathpart{t}$ fixed, and
 only then divide by   $t^{2}$    and let
   $\mathpart{t\downarrow0}$.
     The constant-density    case has $\mathpart{K=0}$  and therefore
  gives a     global upper bound for
    the    complex Hessian.
      The     constant $\mathpart{f_{\max}}$
    bounds the  size  of the density, and $\mathpart{K}$
  bounds   how negative   the
   Laplacian    of $\mathpart{f^{1/n}}$    can    be.
  Zeros   of $\mathpart{f}$ are
 permitted. The distributional hypothesis    means
     \[
%
%       \int_{B_{2AL}}
      \int_{B_{2b_{\mathrm h}L}}
    f^{1/n}\Delta_\R\varphi\,d\xi
%     \ge-K\int_{B_{2AL}}\varphi\,d\xi
    \ge-K\int_{B_{2b_{\mathrm h}L}}\varphi\,d\xi
 \quad\text{for
%   every }0\le\varphi\in     C_{c}^\infty(B_{2AL}),
  every }0\le\varphi\in     C_{c}^\infty(B_{2b_{\mathrm h}L}),
      \]
     where $\mathpart{d\xi}$ is    Euclidean   volume measure.
     Thus    no differentiability
    of $\mathpart{f^{1/n}}$  at
 its    zeros is assumed. If $\mathpart{f^{1/n}}$ is subharmonic,
     we may take   $\mathpart{K=0}$ even when $\mathpart{f}$ varies
  or vanishes.
    We first    prove
 the estimate    for smooth   solutions with $\mathpart{f>0}$.
     Section~\ref{sec:nonnegative-density}    then
   treats     nonnegative densities
    and continuous potentials    by positive-density Dirichlet
    approximation.
   Section~\ref{sec:approximation}
      removes      continuity of the   potential
  using    decreasing Dirichlet
     replacements.
    When   $\mathpart{f\equiv1}$, Theorem~\ref{thm:finite-ball} gives
  the   following estimate.

      \begin{theorem}[Pseudo-${C}^{2}$
   interior   estimate for constant density]
   \label{thm:constant-density}
  Fix $\mathpart{n\ge2}$,  $\mathpart{\delta>0}$  and
  $\mathpart{{H}\ge0}$.   Use \eqref{eq:quadratic_height_bound_constants}
    with   $\mathpart{F=1}$   and
    take ${L}\ge\max\{16,L_{\mathrm  {h}}\}$.
     Let     $\mathpart{w}$ be a      real,
   locally    bounded
%
%    plurisubharmonic   function  on    $\mathpart{B_{2AL}}$
   plurisubharmonic   function  on    $\mathpart{B_{2b_{\mathrm h}L}}$
      satisfying
    \[
 (\sqrt{-1}\partial%
       \overline{\partial}w)^n=\omega_{0}^n,
     \qquad     w(z)\ge\delta\sum_{i=1}^n|x_i|^2\quad(z\in
%
  %
%
%      B_{2AL}),
     B_{2b_{\mathrm h}L}),
   \qquad
     w(0)\le H,
   \]
 where $x_{i}=\operatorname{Re}z_i$ and
   the equation is understood in   the
  Bedford--Taylor  sense.    Take
    $\Lambda,t_{0},\varepsilon,C$  from \eqref{eq:pseudo_c2_estimate_constants}
    with  $\mathpart{f_{\max}=1}$    and
    $\mathpart{M}$ from  \eqref{eq:quadratic_height_bound_constants}.       Then
   \begin{equation}\label{eq:unit_density_green_defect_bound}
  0\le\mathcal  G_{t}
   w(z)-w(z)
     \le\frac{\Lambda
  t^{2}}{2(n+1)}+C    L^{-\varepsilon}\rho(z)
    \end{equation}
 for $\mathpart{|z|<\sqrt   L}$ and
 $0<t\le    t_{0}$.
   \end{theorem}
      For an entire   solution  satisfying the
    same lower    bound,
    we    may let $\mathpart{{L}\to\infty}$   in
      \eqref{eq:unit_density_green_defect_bound},
  since  the
 constants  are    independent    of $\mathpart{{L}}$.
    This removes   the
  remainder. Dividing by $t^{2}$  and then    letting
  $\mathpart{{t}\downarrow0}$, using
  \eqref{eq:green_average_laplacian_expansion}   in    distributions
   as  in
     Section~\ref{sec:weak-limit},    gives
      $\mathpart{0\le\Delta_\R
     u\le4\Lambda}$.
    As shown in
     Section~\ref{sec:regularity},
   local   regularity    then gives
   uniform  interior    $\mathpart{C^{2,{\alpha}}}$ estimates
      for the   rescaled     solutions. Letting the
    rescaling radius  tend to
   infinity  yields  the
   following    rigidity  theorem.

    \begin{theorem}[Rigidity  under
     quadratic   lower      growth]\label{thm:main}
  Let $\mathpart{n\ge2}$, and let   $\mathpart{u}$
   be a real,    locally bounded plurisubharmonic
      function
  on   $\C^{n}$    satisfying
    \begin{equation}%
        \label{eq:unit_density_bedford_taylor_measure}
   (\sqrt{-1}\partial\overline{\partial}u)^n=\omega_0^{n}
     \qquad\text{in
   the Bedford--Taylor   sense}.
      \end{equation}
  Suppose that, for
  some    $\mathpart{\delta>0}$,
     \begin{equation}\label{eq:real_coordinate_quadratic_lower_bound}
     u(z)\ge\delta\sum_{i=1}^{n}|x_i|^{2},\qquad
      x_i=\operatorname{Re}z_i,\quad
   z\in\C^n.
 \end{equation}
   Then  $\mathpart{{u}}$ is   a
   real   quadratic  polynomial.
    \end{theorem}
     Apply     \eqref{eq:normalized_two_sided_growth_on_ball} on
   expanding balls with $\mathpart{F=1}$
 and    $\mathpart{H=u(0)}$. Since
 $\mathpart{u=v+q-m}$  and $q\le m|z|^{2}$, we obtain
   \begin{equation}%
       \label{eq:global_growth_of_normalized_and_original_potentials}
       m\rho\le      v\le    M\rho,\qquad
    0\le      u\le(M+m)\rho
   \qquad\text{on
   }\C^{n}.
       \end{equation}
  More generally, a
  lower bound $u(z)\ge\delta\sum_i|x_i|^{2}-b$ outside   a
   compact    set suffices. Local boundedness lets   us
    increase $\mathpart{{b}}$   so    that
    the inequality holds   everywhere, and
     $\mathpart{u+b}$    satisfies
   \eqref{eq:real_coordinate_quadratic_lower_bound}.

     \begin{remark}%
    \label{rem:growth-assumptions}
 A lower growth assumption cannot   simply be omitted: boundedness below and   quadratic    upper
   growth do not     imply rigidity.
      For      example, on   $\C^{3}$ the
     function
  \begin{equation}\label{eq:coordinate_product_counterexample}
     w(z)=16^{1/3}
     \bigl(|z_1z_2|%
     +|z_1z_3|+|z_2z_{3}|\bigr)
      \end{equation}
      is a     continuous    nonquadratic
 plurisubharmonic solution    of
    $(\sqrt{-1}\partial\overline{\partial}w)^3=\omega_{0}^3$
   satisfying
    \[
    0\le
       w(z)\le16^{1/3}|z|^{2},
    \qquad
    w(t,0,0)=0.
       \]
   To   verify the equation, put $D=\{z_1z_2z_{3}=0\}$
  and choose local       square
     roots     on
   $\C^{3}\setminus     D$.  Write
    $F=(\sqrt{z_1z_{2}},\sqrt{z_1z_3},\sqrt{z_2z_3})$,  so
      $w=16^{1/3}|F|^{2}$.
   Factoring $F_{j}/2$
    from    row     $\mathpart{j}$ and $1/z_{k}$   from  column $\mathpart{{k}}$ of   $\mathpart{DF}$
     leaves the  matrix
 with rows $\mathpart{(1,1,0),(1,0,1),(0,1,1)}$. Its    determinant
   is $\mathpart{-2}$, and    $|{F}_1F_2F_3|=|z_1z_2z_{3}|$. Thus
   $\mathpart{|\det  DF|=1/4}$ and
  $\det   D^{2}_{\C}w=16|\det   DF|^2=1$ off $\mathpart{D}$.
  The   displayed   sum of  moduli
   defines     a   continuous
   plurisubharmonic function  across $\mathpart{D}$. Both  its
    Bedford--Taylor measure   and    Euclidean   volume give
  zero    mass to    the
  pluripolar set $\mathpart{D}$ \cite{BT},   so   the  equation
      holds on  all     of $\C^{3}$.

     For $\mathpart{n>3}$,   adding
   $\sum_{j=4}^n|z_{j}|^2$   gives  the  same
      counterexample   on
  $\C^{n}$.
   Moreover,
     \[
  {w}+\frac{16^{1/3}%
    }2|z|^{2}
    =\frac{16^{1/3}}2%
  \bigl(|z_1|+|z_2|+|z_3|\bigr)^2
      \]
     is convex, since  the
     sum of the moduli is  nonnegative  and    convex.
    Thus  semiconvexity also does   not suffice for
   rigidity     in this    weak
 class, even under quadratic
   upper        growth.

      \end{remark}
      \begin{remark}\label{rem:smooth-counterexample}
%
 %
\iffalse
    Appendix~\ref{sec:smooth-counterexamples} gives     a
%
  %
%
     counterexample
%
  with the     additional advantage
%
   %
%
  of being smooth    everywhere.
%
  %
%
      In particular,  it      constructs  a
%
    nonquadratic
%
     $v\in C^\infty(\C^{2})$ that   is strictly
  %
%
   plurisubharmonic and   satisfies
\fi
The example in Appendix~\ref{sec:smooth-counterexamples} is
$v=\Phi_{1/2}\circ F$, where $F(z_1,z_2)=(e^{z_1},z_2e^{-z_1})$
and $\Phi_m$ is LeBrun's potential \cite[p.~302]{LeBrunTaubNUT}
for the Euclidean Taub--NUT metric \cite{HawkingInstantons}.
Jiaxiang Wang and Wenlong Wang \cite[Section~7.1.3]{WangWangSingularII}
add a harmonic quadratic polynomial to Warren's harmonic dual. In logarithmic
coordinates, the normalized partial Legendre transform of the sum gives $\Phi_m$.
The pullback $v$ is smooth, strictly plurisubharmonic and nonquadratic on
$\C^2$, and satisfies
     \[
       \det
 D^{2}_{\C}v=1,
     \qquad
    0<v(z)\le8(1+|z|^2).
   \]
 Thus    the failure
    of rigidity    under quadratic    upper growth is   not
   caused by    the singularities
    of       \eqref{eq:coordinate_product_counterexample}.
  In   fact,  for   every
    $\mathpart{2\le  k\le n}$, this   $\mathpart{v}$    extends to $\C^{n}$
    by
     setting
   \begin{equation}\label{eq:smooth_counterexample_sigma_k_extension}
     w(z_1,\ldots,z_n)=v(z_1,z_2)%
   +\sum_{j=3}^{k}|z_{j}|^2,
 \end{equation}
   with    an   empty
     sum    when $\mathpart{{k}=2}$.
   By    \eqref{eq:smooth_counterexample_sigma_k_extension}
      and
     Proposition~\ref{prop:smooth-counterexample}, the     complex
   Hessian       of    $\mathpart{w}$
     has exactly  $\mathpart{k}$ positive eigenvalues with    product $\mathpart{1}$ and
      $\mathpart{n-k}$ zero      eigenvalues.
    Hence    $\sigma_k(D^{2}_{\C}w)=1$
    and
   $\sigma_{j}(D^2_{\C}w)>0$
  for $\mathpart{1\le    j\le k}$,  so $\mathpart{w}$   is
      $\Gamma_{k}$-admissible.
      Formula~\eqref{eq:smooth_counterexample_sigma_k_extension}
      and
       Proposition~\ref{prop:smooth-counterexample}
   also    give
       $0<w(z)\le8(1+|z|^{2})$.
 This contrasts with the   real   quadratic-growth
     Liouville    theorem
   of Li and
  Wu   \cite[Remark~1.2]{LiWu} for
    $\sigma_{2}(D^2_{\R}v)=1$.
   \end{remark}

\begin{remark}\label{rem:linear_lower_growth}
A linear lower bound does not force an entire solution to be quadratic.
On $\C^2$, consider $u(z,w)=|z|^2+|w-z^2|^2$.
% The holomorphic shear $F(z,w)=(z,w-z^2)$ has complex Jacobian determinant one,
The holomorphic transformation $F(z,w)=(z,w-z^2)$ has complex Jacobian determinant one,
so $u=|F|^2$ is strictly plurisubharmonic and $\det D^2_{\C}u=1$.
The triangle inequality gives
\begin{equation}\label{eq:shear_linear_lower_bound}
  |(z,w)| \le \sqrt{|z|^2+|z|^4}+|w-z^2|
  \le |z|^2+\tfrac12+|w-z^2|^2+\tfrac14 = u(z,w)+\tfrac34 .
\end{equation}
Thus $u+\tfrac34$ satisfies a global linear lower bound.
Along $w=z^2$ we have $u(z,z^2)=|z|^2\sim |(z,z^2)|$ as $|z|\to\infty$,
while $u(z,0)=|z|^2+|z|^4$ shows that $u$ is not quadratic.
For $n>2$, adding $|z'|^2$ in the remaining variables $z'\in\C^{n-2}$
gives the same conclusion.
\end{remark}
       \subsection{Organization of the paper}
   \label{sec:proof-overview}
%
       %
%
  %
%
    %   -------------------------------------------------------------------------
   %

   Section~\ref{sec:preliminaries} records       the
    averaging      identities
  and   their  relation    to
   Hessian bounds.
     Section~\ref{sec:smooth} proves       the     pseudo-$C^{2}$
   estimate    for         smooth
     solutions with
     smooth   positive     density.
     Sections~\ref{sec:nonnegative-density}       and
     \ref{sec:approximation}
    extend the        estimate    to nonnegative continuous  densities
   and   continuous
    potentials, and then to
 locally      bounded   Bedford--Taylor solutions.
   Section~\ref{sec:real-c2}  proves
 Theorem~\ref{thm:real-c2}    and
    Corollary~\ref{cor:interior-smoothness}.
       Section~\ref{sec:regularity}
  proves    the Liouville theorem,
       Theorem~%
   \ref{thm:main}.
  Appendix~\ref{app:height}
      proves Lemma~\ref{lem:lower-upper}, which    yields
    the upper   bound \eqref{eq:normalized_two_sided_growth_on_ball}
    for    locally      bounded
      Bedford--Taylor solutions  with
  bounded   nonnegative density.
      Appendix~\ref{sec:geometric-appendix}          gives
    an alternative  geometric
    proof   for  smooth    entire solutions.
     Appendix~\ref{sec:smooth-counterexamples}
   constructs    the smooth two-dimensional
       counterexample.
        \section{From Green-function averages to a Hessian bound}
        \label{sec:preliminaries}
%
    %      =========================================================================
%
 % Section 2 | From   averages
%
 %   to the complex Hessian
 % The small-radius   expansion is proved
%
 %  for a smooth test function. For a
   % weak solution the averaging kernel can be moved    onto
%
 %  that    test    function
%
 % before    taking
%
   % the       limit.
%
     % Equation index (printed number
%   and  source label)
%
%   2.1     eq:test_function_second_order_taylor_expansion
%
%      2.2    eq:green_average_laplacian_expansion
 %   =========================================================================

     We explain   how an estimate for
   the  averaging error can   supply the
      global Hessian    bound
   needed    for  rigidity.   The   calculation for smooth
   functions identifies     the
     second-order quantity being  controlled;
  the  subsequent   discussion    explains
   how to recover the  same    bound
     for   a
     weak        solution.
   The  intermediate
 target  is

     \[
       0\le
   \sqrt{-1}\partial\overline{\partial}u\le\Lambda\omega_{0}
   \quad\text{for a     constant  independent
   of    position}.
     \]
     For a smooth  function
  this is the      matrix    bound
     $0\le
     D^{2}_{\C}u\le\Lambda\Id_{n}$.
     For   the weak solution it   is a
   distributional bound, which we   will
    state explicitly    using     test functions. Once
     it is available, the
  determinant       equation
  bounds every   Hessian eigenvalue away from  zero.
     Local regularity      and rescaling then
   give a  constant    real Hessian.

     We    therefore   begin with
   $\mathcal {G}_{t}\varphi-\varphi$, whose    expansion
     for   smooth $\mathpart{\varphi}$
   in \eqref{eq:green_average_laplacian_expansion}    recovers
    $\tr  D^{2}_{\C}\varphi$
      and        which
     still    makes sense for
  a locally   bounded function. When
   $D^{2}_{\C}\varphi\ge0$,    every
     eigenvalue is   at most
   $\tr
      D^{2}_{\C}\varphi$.
       \subsection{The second-order limit of \texorpdfstring{$\mathcal G_{t}$}{Gt}}
%
  %
%
  %  -------------------------------------------------------------------------
%
     %    Integrate the real Taylor   expansion over balls first, then
%
    %   average in the
%      radius. Both
%
 %  integrations affect  the coefficient of the Laplacian.

    We compute    the second-order limit of
    $\mathcal    G_{t}\varphi-\varphi$
     to determine  precisely
   which Hessian   quantity the averaging  error
   measures. The  operators    $\sphereavg_{{t}}$,
     $\ballavg_{{s}}$, and
   $\mathcal  {G}_t$ are
     defined  in \eqref{eq:sphere_average_definition},
     \eqref{eq:ball_average_definition},     and
      \eqref{eq:green_average_from_ball_averages}.

 Here $\tr     D^{2}_{\C}\varphi=\sum_i\varphi_{i\bar
    i}$ is the   sum
 of the diagonal entries    of the
     complex   Hessian.    Taylor expansion
 in the    real
 coordinates  gives
%
% The linear term integrates to zero over a ball centered at the origin.
  \begin{equation}\label{eq:test_function_second_order_taylor_expansion}
     \varphi(z+h)=\varphi(z)+\nabla_\R\varphi(z)\cdot
    h
  +\frac12h^TD^2_\R%
   \varphi(z)h+o(|h|^{2}),
    \end{equation}
    where $\mathpart{\nabla_\R\varphi}$  is the vector of first
 real derivatives,
    and $\mathpart{h}$ is   viewed as    a
    vector    in $\mathpart{\R^{2n}}$.
  On $B_{s}$, symmetry makes        the
   average   of each       coordinate and of
   each product of     two different coordinates zero.   By
   radial integration
  and  symmetry, the average  of  each
    squared   coordinate   is
  $s^{2}/(2{n}+2)$. Averaging
   \eqref{eq:test_function_second_order_taylor_expansion},      with
      $\tr     D^{2}_\R\varphi=\Delta_\R\varphi$,
   therefore gives
    \[
     \ballavg_{s}\varphi-\varphi
       =\frac{s^2}{4(n+1)}\Delta_\R\varphi+o(s^2).
    \]
   Integrating    this      identity
    with   the  weights in
   \eqref{eq:green_average_from_ball_averages},
     and    using
      $\Delta_\R\varphi=4\tr
  D^{2}_\C\varphi$, yields
  %
%
% Rotational symmetry cancels mixed derivatives and leaves the real Laplacian.
     \begin{equation}\label{eq:green_average_laplacian_expansion}
     \mathcal{G}_{t}\varphi-\varphi
    =\frac{t^2}{8(n+1)}\Delta_%
    \R\varphi+o(t^2)
     =\frac{t^2}{2(n+1)}\tr
  D^2_{\C}\varphi+o(t^2).
     \end{equation}
     Consequently,
     an upper          bound
   $\mathcal{G}_t\varphi-\varphi\le\Lambda
   t^{2}/[2(n+1)]$
   gives    $\tr D^{2}_{\C}\varphi\le\Lambda$
   after multiplication
    by    $2(n+1)/t^{2}$   and    passage to the   limit.
    If    $\mathpart{\varphi}$ is
    plurisubharmonic,      every
   Hessian   eigenvalue  is nonnegative    and
  at    most    $\tr  D^{2}_{\C}\varphi\le\Lambda$.
    We    then        obtain
   $0\le
     D^2_{\C}\varphi\le\Lambda\Id_{n}$.
    \subsection{From the estimates to a Hessian bound}
%
%
  %    -------------------------------------------------------------------------

  We     explain how
  Theorems~\ref{thm:finite-ball}    and
     \ref{thm:constant-density} yield a     bound
   for  $D^{2}_{\C}u$.
   The    essential
   point is   the   order of limits: expanding   the domain
     removes    the  boundary    error
   before shrinking the    averaging   radius
       recovers
    second  derivatives.

   The    domain  $\mathpart{B_{2L}}$ leaves room for the translations
    in $\mathcal      {G}_tu$.
   The    remainder
      $\mathpart{CL^{-\varepsilon}\rho}$    in
 \eqref{eq:unit_density_green_defect_bound}   tends   to
    zero uniformly on each
  fixed     compact     set
   as   $\mathpart{{L}\to\infty}$. On   a sufficiently large    finite
  domain,    Section~\ref{sec:real-c2} instead
    uses the smallness   of    this
   remainder    to prove
    Theorem~\ref{thm:real-c2}.
  For a smooth entire solution
   satisfying  \eqref{eq:real_coordinate_quadratic_lower_bound},  fix
    $\mathpart{H=u(0)}$.
  Apply
    Theorem~\ref{thm:constant-density} with  $\mathpart{{w}=u}$.
      First let $\mathpart{L\to\infty}$ while  $\mathpart{z}$  and  $\mathpart{t}$ remain   fixed.
  Then    multiply
    by $2(n+1)/t^{2}$ and    let   $\mathpart{t\downarrow0}$.
    Formula~\eqref{eq:green_average_laplacian_expansion}    gives
   $\tr  D^{2}_{\C}u\le\Lambda$.
 Together    with
    $D^{2}_{\C}{u}\ge0$,  this  gives
      $0\le D^{2}_{\C}u\le\Lambda\Id_n$. Taking   the limits
  in   this order
      removes $\mathpart{C     L^{-\varepsilon}\rho}$ before
    it  is divided by $t^{2}$.

    For a locally
    bounded Bedford--Taylor     solution,
  Theorem~\ref{thm:constant-density}    already   gives
    the    pointwise    averaging
    estimate. Let $\mathpart{L\to\infty}$  at  fixed     $\mathpart{{t}}$
  and  integrate    the resulting
   inequality against  a   nonnegative
   test   function    $\varphi\in C_c^\infty(\C^{n})$.
   By    symmetry of
   the    averaging     kernel,
    $\int(\mathcal  {G}_tu-u)\varphi=\int  u(\mathcal
       G_t\varphi-\varphi)$.
 Dividing by    ${t}^{2}$   and using
     \eqref{eq:green_average_laplacian_expansion}     for $\mathpart{\varphi}$
    as $\mathpart{t\downarrow0}$ gives $\mathpart{0\le\Delta_\R    u\le4\Lambda}$ in
   distributions,
   as proved in
   Section~\ref{sec:weak-limit}.
    Poisson regularity     gives
    $\mathpart{u\in      W^{2,\psi}_{\mathrm{loc}}}$  for every
   finite  $\mathpart{\psi>1}$. As
   detailed in  Section~\ref{sec:regularity},
     B\l ocki--Dinew   regularity
 \cite[Introduction, (2)]{BD}    gives
      smoothness.
   Wang's  interior
    estimate \cite[Theorem 1.1]{Wang} then gives    the
     uniform   H\"older bound   for the  rescaled real
  Hessians       in
    \eqref{eq:uniform_real_hessian_holder_bound_on_blowdowns}; letting    the rescaling radius
   tend   to infinity  makes
   the    full real    Hessian   constant,
      as in    Section~\ref{sec:blowdown}.

    The      next section
  proves the estimate for    smooth  solutions with
     positive  smooth    density. The
   divergence        theorem gives
      \eqref{eq:sphere_defect_as_green_hessian_trace},
     which connects the averaging     error
      to    the    trace of the
 averaged complex    Hessian.   At a    maximum of the
   weighted quotient     in \eqref{eq:weighted_defect_quotient_definition},
     the        first-order condition
 controls its   gradient terms.
   Sections~\ref{sec:nonnegative-density}
     and  \ref{sec:approximation} then  extend the
      estimate   by      approximation.
    \section{The interior estimate for a positive density}
      \label{sec:smooth}

    Let $\mathpart{w}$ satisfy   the
   hypotheses  of Theorem~\ref{thm:finite-ball}  and write
    $\mathpart{u=w-q+m}$ for   its    normalized potential.
    By \eqref{eq:normalized_two_sided_growth_on_ball},
  $\mathpart{m\rho\le u\le    M\rho}$ on   $\mathpart{B_{2L}}$. The
    two   potentials  have
     the    same complex Hessian and  the same
    averaging  error.
   We prove    the    averaging      estimate for a smooth
  potential $\mathpart{u}$    and    a smooth
      positive    density  $\mathpart{{f}}$.
   In  the calculations below    the same  estimate holds
   with any  known
  local   majorant $\mathpart{M\rho}$.
  The proof    compares   a    corrected averaging error with
   a power      of a
      positive function  built from   $\mathpart{u}$. The first-order
  condition at     a
    maximum gives     an    estimate    with constants
    independent  of the  outer radius.

      Since $D^{2}_{\C}u\ge0$ and
    its determinant is       positive, $\mathpart{u}$  is
      strictly      plurisubharmonic. The root    $\mathpart{f^{1/n}}$
     is smooth, so its
  Laplacian   inequality
    holds  pointwise. Write $Q_{t}=\mathcal   {G}_t    u-u$.
     We first relate
   $Q_{t}$ to $\mathcal LQ_{t}$,  with    $\mathcal    {{L}}$ defined
  in \eqref{eq:linearized_operator_and_inverse_trace}, and    then  use
      this relation  at    a maximum
      of the   weighted
 quotient. Throughout the    averaging calculations,
      $z\in  B_{L}$  and
   $\mathpart{0<t\le1}$.   Thus every translated     point lies    in
       $\mathpart{B_{L+1}\subset
    B_{2L}}$,    where the  hypotheses   hold.

    \subsection{The averaging identities}
   \label{sec:weighted-balls}
%
   %
%
%
    %     -------------------------------------------------------------------------
%
%
     %

      We derive    the   two
  inequalities  that    connect averaging to  the
    determinant   equation:    a    relation
     between the    averaging error and
    the trace       of
    the averaged    Hessian,   and   a lower bound    for the
      linearized    operator applied to  that   error.
    These  will provide the
  analytic input   for    the comparison
   argument.

    We need an  operator that uses the  determinant
   equation at   the
 center   of  each average.
     With   $A=D^{2}_{\C}u$, define
  \begin{equation}\label{eq:linearized_operator_and_inverse_trace}
    \mathcal{L}\varphi=\tr(A^{-1}D^{2}_{\C}\varphi),\qquad
       S=\tr
   A^{-1}.
      \end{equation}
      Strict plurisubharmonicity     gives $\mathpart{A^{-1}>0}$,
   so    $\mathcal   {L}$    is
      elliptic.     Taking
   traces,  and   then applying    the
 arithmetic--geometric mean inequality
     to   $\mathpart{A^{-1}}$,  gives
     \begin{equation}\label{eq:linearized_potential_identities_and_inverse_trace_bound}
  \mathcal{L}u=n,\qquad
      \mathcal{L}|z|^{2}=S,\qquad
    S\ge
     n  f^{-1/n}.
  \end{equation}
   The  last   inequality uses $\mathpart{\det
    A^{-1}=f^{-1}}$.    No     uniform
       eigenvalue
   bound    is being assumed.

   We next      relate   $\tr
      D^{2}_{\C}(\mathcal  {G}_t u)$ to
 $\sphereavg_{t}    u-u$, where   $\sphereavg_{t}$
  and $\mathcal    G_{t}$ are
    defined in \eqref{eq:sphere_average_definition}
  and     \eqref{eq:green_average_from_ball_averages}.
 To     do this,
     temporarily  set
    $Z_t=\sphereavg_{t}
  u-u$.    We    first  differentiate $\sphereavg_{s}    u(z)$
   with respect to $\mathpart{{s}}$,
    keeping the center   $\mathpart{z}$ fixed. Both    the    sphere
 $\partial B_{{s}}$   and    its    area
      depend  on $\mathpart{s}$. We therefore use
    $\mathpart{h=s\theta}$ to write the    definition
    in        \eqref{eq:sphere_average_definition}
    as an integral over   the      fixed
     unit sphere:
      \[
      \sphereavg_{s}
    u(z)
      =\frac1{|\partial
    B_1|}\int_{\partial    B_{1}}u(z+s\theta)\,d\sigma(\theta).
   \]
   Indeed, surface  measure    contributes
   the  factor $\mathpart{s^{2n-1}}$,      which
      cancels     with
     the same   factor   in    $|\partial    B_{s}|$.
     Since   $\mathpart{u}$ is     smooth,
    differentiation under    this    fixed    integral gives
      \[
   \frac   {d}{ds}\sphereavg_s
     u(z)
       =\frac1{|\partial
 B_{1}|}\int_{\partial B_1}
   \nabla_\R
     u(z+s\theta)\cdot\theta\,d\sigma(\theta).
    \]
  On  $\partial   B_{s}$, the outward
     unit    normal  is $\mathpart{\nu(h)=h/s}$.
     Substituting $\mathpart{h=s\theta}$ back    into the last
     integral, and   then
    applying       the divergence
    theorem to    the vector   field
   $\mathpart{h\mapsto\nabla_\R
   u(z+h)}$,    gives
  \[
       \begin{aligned}
    \frac
   d{ds}\sphereavg_s   u(z)
      &=\frac1{|\partial
 B_{s}|}\int_{\partial  B_s}
   \nabla_\R
       u(z+h)\cdot\nu(h)\,d\sigma(h)\\
    &=\frac1{|\partial   B_s|}\int_{B_s}\Delta_\R
        u(z+h)\,dh.
   \end{aligned}
 \]
     Because   the    ball
    has real dimension $\mathpart{2{n}}$,     its volume     and surface
   area satisfy  $|\partial
     B_s|=(2n/s)|B_{s}|$.
   The      definition
  of    $\ballavg_{{s}}$ in \eqref{eq:ball_average_definition}
    therefore      turns
   the last       expression       into
       $[s/(2n)]\ballavg_{s}(\Delta_\R
       u)(z)$.
  Our      derivative
      convention   gives
     \[
   u_{j\bar
   j}=\frac14(u_{x_jx_j}+u_{y_jy_j}),
       \qquad
  \Delta_\R  u=4\sum_{j=1}^n
     u_{j\bar   j}
     =4\tr
   D^{2}_{\C}u.
      \]
      The    divergence theorem    and
     \eqref{eq:ball_average_definition},      together
 with      $\Delta_\R
  u=4\tr    D^{2}_\C    u$, yield
    \[
  \frac
     d{ds}\sphereavg_s  u
       =\frac{s}{2n}\ballavg_s(\Delta_\R
      u)
  =\frac{2s}{n}\ballavg_s(\tr
      D^{2}_{\C}u).
       \]
    As   $\mathpart{s\downarrow0}$, continuity gives $\sphereavg_{s}
  u(z)\to u(z)$.
 Integrating     the
 preceding  identity from   zero to $\mathpart{t}$    and  using
    $\mathcal   G_t=(2/t^2)\int_0^t
   s\,\ballavg_{s}\,ds$  from
    \eqref{eq:green_average_from_ball_averages},     we
       obtain
   \begin{equation}%
     \label{eq:sphere_defect_as_green_hessian_trace}
     Z_t=\frac{t^2}{n}\tr
  D^2_{\C}(\mathcal{G}_{t}    u).
     \end{equation}
      Here     $D^2_{\C}(\mathcal
    {G}_{t} u)=\mathcal   G_t(D^2_{\C}u)$,
 because $\mathpart{u}$  is
  smooth on the larger  ball    and differentiation
    commutes with the  integrals in
   \eqref{eq:green_average_from_ball_averages}.

    Since $\mathpart{\Delta_\R {u}\ge0}$, the
 divergence theorem   gives
   $\frac
    d{ds}\sphereavg_{s}u=\frac{s}{2n}\ballavg_s(\Delta_\R      u)\ge0$.
      Thus $\sphereavg_{s}u$    is
      nondecreasing   in $\mathpart{{s}}$.
     The  submean      inequality gives
   $u\le\ballavg_{s}u$.      Since
  $\ballavg_{s}u$   in \eqref{eq:ball_average_definition}
   is     a positive weighted
    average of $\sphereavg_{r}u$ for  $\mathpart{0<r\le    s}$, this
     monotonicity     gives
      $u\le\ballavg_{s}u\le\sphereavg_tu$
  for   $\mathpart{0<s\le   t}$.
  Averaging   these
 inequalities with the      weights in
     \eqref{eq:green_average_from_ball_averages}      gives the
    first assertion    below.
  The second follows    by    averaging       the
 positive    matrices
  $D^{2}_{\C}u(z%
     +h)$:
      \begin{equation}\label{eq:green_defect_bounded_by_sphere_defect}
     0\le   Q_{t}\le
 Z_t,\qquad D^2_{\C}(\mathcal{G}_t u)\ge0.
    \end{equation}

   These weights also explain   the
     name of the  averaging operator.
   If $\Gamma_{t}$    is  the Dirichlet     Green
   function  of    $\mathpart{-\Delta_{\R}}$
   on $B_{t}$ with  pole at
  zero,   interchanging the   integrals   in
  \eqref{eq:green_average_from_ball_averages}
 gives
     \begin{equation}%
  \label{eq:dirichlet_green_kernel_representation}
  \mathcal{G}_{t}%
  \varphi(z)
   =\frac{4n}{t^2}\int_{B_t}\Gamma_{t}(h)\varphi(z+h)\,dh.
  \end{equation}
  Its    kernel
 is    nonnegative   and has total mass one.

 To use a maximum  principle,
   we also      need a lower bound for
   $\mathcal {L}Q_{t}$. We
   compare the Hessians at $\mathpart{z+h}$ and $\mathpart{z}$.
    Using the  equation  $\mathpart{\det    A=f}$ and the identity
   $\mathcal      {L}u=n$ in
     \eqref{eq:linearized_potential_identities_and_inverse_trace_bound},
     the  arithmetic--geometric  mean inequality
     for    $\mathpart{A(z)^{-1/2}A(z+h){A}(z)^{-1/2}}$
   yields
   \begin{equation}\label{eq:translated_potential_density_inequality}
      \mathcal
     {L}(u(\cdot+h)-u)(z)
      \ge\frac{n}%
   {f(z)^{1/n}}
   \bigl(f(z+h)^{1/n}-f(z)^{1/n}\bigr).
        \end{equation}
  The coefficients     of
  $\mathcal {L}$ in this   formula are   evaluated
      at    $\mathpart{z}$ and    stay  fixed as $\mathpart{h}$
   varies.

   The   hypothesis
  $\mathpart{\Delta_\R(f^{1/n})\ge-{K}}$ makes
  $h\mapsto f(z+h)^{1/n}+K|h|^{2}/(4n)$
        subharmonic.
     Its     ball submean
  inequality, together with
    $\ballavg_s|h|^2=ns^{2}/(n+1)$,
   gives
     \[
     \ballavg_{s}(f^{1/n}%
      )(z)-{f}(z)^{1/n}
    \ge-\frac{Ks^2}{4(n%
 +1)}.
  \]
   Averaging    this
    inequality   with   the     weights   $2s/t^{2}$  replaces
    $s^{2}$    by $t^{2}/2$. Applying     the   same
   averages      to
 \eqref{eq:translated_potential_density_inequality}
  therefore   gives
     \[
    \mathcal
     LQ_{t}\ge
    -\frac{n}{f(z)^{1/n}}\frac{Kt^2}{8(n+1)}.
  \]
      Since    $\mathpart{K\ge0}$,
    the bound    $\mathpart{{n}f^{-1/n}\le    S}$ in
     \eqref{eq:linearized_potential_identities_and_inverse_trace_bound}
        implies
   \begin{equation}%
  \label{eq:linearized_green_defect_source_bound}
   \mathcal
  {L}Q_t\ge-\frac{SKt^2}{8(n+1)}.
    \end{equation}

 \subsection{Prepare the comparison function}
     \label{sec:boundary}
%
%   -------------------------------------------------------------------------
%
% The barrier V    stays positive  after subtracting half of the
%
 % lower   quadratic
%
 % term. The correction in E_t absorbs the    contribution of
%
    %    the lower bound on
% the  Laplacian
%
%  of  the density root.

  We construct the functions used in
  the  maximum principle.  They
   must  control    the  inverse Hessian when
  its   trace is    large, account
  for the   density    error,    and
   suppress the distant boundary.   A weighted
    quotient will   then allow    the first-order
  maximum condition    to
   control
   its   gradient terms.

      We   begin with a  positive  function whose
    image under   $\mathcal    {L}$
  becomes negative    when
    $\mathpart{S}$      grows.

    By  \eqref{eq:linearized_potential_identities_and_inverse_trace_bound}, subtracting
  a multiple    of    $|{z}|^{2}$
   from $\mathpart{u}$ has   exactly this effect.    We use half
   the lower    growth
      coefficient so     that   positivity    is
     preserved, and   set
   \[
      V=u-\frac
    m2|z|^{2}.
       \]
   The bounds  \eqref{eq:normalized_two_sided_growth_on_ball} for
   the   normalized potential give
 $V\ge m+(m/2)|z|^{2}\ge(m/2)\rho$
     and   $\mathpart{V\le    M\rho}$.
   Applying   \eqref{eq:linearized_potential_identities_and_inverse_trace_bound}  to
   the  definition of $\mathpart{V}$
     then
     gives
%
% The inverse trace enters with a negative coefficient when the operator acts on V.
     \begin{equation}\label{eq:barrier_growth_and_linearized_operator}
  \frac {m}2\rho\le  V\le
  M\rho,\qquad
 \mathcal  LV=n-\frac
  {m}2S.
  \end{equation}

   We will  use    this negative
 contribution when    $\mathpart{S}$ is    large and
   the       determinant
  equation when $\mathpart{S}$ is bounded.    To specify the
    threshold   and    the
   corresponding    eigenvalue bound,  put
    \begin{equation}\label{eq:inverse_trace_threshold_and_eigenvalue_floor}
     T=\frac{4n}{m},\qquad
     \mu=\frac{1}{f_{\max}T^{{n}-1}}.
    \end{equation}
  Where $\mathpart{S\le   T}$, each eigenvalue of $\mathpart{A^{-1}}$
    is  at    most $\mathpart{T}$,
  while    their product     is
    $\mathpart{f^{-1}\ge   f_{\max}^{-1}}$.
   Dividing   this product
   by   the other    $\mathpart{{n}-1}$ eigenvalues   shows
   that    each
    eigenvalue  is   at least   $\mathpart{\mu}$.
  Thus     $A^{-1}\ge \mu\Id_{{n}}$   wherever
  $\mathpart{S\le  T}$.
    Since   $D^{2}_{\C}(\mathcal    G_t
    u)\ge0$   by
  \eqref{eq:green_defect_bounded_by_sphere_defect},    this
     matrix     inequality gives
     \[
 \tr    D^2_{\C}(\mathcal
   G_t     u)
    \le    \mu^{-1}\mathcal    L(\mathcal
  G_{t}  u).
     \]
   Combining   this    with
   \eqref{eq:sphere_defect_as_green_hessian_trace}        and
      $Q_t\le
    Z_{t}$ in \eqref{eq:green_defect_bounded_by_sphere_defect} bounds    $Q_{t}$
    by    $t^2\mathcal   {L}(\mathcal
   G_t      u)/(n\mu)$.
   Finally,       $\mathcal
    G_{t}    u=u+Q_t$ and
    $\mathcal   {L}u=n$ from
     \eqref{eq:linearized_potential_identities_and_inverse_trace_bound}       yield
%
% The inverse Hessian has an eigenvalue lower bound throughout this region.
    \begin{equation}\label{eq:green_defect_bound_below_trace_threshold}
    Q_t\le\frac{t^2}{\mu}+\frac{t^2}{n\mu}\mathcal
      LQ_{t}
      \qquad\text{where
    }S\le  T.
      \end{equation}

   The      constant term $t^{2}/\mu$ in
     this estimate    will   need to be
   no larger    than
   $\mathpart{{V}}$.    The   lower  bound $\mathpart{V\ge m/2}$ in
  \eqref{eq:barrier_growth_and_linearized_operator}  allows
 us to ensure this by  choosing
       \begin{equation}\label{eq:admissible_averaging_radius}
  t_{0}=\min\left\{1,\sqrt{\frac{m\mu}{2}}\right\}.
     \end{equation}
    Fix $0<t\le t_{0}$ from      now  on.
   This choice  gives
       \begin{equation}\label{eq:radius_term_bounded_by_barrier}
     \frac{t^2}{\mu}\le\frac
     {m}2\le  {V}.
       \end{equation}

   We now  subtract $t^{2}/{\mu}$   from
    $Q_{t}$ to   cancel  $t^{2}/\mu$ in
    \eqref{eq:green_defect_bound_below_trace_threshold}. We     also subtract
    a multiple   of    $\mathpart{V}$:
      the     term   $\mathpart{-(m/2)S}$
  in   $\mathcal {L}V={n}-(m/2)S$ from
   \eqref{eq:barrier_growth_and_linearized_operator} will   compensate
    for    $-SKt^{2}/[8(n+1)]$ in
   \eqref{eq:linearized_green_defect_source_bound}.
 With these two
   purposes,   define
        \begin{equation}\label{eq:corrected_green_defect_definition}
     \kappa=\frac{Kt^{2}}{2m(n+1)},\qquad
      E_t=Q_t-\frac{t^2}{\mu}-\kappa
    V.
     \end{equation}
  By    the   definition   of  $\mathpart{\kappa}$
   in \eqref{eq:corrected_green_defect_definition}, the    source bound
     \eqref{eq:linearized_green_defect_source_bound}
       becomes
     \begin{equation}\label{eq:linearized_green_defect_kappa_bound}
     \mathcal
    LQ_{t}\ge-\frac{m\kappa}{4}S.
     \end{equation}
  Using $\mathcal {L}V=n-(m/2)S$   from
  \eqref{eq:barrier_growth_and_linearized_operator}, we obtain
  \[
     \begin{aligned}
    \mathcal    L
     E_{t}
    &=\mathcal
    LQ_t-\kappa\mathcal      LV\\
      &\ge-\frac{m\kappa}{4}S
      -\kappa\left(n-\frac
   m2S\right)\\
  &=\frac{m\kappa}{4}%
      S-n\kappa.
     \end{aligned}
    \]
    The coefficient  of $\mathpart{{S}}$ in this    lower   bound
      is   nonnegative    because
     $\mathpart{\kappa\ge0}$.
   Before introducing the quotient, we  need a rough
   bound for  $E_{t}$.
   Fix  ${z}\in    B_{L}$. If $\mathpart{0<s\le
     t\le1}$ and ${h}\in    B_{s}$,  then
    $\mathpart{z+h\in B_{L+1}\subset  B_{2L}}$. The
  derived    upper  bound  on
    $\mathpart{B_{2L}}$  therefore   gives
     $\mathpart{u({z}+h)\le   M\rho(z+h)}$. Since every
   weight   in \eqref{eq:green_average_from_ball_averages}
   is  nonnegative,
   \[
    \begin{aligned}
    \mathcal
      G_t   u(z)
    &=\frac{2}{t^2}\int_{0}^t\frac{s}{|B_s|}
     \int_{B_s}u(z+h)\,dh\,ds\\
    &\le\frac{2M}{t^2}\int_0^t\frac{s}{|B_s|}
       \int_{B_s}\rho(z+h)%
    \,dh\,ds
       =M\mathcal
   G_t\rho(z).
    \end{aligned}
     \]
    Moreover,    $\mathpart{u(z)\ge
   m\rho(z)>0}$,   and hence
     \[
   Q_t(z)=\mathcal      G_t
    u(z)-u(z)
  \le\mathcal   G_{t}
   u(z)\le M\mathcal       G_t\rho(z).
  \]
      To    compute  $\mathcal
   G_{t}\rho(z)$,     write
    $\rho(z+h)=\rho(z)+2\operatorname{Re}\langle
      z,h\rangle+|{h}|^{2}$.
      The    linear term       has zero  average by
    symmetry, while      in
 $\mathpart{\R^{2n}}$ the  mean
      of $|h|^{2}$ over     $B_{s}$   is
     $ns^{2}/(n+1)$.    Averaging   with
   \eqref{eq:green_average_from_ball_averages}   gives
     \[
      \ballavg_s\rho(z)=\rho(z)+\frac{n}{n+1}s^2,
      \qquad
   \mathcal
    G_{t}\rho(z)
    =\rho(z)+\frac{nt^2}{2(n+1)}
    \le2\rho(z),
   \]
     where       the last
  inequality    uses    $\mathpart{{t}\le1}$    and    $\mathpart{\rho\ge1}$.
   Finally,    \eqref{eq:corrected_green_defect_definition},
   $\mathpart{\kappa\ge0}$, and  $\mathpart{V>0}$
     give $E_t\le    Q_{t}$. Combining
   this    with
 $Q_t\le   M\mathcal  {G}_t\rho$
    and $\mathcal      G_{t}\rho\le2\rho$, we     obtain
    \begin{equation}\label{eq:coarse_corrected_green_defect_bound}
   E_{t}\le  Q_{t}\le2M\rho\qquad\text{on
  }B_L.
     \end{equation}

    We now  choose a quotient whose
  denominator grows    slightly faster
    than    $\mathpart{{\rho}}$. This
      extra power will   suppress the outer boundary.
   The     exponent   of $\mathpart{V}$    is chosen below one so that,
  at a maximum of
     the      quotient, weighted Cauchy--Schwarz   controls
      the     gradient terms.
   Set
  \begin{equation}\label{eq:weighted_defect_quotient_definition}
  \varepsilon=\frac{m}{16M+m},\qquad
     a=1-\varepsilon,\qquad
      J=\frac{E_{t}}{V^a\rho^{2\varepsilon}}.
     \end{equation}
      These
      constants    satisfy
       \begin{equation}%
   \label{eq:quotient_exponent_balance}
    \mathpart{0<a<1,\qquad
    a+\varepsilon}=1,\qquad
    4\varepsilon=%
   \frac{am}{4M}.
    \end{equation}
   The    denominator    of $\mathpart{J}$ is    positive   by
      \eqref{eq:barrier_growth_and_linearized_operator}.
    Thus $\mathpart{J}$    is smooth even where    $E_{t}\le0$; we
    will    take   its logarithm
    only      near
    a hypothetical positive maximum.
        \subsection{Excluding the interior maximum}
%
      %
%
   %    -------------------------------------------------------------------------
%
 % The logarithmic    maximum calculation
%
 %   is  made only at a positive value of
   % the quotient. At   such    a
%
%   point the numerator E_t is positive as well.

     We show   that  $\mathpart{J}$
     has no     positive interior maximum. The differential
    inequality will   first  force  $\mathpart{S\le
      T}$,     after  which
      \eqref{eq:green_defect_bound_below_trace_threshold}       will   contradict  the
     maximum         inequality.
   The remaining boundary bound will  give the
  power decay.

   Suppose  that $\mathpart{J}$ attains a positive      maximum
     on $\overline{B_{{L}}}$
     at a point $x_0\in B_{L}$.
     By  \eqref{eq:weighted_defect_quotient_definition}    and
  \eqref{eq:barrier_growth_and_linearized_operator},
   $E_t(x_{0})>0$. The first- and second-order
   conditions
   give
     \begin{equation}\label{eq:log_quotient_conditions_at_maximum}
  \partial\log
      J(x_0)=0,\qquad
   \mathcal    L\log
   J(x_{0})\le0.
    \end{equation}
    Until the       contradiction
     below, all quantities are   evaluated       at $x_{0}$.
      For   a  real
   function $\mathpart{b}$, write
      \[
      |\partial
        b|_{A^{-1}}^{2}
   =\tr\bigl(A^{-1}(b_i     b_{\bar
     j})\bigr).
       \]
  The first   identity
   in \eqref{eq:log_quotient_conditions_at_maximum}    and the    definition
   of $\mathpart{J}$ in \eqref{eq:weighted_defect_quotient_definition}
   give
      \begin{equation}%
   \label{eq:log_defect_gradient_at_maximum}
     \partial\log
    E_{t}
       =a\,\partial\log
   V+\varepsilon(2\partial\log\rho).
    \end{equation}
   Since    $\mathpart{a+\varepsilon=1}$
     by \eqref{eq:quotient_exponent_balance},   weighted
      Cauchy--Schwarz
   in    the positive  Hermitian norm gives
   \begin{equation}\label{eq:weighted_log_defect_gradient_bound}
   |\partial\log
   E_{t}|_{A^{-1}}^2
   \le a|\partial\log
      V|_{A^{-1}}^2
     +4\varepsilon|\partial\log\rho|_{A^{-1}}^2.
  \end{equation}
 This inequality    will cancel the
     derivatives       of $\mathpart{V}$ in
  $\mathcal       {{L}}\log
    J$;    no separate gradient estimate for $\mathpart{V}$ is    needed.

 For every    positive smooth    function
     $\mathpart{{b}}$,    the chain rule   gives
    $\mathcal    L\log   b=(\mathcal    {L}b)/b-|\partial\log
    b|_{A^{-1}}^{2}$.
 Applying    it separately to the      three
   logarithms       in
   $\log  J=\log
  E_{t}-a\log   V-2\varepsilon\log\rho$    yields
    \begin{equation}%
    \label{eq:log_quotient_linearized_chain_rule}
      \begin{aligned}
       \mathcal
   L\log    J
      ={}&\frac{\mathcal
   L E_t}{E_t}-a\frac{\mathcal  {L}V}{V}
   -\frac{2\varepsilon
 S}{\rho}
  -|\partial\log
       E_t|_{A^{-1}}^2\\
      &+a|\partial\log
    V|_{A^{-1}}^2
  +2\varepsilon|\partial\log\rho|_{A^{-1}}^2.
     \end{aligned}
   \end{equation}
    Here  $\mathcal    {{L}}\rho=S$ follows
  from  \eqref{eq:linearized_potential_identities_and_inverse_trace_bound}.
     Also $\mathpart{\partial\rho=\overline    z}$ and  the
  largest       eigenvalue of
  $\mathpart{A^{-1}}$ is  at most
   $\mathpart{{S}}$,    so
    \begin{equation}\label{eq:inverse_hessian_log_weight_gradient_bound}
   |\partial\log\rho|_{A^{-1}}^{2}
      \le\frac{|z|^2}{\rho^2}S\le\frac
  S\rho.
     \end{equation}
     Using  \eqref{eq:weighted_log_defect_gradient_bound}   and
      \eqref{eq:inverse_hessian_log_weight_gradient_bound}     in
   \eqref{eq:log_quotient_linearized_chain_rule},    and
    then
       \eqref{eq:log_quotient_conditions_at_maximum},
     gives
   \begin{equation}\label{eq:log_quotient_differential_inequality}
    0\ge\mathcal
      L\log J
 \ge\frac{\mathcal
      L  E_t}{E_{t}}-a\frac{\mathcal LV}{V}
   -\frac{4\varepsilon
    S}{\rho}.
   \end{equation}
      By \eqref{eq:quotient_exponent_balance}
  and $\mathpart{V\le    M\rho}$ in
   \eqref{eq:barrier_growth_and_linearized_operator},
  \[
      \mathpart{\frac{4\varepsilon
  S}{\rho}}
 =\frac{amS}{4{M}%
 \rho}\le\frac{amS}{4V}.
   \]
     Substituting   this
  bound       and $\mathcal  {L}V=n-(m/2)S$ from
       \eqref{eq:barrier_growth_and_linearized_operator}
   into       \eqref{eq:log_quotient_differential_inequality} yields
    \begin{equation}\label{eq:linearized_defect_ratio_upper_bound}
  \frac{\mathcal
     L      E_t}{E_{t}}
  \le\frac
      aV\left(n-\frac     m4S\right).
     \end{equation}

      Since    $\mathpart{a>0}$ and $\mathpart{V>0}$,
  \eqref{eq:linearized_defect_ratio_upper_bound}    has a
     negative
   right-hand side if  $\mathpart{S>T=4n/m}$.
    On the  other  hand,
        \eqref{eq:linearized_green_defect_kappa_bound},
   \eqref{eq:corrected_green_defect_definition},
    and   \eqref{eq:barrier_growth_and_linearized_operator}   give
    $\mathcal {L}    E_t\ge\kappa((m/4)S-n)\ge0$,
    since   $\mathpart{\kappa\ge0}$.
    As  $E_{t}>0$, this
  contradicts \eqref{eq:linearized_defect_ratio_upper_bound}.
     Hence      $\mathpart{S\le
   {T}}$ at    $x_{0}$.
 We may    therefore use
     \eqref{eq:green_defect_bound_below_trace_threshold},  which,   after
   subtracting    $t^{2}/\mu$  and    using
     \eqref{eq:corrected_green_defect_definition},    reads
     \[
      E_{t}+\kappa
    V\le\frac{t^2}{n\mu}\mathcal   LQ_t.
 \]
    Since $E_{t}+\kappa      V>0$  and
  $t^{2}/\mu\le V$ by
     \eqref{eq:radius_term_bounded_by_barrier},
     this     implies
  \[
   \mathcal    {L}Q_t\ge\frac
    nV(E_t+\kappa    V).
       \]
     Using
 \eqref{eq:corrected_green_defect_definition},    subtracting
   $\kappa\mathcal
  {{L}}V=\kappa(n-(m/2)S)$ and   dividing by $E_{t}>0$    gives
   \begin{equation}%
     \label{eq:linearized_defect_ratio_lower_bound}
    \frac{\mathcal   L
  E_t}{E_t}
      \ge\frac       nV+\frac{\kappa
 mS}{2E_t}\ge\frac  {n}V.
       \end{equation}
     Equation~\eqref{eq:linearized_defect_ratio_upper_bound}
     instead bounds     the  same
  quantity    by
  $\mathpart{an/V<n/V}$,     since     $\mathpart{a<1}$.
   This contradiction rules   out every   positive
    interior maximum    of $\mathpart{{J}}$.

     \subsection{Estimate the boundary contribution}
%
  %
%
     %
%    -------------------------------------------------------------------------
%
  %

  %
%
  We  now bound $\mathpart{J}$  on    $\partial  B_{L}$   and transfer that
  bound to $E_{{t}}$.
  On    this sphere, $\rho=1+L^{2}$.
   By \eqref{eq:weighted_defect_quotient_definition},
  \eqref{eq:coarse_corrected_green_defect_bound},
   \eqref{eq:barrier_growth_and_linearized_operator},     and
  $\mathpart{{a}+\varepsilon=1}$ from
  \eqref{eq:quotient_exponent_balance},
  \[
   J\le\frac{2M%
     \rho}{((m/2)\rho)^{a}\rho^{2\varepsilon}}
      =2M\left(\frac2m\right)^a(1+L^2)^{-\varepsilon}.
  \]
   If   the maximum of
     $\mathpart{J}$ is    positive,   it lies on this  boundary by the
  preceding    subsection;    if
    it is   nonpositive,  the displayed upper
 bound holds   automatically. Using   \eqref{eq:weighted_defect_quotient_definition},
  multiply  by $V^{a}\rho^{2\varepsilon}$.
  The  bound   $\mathpart{V\le M\rho}$ from
   \eqref{eq:barrier_growth_and_linearized_operator}    and
    $\mathpart{a=1-\varepsilon}$  from
  \eqref{eq:quotient_exponent_balance} then   yield
    \begin{equation}%
      \label{eq:corrected_green_defect_radial_decay}
      E_{t}(z)%
     \le
    2M\left(\frac{2M}{m}\right)^{1-\varepsilon}
     \rho(z)\left(%
   \frac{\rho(z)}{1+L^2}\right)^\varepsilon
   \qquad
  (z\in B_{L}).
      \end{equation}
    For  $\mathpart{|z|<\sqrt  L}$,
     we  have      $\mathpart{\rho\le1+{L}}$ and
    $(1+L)/(1+L^{2})\le2/L$.
  Hence
    \[
  E_{t}\le4M\left(\frac
     Mm\right)^{1-\varepsilon}
     L^{-\varepsilon}\rho
   \le\frac{4M^2}{m}L^{-\varepsilon}\rho,
   \]
    where   the last  inequality    uses
 $\mathpart{M\ge m}$.
    Finally,  substitute
     $Q_{t}=E_{t}+t^2/\mu+\kappa    V$ from
    \eqref{eq:corrected_green_defect_definition},
 and   use    $\mathpart{V\le   M\rho}$ and
  $\kappa=Kt^{2}/[2m(n+1)]$.     Together
    with    $Q_{t}\ge0$    from
   \eqref{eq:green_defect_bounded_by_sphere_defect},
   this  gives
    \begin{equation}%
     \label{eq:smooth_positive_density_green_defect_bound}
     0\le
   Q_t(z)\le\frac{t^2}{\mu}
      +\frac{MKt^{2}}{2m(n%
   +1)}\rho(z)
    +\frac{4M^2}{m}L^{-\varepsilon}\rho(z),
    \qquad
   |z|<\sqrt     L.
   \end{equation}
   Taking $\mathpart{\Lambda=2(n+1)/\mu}$
  proves \eqref{eq:finite_ball_green_defect_bound}    for
     smooth     solutions with
  positive smooth     density.

  \begin{remark}[The domain needed    by    the
    smooth  argument]
     \label{rem:local-domain}
   The   maximum-principle
  calculation   takes    place on $\overline{B_{L}}$,
     and  all translates used in the
   averaging   identities lie     in
  $\mathpart{B_{L+1}}$ because
     $\mathpart{t\le1}$.
   The    proof    therefore
   remains valid   when the solution    and density
 are smooth     on a
   neighborhood of $\mathpart{\overline{B_{{L}+1}}}$,   with  the
     equation,    density
      bounds,  and   growth bounds holding on
   $\mathpart{\overline{B_{L+1}}}$.
    The hypotheses     of Theorem~\ref{thm:finite-ball}
  are    posed  on
%
%  $\mathpart{B_{2AL}}$. After the reduction     in
 $\mathpart{B_{2b_{\mathrm h}L}}$. After the reduction     in
    \eqref{eq:normalized_two_sided_growth_on_ball},    the
   normalized bounds hold on  $\mathpart{B_{2L}}$,
 which provides
  room  for the    approximations in   the    next   two sections.
       \end{remark}
    \section{Nonnegative densities and continuous potentials}
     \label{sec:nonnegative-density}
%
% =========================================================================
    % Section 4   | Nonnegative
%
%   density  and continuous potential
%
    % The  comparison estimate
% uses the nth root of the density. It    controls the
%
 % error  made by replacing the density by
%
  %   a smooth positive    one, including
   %   near
%
  %   its   zeros.
%
   % Equation index
  %   (printed number and source label)
%   4.1        eq:determinant_root_identity_shift_bound
%
   %        4.2    eq:dirichlet_stability_in_density_roots
%
      %      4.3     eq:shifted_barrier_monge_ampere_measure_bound
%
%   4.4    eq:dirichlet_solutions_cauchy_bound
    %    4.5    eq:decreasing_shifted_dirichlet_subsequence
%
   %   4.6    eq:green_average_bound_continuous_potential
   %      4.7    eq:positive_smooth_density_root_approximation
%
%     4.8       eq:mollified_density_root_laplacian_and_convergence
%
 %   4.9    eq:dirichlet_boundary_and_density_error
   %   4.10   eq:smooth_dirichlet_approximants_uniform_error
   %    4.11   eq:quadratic_growth_bounds_smooth_approximants
%
%
 %    4.12   eq:green_estimate_constants_smooth_approximants
%
 %    4.13   eq:green_average_bound_smooth_dirichlet_approximants
     %   =========================================================================
     %

 We now  allow      the density to  vanish and the
   potential    to be   merely
   continuous. We    first   prove  a comparison
     estimate   that controls a
      Dirichlet solution when its boundary   values
    and      density change.
    We then  use   this estimate
     to       approximate the  potential uniformly by
 smooth solutions with positive
    densities.  This    uniform convergence
    will preserve      the growth constants in
  Theorem~\ref{thm:finite-ball}.

      \subsection{Comparison of Dirichlet solutions}
%
  %
%
      %      -------------------------------------------------------------------------
%
  %

%
   The matrix inequality    behind the
  comparison   is
    \begin{equation}\label{eq:determinant_root_identity_shift_bound}
    \mathpart{\det(A+dI)^{1/n}}\ge\det(A)^{1/n}+d
   \qquad(A%
      \ge0,\ d\ge0).
     \end{equation}
    It  follows from  concavity and homogeneity
      of $\mathpart{\det^{1/n}}$;     continuity
   includes singular    matrices. The mixed
    Monge--Amp\`ere        inequality
 \cite{Dinew} provides   the
   corresponding   statement  for bounded
  plurisubharmonic      functions,   which
    we    use     next.

   \begin{lemma}[Stability    under   changes in     the
     density   $g_i^{n}$]
  \label{lem:root-stability}
     Let  $u_1,u_{2}$ be
 continuous on $\cl{B_{{R}}}$ and   plurisubharmonic
   on $B_{{R}}$,   and
     suppose
      \[
     (\sqrt{-1}\partial\overline{\partial}u_i)^n=g_i^n\omega_0^n\quad\text{in
      }B_R,
    \qquad    g_{i}\in
  C(\cl{B_R}),\quad  g_i\ge0.
     \]
     Set
    \[
      b=\|u_{1}-u_2\|_{L^\infty(\partial
    B_R)},\qquad
        d=\|g_1-g_2\|_{L^\infty(B_{R})}.
    \]
   Then
   \begin{equation}\label{eq:dirichlet_stability_in_density_roots}
    |u_{1}(z)-u_2(z)|\le
     b+d(R^2-|z|^2)
  \quad\text{for
      }z\in\cl{B_R}.
      \end{equation}
   In   particular,
      $\|u_1-u_{2}\|_{L^\infty(B_{R})}\le  b+R^2d$.
  \end{lemma}
       \begin{proof}
    To   bound  $u_1-u_{2}$  from
  above, we  add $d|z|^{2}$ to $u_{1}$
  so    that its Monge--Amp\`ere measure
  is  at      least
  $g_2^{n}\omega_0^n$. We then   subtract
   $dR^{2}+b$ so that
   the resulting    function is
  at most $u_{2}$ on $\partial    B_{{R}}$.
   Thus
   we  set
     \[
   v_1=u_1-b+d(|{z}|^2%
    -R^{2}).
    \]
   Since    $u_{1}$ is
     plurisubharmonic and $\mathpart{d\ge0}$,  so    is      $v_{1}$.
     Dinew's     mixed
   Monge--Amp\`ere
     inequality \cite[Theorem 1.2]{Dinew},     applied
    to   $u_{1}$    and
   $d|z|^{2}$,
    gives
     \begin{equation}\label{eq:shifted_barrier_monge_ampere_measure_bound}
    (\sqrt{-1}\partial%
    \overline{\partial}v_1)^n\ge(g_1+d)^n\omega_0^n
     \ge
    g_2^{n}\omega_0^n=(\sqrt{-1}\partial\overline{\partial}u_2)^n.
      \end{equation}
 On    $\partial B_{R}$  we
   have  $v_{1}=u_1-b\le u_2$.
    The       Bedford--Taylor   comparison
    principle \cite{BT},     with
 \eqref{eq:shifted_barrier_monge_ampere_measure_bound}, therefore
      yields $v_{1}\le    u_2$ in    $B_{R}$.
 Rearranging    this
   inequality       gives
  $u_{1}-u_2\le
  b+d(R^2-|z|^2)$.
     Interchanging $u_{1}$    and
   $u_{2}$ proves \eqref{eq:dirichlet_stability_in_density_roots}.
  \end{proof}
   We    will   also    need
    continuous  Dirichlet solutions   when  the   prescribed
    root      vanishes.  The preceding  lemma constructs
    them  directly from
    smooth    positive-density
     solutions.
  \begin{lemma}[Continuous   Dirichlet
  solutions]
       \label{lem:continuous-dirichlet}
     Let $g\in  C(\cl{B_{R}})$
  be     nonnegative,  and let
    $\psi\in     C^\infty(\partial
     B_{R})$    be real valued.
    There is a  unique   function
   $v\in  {C}(\cl{B_{R}})$, plurisubharmonic
   on  $B_{R}$,  such
     that
      \[
    (\sqrt{-1}\partial\overline{\partial}v)^n=g^n\omega_0^{n}\quad\text{in
   }B_R,
     \qquad v=\psi\quad\text{on
   }\partial   B_{R}.
      \]
     \end{lemma}

  \begin{proof}
   Choose smooth  functions $g_{j}>0$  on     a neighborhood
     of  $\cl{B_{R}}$
    that converge uniformly    to    $\mathpart{g}$.
    The smooth Dirichlet   theorem
     \cite{CKNS}    provides    smooth    strictly
    plurisubharmonic  solutions
      $v_{{j}}$  on $\cl{B_{{R}}}$ with   density $g_{j}^{n}$ and
   boundary   values
     $\mathpart{\psi}$. Its  subsolution requirement is   met
     by    adding a sufficiently
      large multiple    of $|z|^2-R^{2}$ to  a
 smooth extension of   $\mathpart{\psi}$.
     Lemma~\ref{lem:root-stability}
    gives
%
% Uniform convergence of the density roots yields uniform convergence of the Dirichlet solutions.
    \begin{equation}\label{eq:dirichlet_solutions_cauchy_bound}
      \|v_{j}-v_k\|_{L^\infty(B_R)}
      \le
       R^2\|g_j-g_k\|_{L^\infty(B_R)}.
   \end{equation}
     Since $g_{j}\to     g$ uniformly,
    \eqref{eq:dirichlet_solutions_cauchy_bound}
     makes the   continuous plurisubharmonic
     functions $v_{j}$   uniformly
 Cauchy.  Their uniform limit $\mathpart{{v}}$  is
     continuous and      plurisubharmonic
  and retains      their
    common    boundary   value  $\mathpart{\psi}$.

  To identify   the
    Monge--Amp\`ere measure       using  decreasing convergence,
 choose  a subsequence
  $v_{j_{k}}$ such     that
     $\|v_{j_{k+1}}-v_{j_{k}}\|_\infty\le2^{-k}$,
  which is     possible
     by
  \eqref{eq:dirichlet_solutions_cauchy_bound}. Define
    \begin{equation}%
   \label{eq:decreasing_shifted_dirichlet_subsequence}
     c_{k}=\sum_{\ell=k}^{\infty}%
    2^{-\ell},\qquad
      \widetilde
    v_k=v_{j_k}+c_k.
        \end{equation}
  Since $c_{k}-c_{k+1}=2^{-k}$, the
      preceding    bound   gives
    $\widetilde        v_{k+1}\le\widetilde
 v_{k}$, and
     \eqref{eq:decreasing_shifted_dirichlet_subsequence}
   gives
     $\widetilde
    v_{k}\downarrow v$.
   The construction   gives
        $(\sqrt{-1}\partial\overline{\partial}v_{j_k})^n=g_{j_k}^n\omega_{0}^{n}$
    and $g_{j_{{k}}}\to    g$
 uniformly.  Since  adding $c_{k}$ leaves  the
     Monge--Amp\`ere   measure unchanged,
   Bedford--Taylor   continuity
     under  decreasing   convergence
   \cite{BT} now      yields
    \[
     (\sqrt{-1}\partial\overline{%
     \partial}v)^n
       =\lim_{k\to\infty}(\sqrt{-1}\partial\overline{\partial}\widetilde
    v_{k})^n
    =\lim_{k\to\infty}g_{j_k}^n\omega_0^{n}
      =g^n%
   \omega_0^n.
     \]
   Uniqueness follows
    from Lemma~\ref{lem:root-stability} with     $\mathpart{b=d=0}$.
    \end{proof}
   \subsection{Approximate the potential on the interior ball}
  %
%
      %
  %   -------------------------------------------------------------------------
 %
%
 %

 The following    local    form
   also applies   to the    continuous  replacements
      used in    the next
  section.    For the    normalized potential $\mathpart{w-q+m}$ from
     Theorem~\ref{thm:finite-ball},
   the     upper    majorant in   the    proposition is supplied by
  \eqref{eq:normalized_two_sided_growth_on_ball}.
    We         solve
     the    approximating
  Dirichlet    problems on $B_{{R}}$, with $\mathpart{R=3L/2}$,
      and    assume   that
 $\mathpart{f}$ is  defined on    $\mathpart{B_{2L}}$   so that
 its $\mathpart{n}$th root $\mathpart{f^{1/n}}$ can      be
    mollified on a neighborhood of
   $\overline{B_{R}}$. The    smooth estimate uses
    the equation and    growth
  bounds  only on
  $\mathpart{\overline{B_{L+1}}}$, as   recorded  in
       Remark~%
    \ref{rem:local-domain}.
    \begin{proposition}[The power    estimate    on
    the  interior    ball]
      \label{prop:continuous-power}
    Let $\mathpart{{L}\ge16}$,   $\mathpart{{R}=3L/2}$, $\mathpart{{m}>0}$,  and   $M_{0}>0$.
  Let $u\in    C(\cl{B_{R}})$ be
  plurisubharmonic on   $B_{R}$  and
    satisfy
  \[
     (\sqrt{-1}\partial\overline{\partial}u)^n=f\omega_0^{n}\quad\text{in
    }B_R,
   \qquad  m|z|^2\le
 u(z)\le   M_0\rho(z)\quad\text{on  }\cl{B_{L+1}}.
  \]
      Assume   that $\mathpart{f\in C(B_{2L})}$ is  nonnegative
   and that,  for    some
   $\mathpart{{f}_{\max}>0}$ and
      $\mathpart{K\ge0}$,
       \[
     \mathpart{f}\le
   f_{\max},\qquad
  \Delta_\R(f^{1/n})\ge-K\quad\text{in
      }B_{2L},
    \]
    where      the last  inequality is distributional.
  Take    $M=M_{0}+m$.
   If $\mathpart{u\ge {m}\rho}$  on
    $\mathpart{\cl{B_{L+1}}}$, one   may take $M=M_{0}$. Set
    \[
    \mu=\frac{m^{n-1}}{f_{\max}(4n)^{n-1}},\qquad
      t_0=\min\left\{1,\sqrt{\frac{m\mu}{2}}\right\},\qquad
      \varepsilon=\frac{m}{16M+m},\qquad
     C=\frac{4M^{2}}{m}.
  \]
     Then
   \begin{equation}%
    \label{eq:green_average_bound_continuous_potential}
     0\le\mathcal
    G_tu(z)-u(z)
     \le\frac{t^{2}}{\mu}
     +\frac{MKt^2}%
      {2m(n+1)}\rho(z)
       +C
     L^{-\varepsilon}\rho(z)
   \end{equation}
 for   $\mathpart{|z|<\sqrt
 L}$ and  $0<t\le t_{0}$.
      \end{proposition}

     \begin{proof}
     Adding   $\mathpart{{m}}$   to $\mathpart{u}$
     gives   $m\rho\le u+{m}\le(M_{0}+m)\rho$ on
 $\mathpart{\cl{B_{L+1}}}$     and
      preserves both the equation     and $\mathcal {G}_{t}u-u$.
   When $\mathpart{{u}\ge      m\rho}$
   already, we   use   $\mathpart{u}$ itself and ${M}=M_{0}$. Denote the
   resulting    potential
  again by $\mathpart{u}$, so that $\mathpart{m\rho\le u\le   M\rho}$.

 We approximate the density root    so that
  its Laplacian lower bound
    is unchanged.    Write $\mathpart{g=f^{1/n}}$, choose
      a nonnegative     smooth radial
    convolution    kernel $\mathpart{\chi}$
   of  integral one supported  in ${B}_{1}$, and  set
       $\mathpart{\chi_\delta(z)=\delta^{-2n}%
      \chi(z/\delta)}$.
    For      sequences   $\delta_{{j}}\downarrow0$     and
  $s_{j}\downarrow0$      with
    $0<\delta_{j}<
    L/4$ and  $s_{j}>0$,   define
%
  %
%
% Adding a positive constant after convolution prevents zeros of the approximating density.
  \begin{equation}\label{eq:positive_smooth_density_root_approximation}
     g_{j}=g*\chi_{\delta_j}+s_j,
 \qquad
   f_j=g_j^n,
     \qquad
    f_{\max,j}=(f_{\max}^{1/n}+s_j)^{n}.
    \end{equation}
   These functions  are   smooth on    a neighborhood of
      $\cl{B_{R}}$.
   The    positivity and total   mass of
  the    kernel, together with $s_{j}>0$
   and     the density
     bounds in   Proposition~\ref{prop:continuous-power}, give
    \begin{equation}%
      \label{eq:mollified_density_root_laplacian_and_convergence}
    g_{j}>0,\qquad   f_j\le
    f_{\max,j},\qquad
  \Delta_\R g_j=(\Delta_\R
   g)*\chi_{\delta_j}\ge-K,
    \qquad
  \|g_j-g\|_{L^\infty(B_R)}\longrightarrow0.
      \end{equation}
   The    final convergence  follows
 from     uniform   continuity  of $\mathpart{g}$ on a
     compact neighborhood   of
     $\cl{B_{R}}$.
     
    We next  approximate   the
     boundary values. Choose   smooth  functions
    $\psi_{j}$    on $\partial  B_{R}$ converging uniformly
    to  $u|_{\partial B_{R}}$.
  When   $\mathpart{u}$ is   continuous  on
  $\mathpart{B_{2L}}$,  as in
 the  continuous case of
     Theorem~\ref{thm:finite-ball},  we may  take
      $\psi_{j}=(u*\chi_{\delta_j})|_{\partial
  B_R}$.
    For the  local
    statement of the  proposition, a smooth   uniform
   approximation to      the
   continuous    boundary trace suffices.

     The     smooth Dirichlet
   theorem \cite{CKNS} gives a   smooth  strictly
     plurisubharmonic function $w_{j}$  on
     $\cl{B_{{R}}}$ with    density    $f_{j}$
    and  boundary  values
        $\psi_{j}$.
  To  compare     it
      with $\mathpart{{u}}$,   put
     \begin{equation}\label{eq:dirichlet_boundary_and_density_error}
   e_j=\|\psi_j-u\|_{L^\infty(\partial
    B_R)}
      +R^{2}\|g_{j}-g\|_{L^\infty(B_R)}.
    \end{equation}
   Since   $\psi_j\to
     u|_{\partial      B_{R}}$ uniformly,
       Lemma~\ref{lem:root-stability}
    and
  \eqref{eq:mollified_density_root_laplacian_and_convergence}
    imply
    \begin{equation}\label{eq:smooth_dirichlet_approximants_uniform_error}
      \|w_j-u\|_{L^\infty(B_{R})}\le
    e_j\longrightarrow0.
  \end{equation}

    We    now recover the   growth bounds needed in
   the smooth   estimate.
    Since        $\mathpart{\rho\ge1}$,
       \eqref{eq:smooth_dirichlet_approximants_uniform_error}
    and    the bounds  $\mathpart{m\rho\le
    u\le   M\rho}$ give
    \begin{equation}\label{eq:quadratic_growth_bounds_smooth_approximants}
    m_j\rho\le     w_j\le M_j\rho\quad\text{on
      }\cl{B_{L+1}},
     \qquad m_j=m-e_{j},\quad
     M_j=M+e_j.
     \end{equation}
    For   all sufficiently
      large $\mathpart{j}$,    $m_{j}>0$.
  The  smooth estimate \eqref{eq:smooth_positive_density_green_defect_bound}
    therefore   applies   with
 the
 parameters
      \begin{equation}%
     \label{eq:green_estimate_constants_smooth_approximants}
   \begin{gathered}
   \mu_j=\frac{m_j^{n-1}}{f_{\max,j}(4n)^{n-1}},\qquad
     t_{0,j}=\min\left\{1,\sqrt{\frac{m_{j}\mu_j}{2}}\right\},\\
    \varepsilon_j=\frac{m_j}{16M_j+m_j},\qquad
    C_j=\frac{4M_j^2}%
      {m_j}.
   \end{gathered}
     \end{equation}
    Here   Remark~\ref{rem:local-domain}
 justifies
    using \eqref{eq:smooth_positive_density_green_defect_bound}      on
   $B_{R}$; the   growth  bounds on
    $\mathpart{B_{L+1}}$ are     supplied
      by     \eqref{eq:quadratic_growth_bounds_smooth_approximants}.
      Consequently,
     \begin{equation}\label{eq:green_average_bound_smooth_dirichlet_approximants}
  0\le\mathcal
  G_tw_j(z)-w_j(z)
      \le\frac{t^2}{\mu_j}
  +\frac{M_jKt^2}{2m_j(n+1)}\rho(z)
  +C_jL^{-\varepsilon_{j}}\rho(z)
    \end{equation}
    for    $\mathpart{|z|<\sqrt
  L}$ and $\mathpart{0<t\le   t_{0,j}}$.

 Fix  $0<t<t_{0}$.
   Since    $e_{j}\to0$       by
   \eqref{eq:smooth_dirichlet_approximants_uniform_error}
   and   $s_{j}\to0$,
    we have $m_{j}\to  m$, $M_{j}\to    M$, and
   $\mathpart{f_{\max,j}\to    f_{\max}}$.
   Formula     \eqref{eq:green_estimate_constants_smooth_approximants}
   therefore       gives
 $t_{0,j}\to      t_{0}$,
 so     \eqref{eq:green_average_bound_smooth_dirichlet_approximants}
 holds   for    this  $\mathpart{t}$ when
  $\mathpart{{j}}$ is      sufficiently large.
     The    averaging operator has nonnegative
  weights  of total mass     one
     by       \eqref{eq:green_average_from_ball_averages};
    hence
        \eqref{eq:smooth_dirichlet_approximants_uniform_error}
 also gives
   $\mathcal     G_tw_{j}\to\mathcal G_tu$ uniformly
     wherever          these
   averages
      are     used.

    Since    $m_{j}\to     m$, $M_{{j}}\to M$,
      and   $\mathpart{f_{\max,j}\to   f_{\max}}$,
     the   parameter formulas
    \eqref{eq:green_estimate_constants_smooth_approximants}
 allow  passage   to the limit
    in   \eqref{eq:green_average_bound_smooth_dirichlet_approximants},
  giving    \eqref{eq:green_average_bound_continuous_potential}
  for   $0<t<t_{0}$.
  Continuity    of $\mathpart{{u}}$ makes    $\mathcal
    {G}_tu(z)$    continuous     for $\mathpart{{t}>0}$,
    so letting  $t\uparrow    t_{0}$
     proves the endpoint  as  well.
      \end{proof}
 For     a continuous   potential
   on  $\mathpart{B_{2{L}}}$, its   restriction      to
   $\mathpart{\cl{B_{3L/2}}}$      satisfies
  the domain    requirements   of
      Proposition~%
   \ref{prop:continuous-power}.
 With   $\mathpart{\Lambda=2(n+1)/\mu}$,
  \eqref{eq:green_average_bound_continuous_potential}
     proves \eqref{eq:finite_ball_green_defect_bound}    for    continuous
      Bedford--Taylor
  solutions, including   densities      that vanish.    The
 next section removes   continuity of the potential.
   \section{Locally bounded Bedford--Taylor solutions}\label{sec:approximation}
      %
%
  %  =========================================================================
%
% Section   5 | Locally
    %   bounded Bedford--Taylor solutions
 %    The    boundary data decrease
%
  %
%
     %   on the     interior ball. The Dirichlet
%
 % replacements solve the   equation with the  same
%
 %   density   as the original
%    potential.
   %
%
% Equation index (printed   number
%
    %   and source label)
  %
%
   %      5.1    eq:decreasing_mollifications_of_potential
%   5.2    eq:determinant_root_bellman_formula
    %   5.3      eq:mollified_potential_determinant_lower_bound
%
   %   5.4     eq:corrected_mollifications_density_and_monotonicity
%
%   5.5       eq:exact_density_dirichlet_replacements
%
     %         5.6    eq:ordering_of_dirichlet_replacements
%
 %
    %   5.7    eq:glued_replacement_monge_ampere_lower_bound
    %
%
     %      5.8     eq:dirichlet_replacements_decrease_to_solution
%
 %      5.9    eq:quadratic_growth_bounds_dirichlet_replacements
%
    %      5.10     eq:green_average_bound_bedford_taylor_potential
%
  %
%
 %   =========================================================================

     Let $\mathpart{w}$     satisfy
     Theorem~\ref{thm:finite-ball}  and  put $\mathpart{u=w-q+m}$.
  By   \eqref{eq:normalized_two_sided_growth_on_ball},
      on $\mathpart{B_{2L}}$ we     have
    \[
  (\sqrt{-1}\partial\overline{\partial}u)^n=f\,\omega_{0}^n,\qquad
     m\rho\le
    u\le          M\rho   .
   \]
  We    use this bound only on the    interior ball.
     Write   $\mathpart{g=f^{1/n}}$. We    construct
    continuous  solutions of      this same   equation that
  decrease  to   the
    given pointwise   values    of
     $\mathpart{u}$. The continuous estimate from
      Proposition~\ref{prop:continuous-power}
      will then pass  to  $\mathpart{u}$.

  \subsection{Construct decreasing smooth boundary data}
      %   -------------------------------------------------------------------------
% Mollification first gives smooth potentials. The quadratic
   % correction   then
%
 % ensures a   lower bound on    their
%
 %  Monge--Ampère measures while    preserving the
%
 % order  of
%
%   the sequence.

     Set $\mathpart{R=3L/2}$   and $\mathpart{R'=7L/4}$.  These  radii
   leave    room   both for
     convolution       inside
      $\mathpart{B_{2{L}}}$    and for comparison outside    $B_{R}$.
  Choose   a nonnegative radial smooth function
  $\mathpart{{\vartheta}}$, supported
     in ${B}_{1}$ and   of integral    one.  For
     $\varepsilon_{j}\downarrow0$  with
    $\varepsilon_{1}<L/8$,
   define
    \[
  u_j=u*\vartheta_{\varepsilon_j},\qquad
     g_j=g*\vartheta_{\varepsilon_{j}}.
     \]
   Both convolutions   are  defined    on   a
      neighborhood of
    $\mathpart{\overline{B_{R'}}}$.    The functions
     $u_{j}$ are   smooth and
   plurisubharmonic.  Radial
    integration expresses  each $u_{j}(z)$ as
   a positive  weighted   average
 of spherical  means of $\mathpart{u}$ about $\mathpart{z}$.
    These   means    increase with  their   radius
   because $\mathpart{u}$ is  subharmonic.
     The       submean    inequality  and    upper   semicontinuity
    therefore    give
     \begin{equation}%
  \label{eq:decreasing_mollifications_of_potential}
       u_{j}\downarrow
    u\quad\text{pointwise on    }B_{{R}'}.
      \end{equation}
     Continuity   of $\mathpart{g}$ also  gives
   $g_{j}\to     g$    uniformly   on
     $\mathpart{\overline{B_{R'}%
     }}$.
     We next compare    the $\mathpart{n}$th
  root $(\det   D^{2}_\C  u_j)^{1/n}$ of
  the    complex    Hessian
 determinant   with $g_{j}$.
   Fix  a positive Hermitian       matrix    $\mathpart{H}$
  with $\mathpart{\det H=1}$, and    set
   $q_{H}(z)=H[z]$.    The assumed
   equation      gives
    $(\sqrt{-1}\partial\overline{\partial}u)^n=g^{n}\omega_0^n$,    while
    $\mathpart{\det  {H}=1}$   gives
      $(\sqrt{-1}\partial\overline{\partial}q_H)^n=\omega_{0}^n$.
    Dinew's  mixed    Monge--Amp\`ere
  inequality  \cite[Theorem 1.2]{Dinew}, applied
     on     relatively
    compact         balls
     to these    two measures,    gives
   \[
      \sqrt{-1}\partial\overline{\partial}u\wedge(\sqrt{-1}\partial\overline{\partial}q_H)^{n-1}\ge
    g\,\omega_{0}^n.
    \]
    Since  $\mathpart{{H}}$   is
      constant,
     the displayed    inequality     is   the linear
   distributional   inequality
    $\mathpart{\tr(H^{-1}(u_{i\bar       j}))\ge
     ng}$. Convolution   thus gives
    $\tr(H^{-1}D^{2}_{\C}u_{j})\ge
    ng_j$.

   For    every positive  semidefinite
  Hermitian    matrix    $\mathpart{A}$,
 \begin{equation}\label{eq:determinant_root_bellman_formula}
     (\det
     A)^{1/n}
   =\inf_{H>0,\,\det
   H=1}\frac{1}n\tr(H^{-1}A).
    \end{equation}
    For $\mathpart{A>0}$, the    arithmetic--geometric  mean
   inequality   proves this
   identity, with   equality
  at $\mathpart{H=(\det     A)^{-1/n}A}$;   singular   $\mathpart{A}$
   follows    by  a
     limit.    Since
 $\tr(H^{-1}D^2_{\C}u_j)\ge  ng_{j}$
      for   every   $\mathpart{H>0}$ with
      $\mathpart{\det {H}=1}$, taking    the infimum
   in   \eqref{eq:determinant_root_bellman_formula}    yields
       \begin{equation}\label{eq:mollified_potential_determinant_lower_bound}
     \bigl(\det
  D^2_{\C}u_{j}\bigr)^{1/n}\ge   g_j.
    \end{equation}
    The root  $g_{j}$  need   not   lie  above $\mathpart{g}$, so we   correct
     the   boundary
      data    before solving    the exact equation.
   Choose   numbers  $a_{{j}}\downarrow0$
     such
  that
      \[
     a_j\ge\|g_j-g\|_{L^\infty(\overline{B_{R'}})},\qquad
 \phi_j=u_j+a_{j}\rho.
    \]
     For example,    take the supremum of the
    displayed norms over  all
   indices at     least
   $\mathpart{{j}}$. The   determinant-shift  inequality
   \eqref{eq:determinant_root_identity_shift_bound}   and
 \eqref{eq:mollified_potential_determinant_lower_bound}  give
     \begin{equation}\label{eq:corrected_mollifications_density_and_monotonicity}
      \bigl(\det
       D^2_{\C}\phi_{j}\bigr)^{1/n}
  \ge     g_j+a_{j}\ge
     g,\qquad
   \phi_j\downarrow
      u.
      \end{equation}
  The  second  assertion
      follows     from  \eqref{eq:decreasing_mollifications_of_potential} and
  $a_{j}\downarrow0$.    Thus the  smooth
   functions $\phi_{j}$  decrease pointwise   to $\mathpart{u}$ and
     satisfy    $(\sqrt{-1}\partial\overline{\partial}\phi_j)^n\ge   f\,\omega_0^{n}$.
   Their restrictions to
      $\partial
 B_{R}$ will supply    the boundary  data.

   \subsection{Solve the exact equation and identify the limit}
 %   -------------------------------------------------------------------------
%
  %

 %
    Lemma~\ref{lem:continuous-dirichlet}
      supplies      continuous
  plurisubharmonic      solutions
  $v_{j}$  on $\overline{B_{R}}$ satisfying
% Only the boundary values vary with the index in these Dirichlet problems.
       \begin{equation}%
     \label{eq:exact_density_dirichlet_replacements}
     (\sqrt{-1}\partial\overline{\partial}v_j)^n=f\,\omega_0^n\quad\text{in
   }B_R,
  \qquad   v_{j}=\phi_{j}\quad\text{on
   }\partial        B_R.
    \end{equation}
   Comparison   with
   \eqref{eq:corrected_mollifications_density_and_monotonicity}     gives
  $v_{j}\ge\phi_j$. Comparison between    consecutive
   solutions,  whose
 densities      agree and whose
     boundary  data decrease, gives
 \begin{equation}\label{eq:ordering_of_dirichlet_replacements}
 v_{j+1}\le v_j,\qquad
    v_j\ge\phi_{j}\ge    u
   \quad\text{on
  }B_R.
     \end{equation}
  These    inequalities   show   that $v_{j}$ decreases to   a
 limit    greater than
 or equal to $\mathpart{u}$. To prove that      this    limit equals $\mathpart{u}$,
      we      extend $v_{j}$
   to    functions    $\widetilde  {v}_{j}$
  on    the  larger ball $\mathpart{B_{R'}}$.  We will
 show that their limit $\mathpart{\widetilde v}$
     agrees      with  $\mathpart{{u}}$ outside $B_{R}$:
   \[
  \widetilde   v=u\qquad\text{on
    }B_{R'}\setminus    B_{R}.
       \]
  We can then     apply the comparison
      principle    on an      intermediate ball
   $\mathpart{B_{\widehat   R}}$, where
   $\mathpart{R<\widehat    R<R'}$. Since  the   functions agree
   near its boundary, this  comparison    requires  no
    boundary    continuity
     of     $\mathpart{u}$. We    therefore
     define
     \[
 \widetilde
   v_j(z)=
   \begin{cases}
      v_{j}(z),&z\in
       B_R,\\
    \phi_{j}(z),&z\in
  B_{R'}\setminus B_R.
      \end{cases}
     \]
   The two definitions agree on $\partial
 B_{R}$.    Since
 $v_{j}\ge\phi_j$ inside, the plurisubharmonic
     pasting   lemma        applies:
    inside the ball the       extension is
   $\max(v_{j},\phi_j)$, and   outside
   it    is $\phi_{j}$. Thus    $\widetilde
     v_{j}$ is    plurisubharmonic  on
 $\mathpart{B_{R'}%
 }$.

  Its  Monge--Amp\`ere
     measure equals $f\,\omega_0^{{n}}$ in $B_{R}$
 by \eqref{eq:exact_density_dirichlet_replacements},    and is  at
 least   that measure
  outside    $\overline{B_{R}}$   by
    \eqref{eq:corrected_mollifications_density_and_monotonicity}.
     The remaining    sphere
    has zero $\omega_{0}^n$-measure,   whereas the
      Bedford--Taylor   measure      is
   nonnegative. Splitting   a Borel set
     into    these
      three parts  therefore gives
      \begin{equation}%
    \label{eq:glued_replacement_monge_ampere_lower_bound}
       (\sqrt{-1}\partial\overline{\partial}\widetilde
    v_j)^{n}\ge      f\,\omega_0^{n}
    \quad\text{on
       }B_{R'}.
       \end{equation}
     By      \eqref{eq:ordering_of_dirichlet_replacements}
 inside  $B_{{R}}$    and
   \eqref{eq:corrected_mollifications_density_and_monotonicity}
  outside, the functions
    $\widetilde   {v}_{j}$  decrease  and
    lie   between  $\mathpart{u}$    and
    $\widetilde   v_{1}$. The   decreasing-limit   property
     of    plurisubharmonic
   functions  therefore  makes their limit  $\mathpart{\widetilde
  v}$  locally    bounded
    and       plurisubharmonic.
  Decreasing
     Bedford--Taylor continuity    \cite{BT}
    and      \eqref{eq:glued_replacement_monge_ampere_lower_bound}
    give
    \[
   (\sqrt{-1}\partial\overline{\partial}\widetilde
     v)^n\ge    f\,\omega_0^{n}=(\sqrt{-1}\partial\overline{\partial}u)^n,
      \qquad
      \widetilde  v\ge    u.
       \]
  Moreover,  \eqref{eq:corrected_mollifications_density_and_monotonicity}
   gives
     $\mathpart{\widetilde   v=u}$  on
  $B_{R'}\setminus    B_{{R}}$.
 Choose   $\mathpart{\widehat  R}$ with
     $\mathpart{R<\widehat R<R'}$.  The two      bounded
   plurisubharmonic
  functions agree    in  a   neighborhood  of
     $\mathpart{\partial B_{\widehat  {R}}}$,    so their
 boundary      comparison requires
   no  continuity    of    $\mathpart{u}$.
    Applying the    comparison principle
   \cite[Lemma 1.6]{EGZ}   to this     boundary
    equality    and
      $(\sqrt{-1}\partial\overline{\partial}\widetilde
    v)^{n}\ge(\sqrt{-1}\partial\overline{\partial}u)^n$    gives
     $\mathpart{\widetilde    v\le
  u}$ on $\mathpart{B_{\widehat {R}}}$.
   Together  with $\mathpart{\widetilde   v\ge
     u}$, this    proves
   \begin{equation}%
 \label{eq:dirichlet_replacements_decrease_to_solution}
     v_{j}\downarrow    u\quad\text{pointwise
  on   }B_R.
      \end{equation}
    \subsection{Pass the local power estimate to the limit}
%
      %
%
  %    -------------------------------------------------------------------------
%
  %

  Since   $\mathpart{u\ge    m\rho}$,
  \eqref{eq:ordering_of_dirichlet_replacements}    gives
      $v_{j}\ge m\rho$.   To    recover
  the   upper bound, fix
       $\mathpart{\eta>0}$.
  Since $\mathpart{L+1<R}$, the closed  ball
  $\mathpart{\overline{B_{{L}+1}}}$ is   compactly contained in
   $B_{R}$.  We  claim that,
    for all  sufficiently
  large $\mathpart{j}$,
% The additive error is absorbed since the weight is at least one.
      \begin{equation}\label{eq:quadratic_growth_bounds_dirichlet_replacements}
 m\rho\le   v_{j}\le
        M\rho+\eta\le(M+\eta)\rho
   \quad\text{on
     }\overline{B_{L+1}}.
   \end{equation}
     Indeed,  if no  such   index existed,
  the    sets
    \[
 E_j=\{z\in\overline{B_{L+1}}:v_{j}(z)-M\rho(z)\ge\eta\}
     \]
      would   be    nonempty. By
    \eqref{eq:ordering_of_dirichlet_replacements} and   continuity
      of    $v_{j}$, they
 are  decreasing compact  sets.    Their intersection
   would contain   a
    point where  $\mathpart{{u}-M\rho\ge\eta}$,   by
    \eqref{eq:dirichlet_replacements_decrease_to_solution},
   contradicting $\mathpart{u\le   M\rho}$.
     This     argument
    uses continuity of the replacements,    not   of  $\mathpart{u}$.
    Apply     Proposition~\ref{prop:continuous-power}
    to $v_{{j}}$,    using    its
      local     form with
    growth   bounds    on    $\mathpart{\overline{B_{L+1}}}$.
    By    \eqref{eq:quadratic_growth_bounds_dirichlet_replacements},
  its constants   are computed from
   $\mathpart{m}$,  $\mathpart{M+\eta}$,   $\mathpart{f_{\max}}$,  and $\mathpart{K}$.
  In   particular, $\mathpart{\mu}$   and
  $t_{0}$   from   \eqref{eq:inverse_trace_threshold_and_eigenvalue_floor}
   and \eqref{eq:admissible_averaging_radius} do not change,
     and the    exponent    and remainder
   coefficient    are
  \[
       \varepsilon_\eta=\frac{m}{16(M%
     +\eta)+m},\qquad
  C_\eta=\frac{4(M+\eta)^{2}}{m}.
     \]

   For fixed $\mathpart{|z|<\sqrt   L}$  and
    $0<{t}<t_{0}$,    the  averaging   set
    ${z}+\overline{B_{t}}$ lies inside $B_{R}$. On    this
 compact  set,
   \eqref{eq:dirichlet_replacements_decrease_to_solution}      gives
   pointwise     convergence,
   and  the    functions are
     bounded    between $\mathpart{{u}}$ and $v_{1}$.
  The nonnegative kernel of $\mathcal  {G}_t$ has
   mass    one, so bounded
    convergence
    gives
     \[
  \mathcal G_t
    v_{j}(z)\longrightarrow\mathcal       G_tu(z),
    \qquad      v_{j}(z)\longrightarrow
   u(z).
     \]
     Thus the    difference also   converges;    no
     monotonicity   of
    $\mathcal    G_{t}    v_j-v_{j}$
     is   needed. For fixed $\mathpart{\eta}$,  we let
   $\mathpart{j\to\infty}$     in \eqref{eq:green_average_bound_continuous_potential},
  applied    to $v_{j}$
    with upper    growth    constant
  $\mathpart{M+\eta}$.   Letting    $\mathpart{\eta\downarrow0}$
      then
  gives
   \begin{equation}%
    \label{eq:green_average_bound_bedford_taylor_potential}
    \begin{aligned}
       0\le\mathcal
       G_tu(z)-u(z)
      \le{}&\frac{t^{2}}{\mu}
     +\frac{MKt^2}{2m(n%
     +1)}\rho(z)
  +\frac{4M^2}{m}L^{-\varepsilon}\rho(z),\\
        &\varepsilon=\frac{m}{16M+m},\qquad
    |z|<\sqrt   L.
   \end{aligned}
   \end{equation}
  The kernel    in \eqref{eq:dirichlet_green_kernel_representation}
  has the dilation     form
    $K_t(h)=t^{-2n}K_{1}(h/t)$
   with $K_{1}\in     L^1$.
    Continuity    of      dilations in
   $L^{1}$, together    with local    boundedness
 of $\mathpart{u}$, therefore   makes $t\mapsto\mathcal
  {G}_{t}u(z)$    continuous
  for
 $\mathpart{t>0}$.
     Letting $t\uparrow   t_{0}$ in
 \eqref{eq:green_average_bound_bedford_taylor_potential} therefore      proves
  the same estimate at $t=t_{0}$. Since
  $\mathcal G_{t}u-u=\mathcal  {G}_tw-w$,
      this    completes the
   proof       of
  Theorem~\ref{thm:finite-ball}     with
   the  constants in
  \eqref{eq:pseudo_c2_estimate_constants}.
    \section{From the pseudo-estimate to a full real \texorpdfstring{$C^{2}$}{C2} estimate}
    \label{sec:real-c2}

     At    a   fixed finite
     radius, the additive term      in
     \eqref{eq:finite_ball_green_defect_bound} survives   as
      $\mathpart{t\downarrow0}$.
  We   now show that    it
  suffices to   make this   term  small, provided
    the  density is
   positive with   a  prescribed Dini  modulus.
    We    first establish
  a local   perturbation   criterion, then use
    compactness   to obtain the
    initial  quadratic       approximations.
     Throughout this section,    $C^{2}$
     norms use  all real  derivatives
    of    order at
   most    two.
     \subsection{Constant-density perturbations}
%
    %
%
 %     -------------------------------------------------------------------------
%
 %

  %
%
   We first    use Savin's theorem to   control
      a    constant-density       solution
     that    is    close
      to   a quadratic polynomial  with    positive  definite
     complex    Hessian.
      All the    derivatives
    in   the   resulting estimate are real      derivatives.
  For a real   symmetric
    matrix   $\mathpart{N}$, write    $\mathcal    {H}({N})$
   for its      associated complex  Hessian,
  so that    $\mathcal  H(D^{2}_\R v)=D^2_\C v$. If
    $\mathpart{N\ge0}$, then
       \begin{equation}\label{eq:hermitian_projection_positivity_and_trace}
       \mathcal
  H(N)\ge0,
     \qquad     \tr\mathcal
 H(N)=\frac{1}4\tr    N.
   \end{equation}

 \begin{lemma}[Constant-density
     perturbations]\label{c2:savin}
      Fix  $\mathpart{0<\lambda\le\Lambda}$.    There
  are      $\mathpart{\beta\in(0,1)}$,
      $\mathpart{e_*>0}$, and $\mathpart{{C}}$,  depending
   only    on   $\mathpart{{n},\lambda,\Lambda}$,
   with the following property. Suppose  that  $\mathpart{P}$
      is a real     quadratic
       polynomial,
      $H=D^{2}_\C      {P}$ satisfies
   $\mathpart{\lambda\Id\le
     H\le\Lambda\Id}$,      and
   $v\in  C(B_{r})\cap\operatorname{PSH}(B_r)$
        satisfies
     \[
   \det D^2_\C
    v=\det          H,
     \qquad
     \|v-P\|_{L^\infty(B_{r})}\le      er^2,
       \qquad
    0\le e\le e_*.
     \]
   The equation is understood
   in  the Bedford--Taylor      sense. Then
     $\mathpart{v\in C^{2,\beta}(B_{r/2})}$
      and
    \begin{equation}\label{eq:constant_density_quadratic_perturbation_estimate}
      \begin{split}
   &r^{-2}\|v-P\|_{L^%
      \infty(B_{r/2})}
    +r^{-1}\|D_\R(v-P)\|_{L^\infty(B_{r/2})}\\
    &\qquad
   +\|D^2_\R(v-P)\|_{L^\infty(B_{r/2})}
     +r^\beta[D^{2}_\R(v-P)]_{C^\beta(B_{r/2})}
   \le
     Ce.
     \end{split}
     \end{equation}
    \end{lemma}
       \begin{proof}
  Rescale by $\mathpart{z=ry}$ and   divide the potential        by $r^{2}$;
    this   reduces
   the   proof to $\mathpart{r=1}$. Set  $\mathpart{w=v-P}$ and
    $\mathpart{c=\det   {H}}$, and define
     \begin{equation}\label{eq:determinant_root_operator_at_quadratic_polynomial}
   F_{H}(N)=
       \begin{cases}
      \det\bigl(H+\mathcal
 H(N)\bigr)^{1/n}-c^{1/n},
       &H+\mathcal
   H(N)\ge0,\\
      -c^{1/n},&H+\mathcal
    H(N)\not\ge0.
       \end{cases}
  \end{equation}
    The      function $F_{H}$ is continuous
      and degenerate    elliptic on    the
    space    of all  real
    symmetric matrices, is smooth    near   zero, and
    satisfies $F_{{H}}(0)=0$.
    Its    differential at zero  is
       \[
      DF_H(0)[N]=\frac{c^{1/n}}n
    \tr\bigl(H^{-1}\mathcal
     {H}(N)\bigr).
   \]
     The     hypothesis
   $\mathpart{\lambda\Id\le   H\le\Lambda\Id}$   and
   \eqref{eq:hermitian_projection_positivity_and_trace}
    give   uniform  real ellipticity in a   fixed
        neighborhood
   of zero. The size of    this
     neighborhood and   the   bounds  for
    the first   two     derivatives of $F_{H}$
   are  uniform as    $\mathpart{H}$
   ranges over the indicated compact  subset of the
    positive-definite
     cone.

   By   the equivalence   between    continuous    pluripotential
    and      viscosity
    solutions  \cite[Proposition~1.5 and Lemma~4.7(1)]{EGZ},
  the    equation
      for $\mathpart{{v}}$ also   holds
   in    the viscosity sense.  A $C^{2}$ test function
  $\mathpart{{\varphi}}$ for $\mathpart{w=v-P}$ corresponds    to    the
     test function   $\mathpart{P+\varphi}$
    for  $\mathpart{v}$,
     and
     $D^2_\C(P+\varphi)=H+\mathcal
    {H}(D^2_\R\varphi)$.
    Since $\mathpart{v}$    is
    plurisubharmonic,   upper    tests  satisfy
     $H+\mathcal
     {H}(D^2_\R\varphi)\ge0$.
   For   tests      satisfying this    inequality, the
    viscosity   inequalities
  for $\mathpart{v}$ give   those for    $F_H(D^{2}_\R   w)=0$. At lower
  test functions outside the positive
    cone,    the    supersolution
  inequality follows directly   from   the
     second  line    of   \eqref{eq:determinant_root_operator_at_quadratic_polynomial}.
      Thus $F_{H}$    satisfies the
   hypotheses    of Savin's     perturbation
   theorem  \cite[Theorem~1.3]{Savin}:
     global degenerate     ellipticity and smooth   uniform
   ellipticity    near
     the      zero
   solution   suffice.

  Apply that      theorem
 on  interior balls covering $\mathpart{B_{3/4}}$.
   By reducing    $\mathpart{e_*}$,   it gives a uniform
    $\mathpart{C^{2,\beta}}$     bound    there,
    with     $\|D^{2}_\C
     w\|_{L^\infty(B_{3/4})}\le\lambda/2$.
     The   equation   $\det D^{2}_\C v=\det   H$
    in      Lemma~\ref{c2:savin},
   together       with  $\mathpart{v=P+w}$,
   therefore         gives
   \[
     0=\log\det(H+D^{2}_\C
     w)-\log\det  H
    =\int_0^1\tr\bigl((H+sD^2_\C    w)^{-1}D^2_\C
   w\bigr)\,ds.
    \]
 This is a linear uniformly elliptic
     equation    for  $\mathpart{w}$ whose
    coefficients have uniformly    bounded $\mathpart{C^\beta}$
  norms. By    the
   flatness hypothesis
    of    Lemma~\ref{c2:savin},
        $\mathpart{\|w\|_{L^\infty(B_{3/4})}\le
     e}$.
    The interior  linear   Schauder
     estimate therefore      gives
    \[
       \mathpart{\|w\|_{C^{2,\beta}(B_{1/2})}}
    \le
   C\|w\|_{L^\infty(B_{3/4})}\le   Ce.
      \]
     Scaling back      proves
      \eqref{eq:constant_density_quadratic_perturbation_estimate}.
       \end{proof}

     To treat   a variable
   density, we replace  $\mathpart{f}$ on a    ball by
  its value at the center
    and  use Lemma~\ref{lem:root-stability}.
 Suppose       that   $\mathpart{{u}}$   is
      continuous  on $\overline{B_{r}(x)}$,
      plurisubharmonic
   on   $B_{r}(x)$, and  satisfies
     \[
      (\sqrt{-1}\partial%
  \overline{\partial}u)^n=f\omega_0^{n}
   \]
   there in the  Bedford--Taylor sense.   Assume
    that $\mathpart{f}$ is  continuous
   on $\overline{B_{r}(x)}$  and
     $0<c_0\le f\le     c_{1}$.
   Set    $\mathpart{c=f(x)}$,    and    let  $\mathpart{{v}}$
    be   the continuous plurisubharmonic
     solution
  of
    \[
      \det D^2_\C  v=c\quad\hbox{in
       }B_{r}(x),
    \qquad      v=u\quad\hbox{on
   }\partial     B_r(x).
  \]
      The continuous  Dirichlet solution
    exists by the      Bedford--Taylor
   theory   \cite{BTDirichlet}.   If
 $\mathpart{f}$ has  modulus of    continuity
  $\mathpart{\omega}$, then
 Lemma~\ref{lem:root-stability} gives
       \begin{equation}%
     \label{eq:frozen_density_dirichlet_comparison}
    \|u-v\|_{L^\infty(B_r(x))}
       \le
   r^{2}\|f^{1/n}-c^{1/n}\|_{L^\infty(B_r(x))}
    \le   C
     r^2\omega(r),
   \end{equation}
     where one may take       $C=c_{0}^{1/n-1}/n$, the
  Lipschitz     bound for
     $\mathpart{s\mapsto
   s^{1/n}}$   on  $[c_0,c_{1}]$.
  If $\mathpart{P}$ is a strictly    plurisubharmonic
    quadratic    polynomial    satisfying
      $\det
   D^{2}_\C    P=c$ and
    $\|u-P\|_{L^\infty(B_{r}(x))}\le
   er^2$, comparison      with
     $P\pm
    er^{2}$    also    gives
  \begin{equation}%
  \label{eq:frozen_solution_quadratic_flatness}
    \|v-P\|_{L^\infty(B_r(x))}\le
   er^{2}.
   \end{equation}
   These estimates  require no derivative    bounds
    at      the boundary
    of   the comparison
    ball.

     \subsection{Summable quadratic approximations}
%
   %
%
 %    -------------------------------------------------------------------------
%
% The error recurrence has a contracting
      % term and a density-oscillation
  % term. Summing the  latter over nested
%
 %  scales  uses exactly the Dini
%
    %    assumption.
%
%
 
   We next    combine
     \eqref{eq:constant_density_quadratic_perturbation_estimate}  with
    \eqref{eq:frozen_density_dirichlet_comparison}.
    The   Dini tail
    $\mathcal  {D}_\omega$  defined
  in \eqref{eq:dini_modulus_tail} tends  to   zero
     as
 $\mathpart{r\downarrow0}$.

      \begin{lemma}[A  local quadratic
    approximation     criterion]
     \label{c2:local-criterion}
   Let $\mathpart{u\in
  C(B_{3/2})\cap\operatorname{PSH}(B_{3/2})}$      solve
 $\det D^{2}_\C u={f}$    in  the    Bedford--Taylor
  sense,  where
   $0<c_0\le  {f}\le c_{1}$ and $\mathpart{f}$ has
  a  common Dini modulus $\mathpart{\omega}$.
    Fix
        $\mathpart{0<\lambda\le\Lambda}$.
    There   is
   $\mathpart{\delta>0}$,   depending only on
     ${n},c_{0},c_1,\lambda,\Lambda$,
   with the following   property.
    Suppose   that
   $0<r_{0}<1/8$ and   that  for  every
   $\mathpart{x\in\cl{B_{5/4}}}$
      there is    a real quadratic   polynomial $P_{x}$
     such
    that
    \begin{equation}%
       \label{eq:initial_quadratic_approximation_hypotheses}
    \begin{gathered}
    \lambda\Id\le
 D^2_\C     P_x\le\Lambda\Id,
   \qquad       \det    D^2_\C
      P_x=f(x),\\
        \|u-P_{x}\|_{L^\infty(B_{r_0}(x))}\le
    e_0r_0^2,
      \qquad    e_0+\mathcal
  D_\omega(r_{0})\le\delta,
    \end{gathered}
     \end{equation}
      and
   \[
     \sup_{x\in%
    \cl{B_{5/4}}}
 \sum_{\ell=0}^{2}|D_\R^\ell
      P_{x}(x)|\le       J.
     \]
    Then     $u\in
   C^{2}(B_{5/4})$       and
    \begin{equation}\label{eq:interior_c2_bound_from_dini_approximations}
      \|u\|_{C^2(B_{5/4})}
 \le  C\bigl(J+e_{0}+\mathcal
      D_\omega(r_0)\bigr).
    \end{equation}
      Its  real Hessian has     a common modulus of
    continuity determined      by
     $n,c_0,c_{1},\lambda,\Lambda,r_0,e_0,\omega$.
        \end{lemma}
      \begin{proof}
      Fix a centre    $\mathpart{x}$ and translate    it   to   zero. Put
    $P_{0}=P_x$     and
   $\mathpart{c=f(x)}$.       Choose
   the  constants   in Lemma~\ref{c2:savin}   for   the
  larger
    cone
     $\mathpart{\frac\lambda2\Id\le
     H\le2\Lambda\Id}$.
      We    construct quadratic
   polynomials $P_{k}$ on    the balls
    $B_{r_{{k}}}$,   where
 $r_{k}=\theta^kr_{0}$ and  $\mathpart{\theta\in(0,1/4)}$
  will   be fixed   below. At
    each    step we require
   \[
     \|u-P_k\|_{L^\infty(B_{r_k})}\le
      e_{k}r_k^2,
  \qquad   \det    D^2_\C
    P_k=c.
       \]

     At the $\mathpart{k}$th      step,
   suppose  that
    \[
       e_{k}\le
   e_*,\qquad
     \frac{\lambda}{2}\Id\le     D^2_\C
   P_k\le2\Lambda\Id.
       \]
  We    verify    below   that these conditions
   hold   at  every step.
 Let $v_{k}$  be the    continuous    solution
 on $B_{r_{k}}$    with
     constant density    $\mathpart{c}$   and
 boundary  values   $\mathpart{u}$.
  By \eqref{eq:frozen_solution_quadratic_flatness}
 and    Lemma~\ref{c2:savin},  we   have
     \[
     v_k\in
     C^{2,\beta}(B_{r_{k}/2}).
   \]
  Together   with    \eqref{eq:frozen_density_dirichlet_comparison},
    these  estimates       give
     \[
      \|u-v_{k}\|_{L^\infty(B_{r_k})}\le
      Cr_k^2\omega(r_k),
   \qquad
     r_k^\beta[D^2_\R(v_k-P_k)]_{C^\beta(B_{r_k/2})}
     \le
      Ce_k.
       \]
   We may    therefore define    $\mathpart{P_{k+1}}$ as
    the   real quadratic    Taylor
   polynomial   of $v_{k}$
     at     zero.    Taylor's formula, the  preceding
     comparison bound, and     the full estimate
     \eqref{eq:constant_density_quadratic_perturbation_estimate}
   give   on   $B_{\theta
  r_{k}}$
        \begin{equation}\label{eq:quadratic_approximation_error_and_coefficient_increments}
   \begin{split}
       (\theta
     r_k)^{-2}
   \|u-P_{k+1}\|_{L^\infty(B_{\theta
    r_k})}
      &\le   C\theta^\beta
    e_k+C\theta^{-2}\omega(r_{k}),\\
   |D_\R^\ell(P_{k+1}%
       -P_k)(0)|
  &\le Ce_kr_k^{2-\ell},\qquad
     \ell=0,1,2.
    \end{split}
    \end{equation}
   Since    $v_{k}$   solves $\det
  D^{2}_\C v_{k}=c$ and $\mathpart{P_{k+1}}$
  has    the   same   Hessian
     as $v_{k}$ at      zero,
    $\det
    D^{2}_\C       P_{k+1}=c$.

  Choose    $\mathpart{\theta}$    so
     that  $\mathpart{C\theta^\beta\le1/2}$, and  define
% The first term halves the old error and the second measures density oscillation.
   \begin{equation}%
      \label{eq:dini_iteration_error_recurrence}
    e_{k+1}=\frac12e_{k}+C\theta^{-2}\omega(\theta^{k}r_0).
       \end{equation}
    Summing    this identity
  and using    monotonicity    of $\mathpart{\omega}$ gives
   \begin{equation}\label{eq:dini_iteration_error_sum}
  \begin{split}
        \sum_{k=0}^\infty
   e_k
    &\le
        C\left(e_0+\sum_{k=0}^\infty\omega(\theta^kr_0)\right)\\
   &\le  C\bigl(e_{0}+\mathcal
  D_\omega(r_0)\bigr).
       \end{split}
        \end{equation}
    Indeed,     for
   $\mathpart{k\ge1}$,
     \[
    |\log\theta|\,%
       \omega(\theta^kr_0)
    \le\int_{\theta^kr_0}^{\theta^{k-1}r_{0}}
    \frac{\omega(s)}s\,ds.
     \]
 By  the initial smallness condition
     \eqref{eq:initial_quadratic_approximation_hypotheses},
     we    can choose    $\mathpart{\delta}$
    so    that   \eqref{eq:dini_iteration_error_sum}
     keeps  every $e_{k}$ below $\mathpart{e_*}$ and the
 total     change of   the
   complex Hessians     in
    \eqref{eq:quadratic_approximation_error_and_coefficient_increments}      below
       $\mathpart{\min\{\lambda/2,\Lambda\}}$.
      Induction  then  keeps    all
    reference   Hessians in   the larger cone,
  and   every step    above
    is
   justified.
  
 The    three coefficient
    estimates   in \eqref{eq:quadratic_approximation_error_and_coefficient_increments}
    show       that $P_{{k}}$    converges
    coefficientwise to    a polynomial
   $\mathpart{P_{\infty,x}}$.
  On   $B_{r_{k}}(x)$, combine  the bound
    $\|u-P_k\|_{L^\infty(B_{r_k}(x))}\le
   e_kr_k^{2}$
 with the    constant,
     linear, and quadratic   coefficient     tails.
     Using  \eqref{eq:dini_iteration_error_sum}
    gives
     \begin{equation}\label{eq:uniform_quadratic_expansion_remainder}
   r_{k}^{-2}\|u-P_{\infty,x}\|_{L^\infty(B_{r_k}(x))}
     \le    C\sum_{j=k}^\infty
    e_j\longrightarrow0.
    \end{equation}
      The sequence on  the right
      is independent of   $\mathpart{x}$. Passing from
     the    radii  $r_{k}$
  to  an  arbitrary smaller    radius changes the
  bound   by   at
   most $\mathpart{\theta^{-2}}$.  Thus there  is    a common
    function
   $\mathpart{\eta(r)\to0}$ such  that
   \[
  \|u-P_{\infty,x}\|_{L^\infty(B_{r}(x))}
    \le%
    \eta(r)r^2
      \quad
   (x\in\cl{B_{5/4}},\ 0<r\le   r_0).
    \]

   For completeness, this uniform   expansion
   gives     classical
 second   derivatives,    not   only pointwise
   quadratic    expansions.
    Write    the       limiting
     polynomial  as
    \[
   P_{\infty,x}(z)=a_x+b_x\cdot(z-x)
   +\tfrac{1}2(z-x)^TM_x(z-x).
     \]
     The
    estimate
  $\|u-P_{\infty,x}\|_{L^\infty(B_{r}(x))}\le\eta(r)r^2$,
  with    $\mathpart{\eta(r)\to0}$, gives  $a_{x}=u(x)$ and
  $D_\R   u(x)=b_{x}$.
   For      $0<d=|x-y|<r_{0}/4$,
   we  compare    the   two    limiting
   polynomials on the   ball    of radius  $\mathpart{d/2}$
   about    $\mathpart{(x+y)/2}$.
     Both approximation errors there
     are    bounded  by
      $4\eta(2d)d^{2}$. After
   rescaling    this ball to unit   radius,
   the supremum  norm  of a  quadratic
  polynomial  controls all
     its    coefficients.  Applying
    this  finite-dimensional   norm
 estimate to
    the difference of    the two    polynomials  gives
    \[
    |M_x-M_y|\le
    C\eta(2d),
    \qquad
      |b_y-b_x-M_{x}(y-x)|\le
       C\eta(2d)d.
       \]
   Since    $D_\R
   u(q)=b_{q}$  at      every    center $\mathpart{{q}}$, these two
   inequalities    show   that $\mathpart{D_\R    u}$ is differentiable
      with  derivative
   $M_{x}$,    and
 that $D^{2}_\R    u=M_x$ is  continuous.
   Finally,  summing    the three
     jet estimates   in
    \eqref{eq:quadratic_approximation_error_and_coefficient_increments},
    using \eqref{eq:dini_iteration_error_sum}, ${r}_{0}<1$,
   and   the    initial   jet    bound
   $\mathpart{{J}}$,    proves \eqref{eq:interior_c2_bound_from_dini_approximations}.
  \end{proof}
   \subsection{A bounded trace and a Dini density give full regularity}
%
      %
     %   -------------------------------------------------------------------------
%
 %

    We    record the   regularity fact      used   both
  in the   compactness argument
     and   when    passing
     from  continuous approximations   to a weak solution.

    \begin{lemma}%
    \label{c2:elliptic-dini}
  Let $\mathpart{u}$ be a locally bounded    plurisubharmonic
   function on  an open
    set  $\mathpart{\Omega}$,    with $(\sqrt{-1}\partial\overline{\partial}u)^n=f\omega_{0}^n$ in
 the Bedford--Taylor
   sense. Suppose    $\mathpart{f}$ is positive  and  locally
   Dini continuous,    and
    $\mathpart{\Delta_\R     u}$    is
 locally   bounded   in distributions.  Then
       $u\in
   C^{2}(\Omega)$.
 \end{lemma}

   \begin{proof}
   Work on   a relatively compact   ball
  on which $f\ge c_{0}>0$ and
     $\tr     D^{2}_\C
   u\le  {H}$ in distributions, with  $\mathpart{{H}>0}$.
    Poisson regularity   gives    $\mathpart{u\in
     W^{2,\psi}_{\mathrm{loc}}}$ for   every
  finite    $\mathpart{\psi>1}$,  and
  hence a   continuous representative.
 This    representative agrees    everywhere with
  the         original
  plurisubharmonic    function, since
   its   ball averages converge to    its
  pointwise values.   Positivity of  the
  complex    Hessian as a  current
   and    its
 bounded    trace give
  $0\le D^{2}_\C
   u\le H\Id$    almost   everywhere.

 To identify the determinant, mollify   on
   a smaller   ball.
 Since   $\mathpart{u\in
    W^{2,\psi}_{\mathrm{loc}}}$ for     $\mathpart{\psi>n}$,
     the   mollified  Hessians
      converge strongly in $\mathpart{L^\psi}$,
   so their determinants converge  in
   $\mathpart{L^{\psi/n}}$.
   The        potentials
  converge locally uniformly.    Bedford--Taylor
   continuity
     and $(\sqrt{-1}\partial\overline{\partial}u)^n=f\omega_0^{n}$ therefore     give
      $\det     D^{2}_\C
    {u}=f$      almost everywhere.    Combining this  with
     $f\ge     c_{0}$
   and $0\le  D^{2}_\C u\le H\Id$ gives
      \begin{equation}\label{eq:complex_hessian_eigenvalue_bounds_from_trace}
        \kappa\Id\le
  D^{2}_\C u\le    H\Id,
    \qquad
  \kappa=\frac{c_0}{H^{n-1}}.
  \end{equation}

    We construct  a smooth       uniformly elliptic
   operator that    agrees with
   $\mathpart{\log\det}$  on   the Hessian
  range in \eqref{eq:complex_hessian_eigenvalue_bounds_from_trace}.
 The resulting equation     will
    satisfy the   hypotheses    of       the
    interior      Dini
   estimate.

 The nonempty Hessian range implies
   $\mathpart{0<\kappa\le        H}$.
 Choose  $\chi\in
 C_{c}^\infty((\kappa/4,4H))$ with
    $\mathpart{0\le\chi\le1}$ and    $\mathpart{\chi=1}$
   on  $\mathpart{[\kappa/2,2H]}$.
      Extend $b(s)=\chi(s)/s^{2}$ by   zero
   to $\mathpart{s\le0}$,      and  define
     a   smooth
   function   $\mathpart{{\phi}}$   on       $\mathpart{\R}$  by
    \[
 \phi'(t)=\frac{1}H-\int_H^t
     b(s)\,ds,
       \qquad
   \phi(H)=\log    H.
        \]
  For $\mathpart{t\ge   H}$,
   the integral   is at most
     $\int_{H}^{4H}s^{%
     -2}\,ds=3/(4H)$.
    For   $\mathpart{t\le H}$,   the   reversed
   integral is   at    most
  $\int_{\kappa/4}^{H}
      s^{-2}\,ds=4/\kappa-1/H$.
   Consequently,
  \[
      \frac1{4H}\le\phi'(t)%
     \le\frac{4}\kappa,
       \qquad
  \phi''(t)=-{b}(t)\le0,
    \qquad
  |\phi''(t)|\le\frac{16}{\kappa^2}
    \quad(t\in%
  \R).
      \]
    On $\mathpart{[\kappa/2,2{H}]}$,     the      integration segment
  joining $\mathpart{{H}}$      and
 $\mathpart{t}$ lies    where
  $\mathpart{\chi=1}$.   Thus $\mathpart{\phi'(t)=1/{t}}$ there,
     and the  normalization    at $\mathpart{H}$  gives
    $\mathpart{\phi(t)=\log      t}$.

  Use the projection $\mathcal
     {{H}}$ from \eqref{eq:hermitian_projection_positivity_and_trace}   and    set
   \begin{equation}\label{eq:concave_elliptic_extension_of_log_determinant}
     F(N)=\operatorname{tr}\phi\bigl(\mathcal
     {H}(N)\bigr),
   \qquad
   N\in\operatorname{Sym}(2n).
    \end{equation}
  The  function $\mathpart{\phi}$ is applied to    the eigenvalues
   of          the Hermitian
  matrix $\mathcal  {H}({N})$.
   Smooth  spectral calculus makes $\mathpart{{F}}$ smooth.
   To check    concavity and  the derivative
    bounds needed  below,  let
     $\mathpart{A}$  be diagonal    with eigenvalues
  $\lambda_{i}$, and   let $\mathpart{Q}$     be
    Hermitian    in    the same orthonormal
   basis.    Then
     \[
     \begin{split}
     D^2(\operatorname{tr}\phi)(A)[Q,Q]
       ={}&\sum_i\phi''(\lambda_i)|Q_{i\bar
       i}|^2\\
  &+2%
   \sum_{i<j}
       \frac{\phi'(\lambda_{i})%
   -\phi'(\lambda_j)}{\lambda_i-\lambda_j}
  |Q_{i\bar
    j}|^2.
      \end{split}
    \]
    At    repeated    eigenvalues,
 the quotient  is interpreted    as
 $\phi''(\lambda_{{i}})$. Since   $\mathpart{\phi''\le0}$
    and   $|\phi''|\le16/\kappa^{2}$,
   the    mean value theorem    shows    that every
   coefficient    is
    nonpositive and has  absolute value at  most
    $32/{\kappa}^{2}$.
 Thus $\mathpart{\operatorname{tr}\phi}$, and hence
    $\mathpart{F}$, is concave and
    has globally       Lipschitz first
      derivatives.
    For   a real positive      semidefinite
   matrix $\mathpart{P}$,   differentiation gives
    \[
   DF(N)[P]
     =\operatorname{tr}\!\left[
     \phi'\bigl(\mathcal
   {H}(N)\bigr)\mathcal  H(P)\right].
     \]
    The   bounds $\mathpart{1/(4H)\le\phi'\le4/\kappa}$
  and the    positivity
   and trace  identities
  in    \eqref{eq:hermitian_projection_positivity_and_trace} therefore imply
       \[
  \frac{1}{16H}\operatorname{tr}_{\R}P
   \le
     DF(N)[P]
    \le\frac1\kappa\operatorname{tr}_{\R}P.
    \]
   Integrating     along    $\mathpart{N+sP}$ gives the corresponding
     finite-increment
    inequalities,
     so  $\mathpart{F}$ is     globally uniformly elliptic on    real
  symmetric
   matrices.

   By \eqref{eq:complex_hessian_eigenvalue_bounds_from_trace}, the complex
 Hessian    of $\mathpart{u}$  lies
  where     $\mathpart{\phi=\log}$.
     Together   with
  $\det    D^{2}_\C    u=f$ almost  everywhere,
     this    gives
   \[
  {F}(D^{2}_{\R}u)
     =\log\det
    D^2_{\C}u
   =\log
        f
   \quad\text{almost
  everywhere}.
    \]
   Since   $\mathpart{u\in  W^{2,\psi}_{\mathrm{loc}}}$
     for      $\mathpart{\psi>2n}$,    the
   strong-to-viscosity implication
  for uniformly elliptic equations
 makes this  a   viscosity equation; see
     \cite{CCKS}.
  On    the      ball under
    consideration,     $f\ge c_{0}$ gives
     \[
   |\log
  f(x)-\log f(y)|
      \le
    c_{0}^{-1}|f(x)-{f}(y)|.
     \]
   Since $\mathpart{{f}}$ is     Dini
 continuous, the   preceding inequality
    makes    $\mathpart{\log   f}$ Dini
  continuous. The operator $\mathpart{F}$ is concave, globally uniformly   elliptic,   and
     has     globally Lipschitz
 first derivatives. For every    real symmetric
   matrix  $\mathpart{{M}}$,
  the shifted operator $\mathpart{N\mapsto  F(N+M)-F(M)}$
     has the same  concavity and
    ellipticity   constants. The Evans--Krylov
    estimate therefore    verifies     the
    uniform      homogeneous estimate
     required      in \cite[Section~3.1]{XJWang}.
   The interior Dini estimate,
  including  its viscosity     version,
   then    gives  $u\in    C^{2}$
  on  every smaller    ball
        \cite[Theorem~3.1 and Section~3.3(ii)]{XJWang}.
    The ellipticity   and  derivative   bounds
     above depend   only   on
    $\mathpart{n,\kappa,H}$. Since     the
    argument  applies on every   relatively
  compact    ball, it  proves
  the lemma.
   \end{proof}
     \subsection{A small averaging remainder gives a real Hessian bound}
%
 %
     %   -------------------------------------------------------------------------
%
 %    Compactness gives a limit   potential with a
%
   % trace    bound. Transfer a
%
  % quadratic   approximation
%
 %  at one fixed scale from that limit   to the
%
      %    approximating
%
%  potentials.
%
  %

    We prove  the   existence of a
   smallness threshold  by letting   the
 averaging remainder   tend to zero   in
   a   compactness argument.
      Lemma~\ref{c2:elliptic-dini}
   regularizes  the   limit, and
    Lemma~\ref{c2:local-criterion} transfers its
  quadratic   approximations
  to the   approximating
   solutions.

    \begin{theorem}[From   a    pseudo-estimate    to a
   $C^{2}$ estimate]
     \label{thm:c2-transfer}
    Fix $\mathpart{{n}\ge2}$,    $K_{0},A>0$,
  $0<c_{0}\le c_1$, $\mathpart{0<{t}_*\le1/4}$, and  a
    Dini modulus $\mathpart{{\omega}}$. There exist
    $\delta_{0}>0$     and $\mathpart{C<\infty}$,
   depending  only on   these
  data,  such that    the following holds.
    Let    $u\in  C(B_{2})$  be
    plurisubharmonic  and  satisfy
   \[
   (\sqrt{-1}\partial\overline{\partial}%
     u)^{n}=f\omega_0^n,\qquad
  \|u\|_{L^\infty(B_2)}\le   K_0,\qquad
    c_{0}\le    f\le c_1,
     \]
  with $\mathpart{|f(x)-f(y)|\le\omega(|x-y|)}$   for
    $\mathpart{|{x}-y|\le1}$.
    If
     \begin{equation}%
       \label{eq:green_average_small_remainder_hypothesis}
  0\le\mathcal     G_tu(z)-u(z)\le
      At^{2}+\delta
      \quad(z\in   B_{3/2},\ 0<t<t_*),\qquad
   0\le\delta\le\delta_0,
    \end{equation}
     then $u\in    C^{2}(B_1)$
  and $\|u\|_{C^2(B_{1})}\le    {C}$,
  where   the norm uses all real
   derivatives of       order at most    two.
     \end{theorem}

       \begin{proof}
     If no such    threshold  and bound   exist, choose
    solutions     $u_{j},f_j$
 with the    fixed data and $\delta_{j}\to0$,
   such   that    either
     $u_j\notin  C^{2}(B_1)$    or
      $\|u_{j}\|_{C^2(B_{1})}>j$.
    For       fixed $\mathpart{t}$, convolution with    the $L^{1}$ kernel of
  $\mathcal    G_{t}$
   sends      the uniformly bounded    $u_{{j}}$ to an
    equicontinuous    family.
    Estimate \eqref{eq:green_average_small_remainder_hypothesis}       shows
    that   $u_{j}$ differs from
 this family  by at most
      $At^2+\delta_{j}$.
    Choosing  $\mathpart{{t}}$      small
   and then    $\mathpart{j}$     large proves local uniform
    precompactness  on
     $\mathpart{B_{3/2}}$. The common modulus    gives the  same
   compactness   for $f_{j}$. After passing
   to a     subsequence,
      \begin{equation}%
    \label{eq:local_uniform_convergence_of_potentials_and_densities}
      u_j\longrightarrow
  u_\infty,\qquad
      f_{j}\longrightarrow
 f_\infty
       \quad\text{locally
 uniformly      on    }B_{3/2}.
       \end{equation}
    Bedford--Taylor
  continuity   gives
    $(\sqrt{-1}\partial\overline{\partial}%
     u_\infty)^n=f_\infty\omega_0^{n}$.

   By    \eqref{eq:green_average_small_remainder_hypothesis},
     \eqref{eq:local_uniform_convergence_of_potentials_and_densities},
    and     $\delta_{j}\to0$,
    the       limit       satisfies
      $0\le\mathcal
   G_{t}u_\infty-u_\infty\le   At^2$
  on each compact subset   for
  sufficiently  small $\mathpart{{t}}$.
   For a       nonnegative
    smooth compactly   supported test function $\mathpart{\varphi}$,
 symmetry of      the  averaging
      kernel    yields
      \[
   \int   u_\infty\frac{\mathcal
  G_t\varphi-\varphi}{t^2}\,d\xi
 =\int\varphi\frac{\mathcal
    {G}_tu_\infty-u_\infty}{t^2}\,d\xi
     \le
   A\int\varphi\,d\xi.
   \]
  The expansion \eqref{eq:green_average_laplacian_expansion}
    therefore   gives
%
% Testing against smooth functions transfers both inequalities to the distributional Laplacian.
       \begin{equation}\label{eq:distributional_laplacian_bound_for_limit}
      \mathpart{0}\le\Delta_\R
       u_\infty\le8(n+1)A
      \quad\text{in distributions  on
      }B_{3/2}.
       \end{equation}
       Lemma~\ref{c2:elliptic-dini}
  shows     that $u_\infty\in C^{2}(B_{3/2})$.
      Its complex Hessian lies    in a  fixed
   positive  cone  by
    \eqref{eq:complex_hessian_eigenvalue_bounds_from_trace},
   using     $\mathpart{H=1+2(n+1){A}}$.

   Let   $\mathpart{E=\overline{B_{5/4}}}$, and  let $P_{x}$ be
    the    quadratic     Taylor
      polynomial    of $\mathpart{u_\infty}$
  at   $\mathpart{x\in    E}$.
      Continuity  of   its
 full  real Hessian    implies
   \begin{equation}\label{eq:uniform_taylor_approximation_of_limit}
     \sup_{x\in
        E}r^{-2}
      \|u_\infty-P_x%
     \|_{L^\infty(B_{r}(x))}\longrightarrow0
    \quad(r%
    \downarrow0).
  \end{equation}
     Choose       $\mathpart{0<\lambda\le\Lambda}$
   so   that
      \[
   2\lambda\Id\le   D^2_\C
     P_{x}\le\frac{\Lambda}{2}\Id
    \qquad(x\in
     E).
   \]
   These  margins leave  room
   for the     small corrections introduced below.
   We take  the smallness
   threshold    from   Lemma~\ref{c2:local-criterion}
     with parameters  $\mathpart{\lambda}$
     and     $\mathpart{\Lambda}$.
     We   then choose one  ${r}_{0}<1/8$    so    that
   the error   in
     \eqref{eq:uniform_taylor_approximation_of_limit}     and
     $\mathcal       {D}_\omega(r_0)$
      are  small  enough    for that
 threshold. We must    also  adjust   each
     initial  polynomial so that    its    complex
  Hessian  has determinant $f_{j}(x)$.
   The   identity
    $\det D^2_\C  P_{x}=f_\infty(x)$, uniform
   positivity of $D^2_\C   P_{{x}}$,  and
      \eqref{eq:local_uniform_convergence_of_potentials_and_densities}
  give numbers $\mathpart{a_{j,x}\to0}$,   uniformly
 for $\mathpart{{x}\in   E}$, satisfying
      \[
   P_{j,x}=P_x+a_{j,x}|z-x|^2,\qquad
     \det    D^{2}_\C
  P_{j,x}=f_j(x).
        \]
 Indeed, the determinant
  of $D^{2}_\C     P_x+a\Id$   is  strictly
      increasing      in
    $\mathpart{{a}}$  near zero, with a  uniform  positive derivative.
   At the fixed  radius
    $r_{0}$,
    \begin{equation}%
       \label{eq:transfer_of_initial_quadratic_approximation}
     \begin{split}
    r_0^{-2}\|u_j-P_{j,x}\|_{L^\infty(B_{r_0}(x))}
  \le{}&r_0^{-2}\|u_j-u_\infty\|_{L^%
      \infty(B_{11/8})}\\
      &+r_0^{-2}%
     \|u_\infty-P_x\|_{L^\infty(B_{r_{0}}(x))}
   +|a_{j,x}|.
     \end{split}
       \end{equation}
    By     \eqref{eq:local_uniform_convergence_of_potentials_and_densities},
    the      choice of   ${r}_{0}$ using
 \eqref{eq:uniform_taylor_approximation_of_limit},  and
   $\mathpart{a_{j,x}\to0}$    uniformly,
    the  right side is uniformly   small for  all $\mathpart{x\in E}$  when
   $\mathpart{j}$ is large.
  Since    $u_\infty\in    C^{2}$    near
   the     compact  set   $\mathpart{E}$ and
       $\mathpart{a_{j,x}\to0}$
     uniformly, the real    coefficients  of   $\mathpart{P_{j,x}}$
    are    uniformly
   bounded. The       margins
    $2\lambda\Id\le
   D^{2}_\C   P_x\le(\Lambda/2)\Id$ then   give
       \[
    \lambda\Id\le    D^{2}_\C
  P_{j,x}\le\Lambda\Id.
       \]
  The    radius  ${r}_{0}$ and the coefficient
 bound may depend      on   the    chosen
  limit $\mathpart{{u}_\infty}$, but   they
  are fixed  along   this  subsequence.
 Lemma~\ref{c2:local-criterion}  therefore    gives
   $C^{2}$    regularity
 and a    full real Hessian bound independent of $\mathpart{{j}}$
    on    $\mathpart{B_{5/4}}$.
   Interior    interpolation,      together
     with   $\|u_j\|_\infty\le    K_{0}$,
    bounds       their
     gradients   on $B_{1}$. This contradicts the choice
     of   the     sequence and  proves
  the  theorem.
    \end{proof}
   \subsection{Apply the criterion to the power remainder}
 %      -------------------------------------------------------------------------
%
        %

   %
%
      \begin{proof}[Proof
   of   Theorem~\ref{thm:real-c2}]
    Denote the  potential    in the statement
      by   $\mathpart{w}$ and   put   $\mathpart{u={w}-q+m}$.
    We apply  Theorem~\ref{thm:c2-transfer}
     to   the    continuous   exact-density
   replacements    constructed in   Section~\ref{sec:approximation}. On    $\mathpart{B_{2{L}}}$  we have the
    bound      \eqref{eq:normalized_two_sided_growth_on_ball}
     with   $\mathpart{M}$   from
    \eqref{eq:quadratic_height_bound_constants}.
    We
       set
     \[
  \mu=\frac{m^{n-1}%
    }{c_{1}(4n)^{n-1}},\qquad
       \gamma=\frac{m}{32M+m},%
   \qquad
   t_*
    =\min\left\{\frac14,\sqrt{\frac{{m}\mu}{2}}\right\}.
      \]
 The  fixed data for   the transfer
   theorem  will  be
      \[
      K_{0}=10M,%
   \qquad
  A_0=\frac1\mu+%
    \frac{13MK}{4m(n+1)},\qquad
   c_0,c_1,%
    \omega,t_*.
  \]
   Let    ${\delta}_{0}$ be    its smallness threshold  for
     these   data, and    choose
     \begin{equation}\label{eq:outer_radius_threshold_for_interior_c2_bound}
     L_0=\max\left\{16,L_{\mathrm
       h},
     \left(\frac{52M^{2}}{m\delta_0}\right)^{1/\gamma}\right\}.
       \end{equation}

     For   $\mathpart{R=3L/2}$,
  Section~\ref{sec:approximation} provides    continuous
  plurisubharmonic functions  $v_{j}$ on
    $\overline{B_{R}}$  with
  \[
       (\sqrt{-1}\partial\overline{\partial}v_j)^{n}=f\,\omega_0^n\quad\text{in
  }B_R,
     \qquad   v_j\downarrow
    u\quad\text{pointwise     in }B_R.
     \]
 Taking $\mathpart{\eta=M}$ in
    \eqref{eq:quadratic_growth_bounds_dirichlet_replacements}   gives
   \[
  m\rho\le      v_{j}\le2M\rho\quad\text{on
     }\overline{B_{L+1}}
   \]
  for all  sufficiently
      large    $\mathpart{j}$.    Proposition~\ref{prop:continuous-power}
   therefore applies    with
  upper growth constant $\mathpart{2M}$.   Since
  $\mathpart{{\rho}\le13/4}$   on
   $\mathpart{B_{3/2}}$, it      gives
      \[
 0\le\mathcal
   G_t     v_j-v_j
       \le
       A_0t^2+\frac{52M^{2}}{m}L^{-\gamma}
   \quad\left(z\in
     B_{3/2},\quad0<t<t_*\right).
      \]
 The growth bound
    \eqref{eq:quadratic_growth_bounds_dirichlet_replacements},  with  $\mathpart{\eta=M}$,
  also  gives $\|v_j\|_{L^\infty(B_{2})}\le10M$.
    The choice
 \eqref{eq:outer_radius_threshold_for_interior_c2_bound}      makes
     the   additive remainder at   most
      ${\delta}_{0}$,
    so   Theorem~\ref{thm:c2-transfer} yields
      \[
    \|v_j\|_{C^2(B_{1})}\le
   C
    \]
  with $\mathpart{C}$ independent    of
 $\mathpart{j}$    and $\mathpart{L}$.

     These bounds      and the pointwise
 convergence    imply   local $C^{1}$
   convergence after passage to a        subsequence. In
   particular,
    $u\in    C^{1,1}_{\mathrm{loc}}(B_{1})$,    and
     its    weak    real Hessian
   has   the   same
    uniform  bound. Its complex   Hessian   is positive
      semidefinite with bounded    trace, while
    its   Bedford--Taylor   density
   is  the positive Dini-continuous
   function $\mathpart{f}$.  Lemma~\ref{c2:elliptic-dini}
    therefore       gives    $u\in  C^2(B_{1})$.
  The classical    derivatives    agree
   with the weak derivatives, so their
   uniform   bounds      pass    to    $\mathpart{u}$.
  This     gives
    $\|u\|_{C^{2}(B_1)}\le   C$.    Since
   $\mathpart{w=u+q-m}$ and $\|q-m\|_{C^{2}(B_1)}\le C_{n}\delta$,
  the         asserted
    bound
  follows    for    $\mathpart{w}$.
    \end{proof}

  \begin{proof}[Proof    of
    Corollary~\ref{cor:interior-smoothness}]
    Theorem~\ref{thm:real-c2} gives
     $u\in     C^{2}(B_1)$.
    Set $A=D^{2}_\C u$.
   Since $u\in     C^{2}$   and    $\mathpart{f}$ is    continuous,
   the equality      of Monge--Amp\`ere measures
    gives $\mathpart{\det    A=f}$
     pointwise.
   Plurisubharmonicity  gives     $\mathpart{A\ge0}$, and
 $\det  A=f\ge c_{0}>0$ makes   $\mathpart{A}$ positive
    definite.
    The  $C^{2}$    bound  from
      Theorem~\ref{thm:real-c2}  bounds its
   eigenvalues from
     above; their   product $f\ge      c_{0}$ then   bounds
   each  eigenvalue from     below.
 Since   $\mathpart{f}$   is smooth    and positive,
 Wang's interior
 estimate \cite[Theorem~1.1]{Wang}, applied  on
     smaller   balls,
  gives
     $u\in  C^{2,\alpha}_{\mathrm{loc}}(B_{1})$ for   some
       $\mathpart{\alpha\in(0,1)}$.

     To obtain  one     more
   derivative,    fix a   real  coordinate     direction
   $\mathpart{e}$   and    set $v_{h}(z)=(u(z+he)-u(z))/h$ for
     small    $\mathpart{h\ne0}$.
    The     fundamental theorem   of
 calculus applied     to    $\mathpart{\log\det}$
   along the   positive-definite   segment    joining  $\mathpart{{A}(z)}$ and
     $\mathpart{A({z}+he)}$          gives
     \begin{equation}%
   \label{eq:linearized_equation_for_difference_quotients}
       \begin{aligned}
    \tr\bigl(a_h(z)D^2_\C
        v_h(z)\bigr)
      &=\frac{\log
   f(z+he)-\log   f(z)}h,\\
   a_{h}(z)&=\int_0^1\bigl((1-t)A(z)+tA(z+he)\bigr)^{-1}\,dt.
  \end{aligned}
     \end{equation}
  Theorem~\ref{thm:real-c2}    and $\det
    A=f\ge  c_{0}$ bound
   the eigenvalues  of   $\mathpart{A}$  above and below.
   Together with
    $A\in   C^\alpha_{\mathrm{loc}}(B_{1})$,    these
   bounds    make   $a_{h}$
   uniformly elliptic with locally bounded
   $\mathpart{C^\alpha}$      norm,
     independently
   of small $\mathpart{{h}}$.
     The    smoothness and   positivity of
   $\mathpart{f}$ give the same local
   $\mathpart{C^{\alpha}}$  bound
  for  the  right-hand side of
   \eqref{eq:linearized_equation_for_difference_quotients},     and
    $u\in    C^{1}$ bounds $v_{h}$.
 Interior Schauder estimates
      \cite[Chapter~6]{GT}     therefore   bound
  $v_{h}$ in   $\mathpart{C^{2,\alpha}}$ on
     smaller     compact subsets.
   Passing  to   $\mathpart{{h}\to0}$
      gives $\partial_eu\in C^{2,\alpha}_{\mathrm{loc}}(B_{1})$.
 The    coefficients $\mathpart{A^{-1}}$  now
 belong   to $\mathpart{C^{1,\alpha}}$,
  and $\mathpart{\log f}$  is   smooth. Differentiating
   the    resulting  linear
   equations      and  successively applying
   interior      Schauder    estimates
    \cite[Chapter~6]{GT}   gives derivatives of every order. Thus
    $u\in       C^\infty(B_{1})$.
     \end{proof}
      \section{Proof of the Liouville theorem}\label{sec:regularity}
%
%
    %
%
  %   =========================================================================
%
   % Section 7 | Proof
 %  of the Liouville theorem
     % Keep the averaging  radius fixed while the outer ball
%
 %   expands. After    the
%
 % boundary  contribution vanishes, the small-radius limit
%
 % yields      the trace
%
 %    estimate
%
   %     in    distributions.
%
 % Equation index (printed number and
   %     source    label)
%
 %
     %
%
    %      7.1    eq:global_green_average_bound_constant_density
%
 %   7.2    eq:green_average_self_adjointness_on_test_functions
 %      7.3    eq:global_laplacian_bound_test_function_form
%
      %   7.4          eq:global_distributional_laplacian_bound
%
 %
   %
 %     7.5    eq:local_second_order_sobolev_regularity
%
%     7.6    eq:quadratic_blowdown_rescaling
%
     %    7.7       eq:uniform_real_hessian_holder_bound_on_blowdowns
%
  %     =========================================================================

   The    averaging estimate
   in  Theorem~\ref{thm:constant-density}
     reduces  Theorem~\ref{thm:main}   to
   standard local regularity
     and   one rescaling    argument.
   We    first
    remove the   boundary remainder on expanding
  balls. The   resulting
  averaging bound gives      a
    bounded distributional Laplacian, and   an
    interior H\"older        estimate then
    forces    the full real Hessian     to be
     constant.

   \subsection{From the averaging bound to a bounded Laplacian}
    \label{sec:weak-limit}
  %
%
 % -------------------------------------------------------------------------

   Let $\mathpart{{u}}$      satisfy   the
      hypotheses    of   Theorem~\ref{thm:main}.
    Take  $\mathpart{H=u(0)}$ and $\mathpart{M}$  from  \eqref{eq:quadratic_height_bound_constants} with
     $\mathpart{F=1}$. We apply    the   averaging estimate  directly
    to $\mathpart{u}$,    which  satisfies
     \eqref{eq:real_coordinate_quadratic_lower_bound}.
     Fix  $z\in\C^{n}$    and    $0<t\le
   t_{0}$.     Apply
  \eqref{eq:unit_density_green_defect_bound}   on
    the working ball and     let
    $\mathpart{L\to\infty}$. The constants    are
      independent       of $\mathpart{L}$,    so
      $\mathpart{\varepsilon>0}$
      gives
    \begin{equation}%
       \label{eq:global_green_average_bound_constant_density}
  0\le\mathcal
   G_tu(z)-u(z)
     \le\frac{\Lambda
     t^2}{2(n+1)}
    \qquad(z\in\C^{n},\ 0<t\le
      t_0).
      \end{equation}
  For    fixed $\mathpart{t>0}$, the function $\mathcal G_{t}u$
     is  continuous:
    its    compactly    supported    kernel
     belongs to ${L}^{1}$,    and $\mathpart{{u}}$ is
     locally      bounded. Estimate
    \eqref{eq:global_green_average_bound_constant_density}    therefore  makes
     $\mathpart{u}$ a    locally uniform
    limit of    continuous functions as   $\mathpart{{t}\downarrow0}$.
   In particular, the
    given pointwise  function  $\mathpart{u}$ is continuous.

  We next recover   its  Laplacian by
   testing rather  than
    differentiating    $\mathpart{u}$.
  For $0\le\varphi\in C_c^\infty(\C^{n})$,
     the    evenness
  of the    averaging kernel   gives
    \begin{equation}%
  \label{eq:green_average_self_adjointness_on_test_functions}
   \int       u\frac{\mathcal
     G_t\varphi-\varphi}{t^2}\,d\xi
   =\int\varphi\frac{\mathcal
       G_{t}u-u}{t^2}\,d\xi.
     \end{equation}
    By  \eqref{eq:green_average_laplacian_expansion},   the
  quotient    on the    left
     converges    uniformly      to
    $\mathpart{\Delta_\R\varphi/[8(n+1)]}$,     with
   support in   a fixed
   compact     set.    The right side   is bounded by
     \eqref{eq:global_green_average_bound_constant_density},        so
 letting   $\mathpart{t\downarrow0}$ yields
       \begin{equation}%
      \label{eq:global_laplacian_bound_test_function_form}
    \mathpart{0}\le\int
     u\Delta_\R\varphi\,d\xi
       \le4\Lambda%
     \int\varphi\,d\xi.
  \end{equation}
    Thus
   \begin{equation}\label{eq:global_distributional_laplacian_bound}
    0\le\Delta_\R
   u\le4\Lambda
   \quad\text{in
      distributions   on  }\C^{n}.
     \end{equation}
  The order of limits    is   essential:
 $\mathpart{L\to\infty}$   removes    the
   boundary  remainder    before    division
 by    ${t}^{2}$.

    \subsection{Recover smoothness}
%
    %  -------------------------------------------------------------------------
   %

  %
%
     Interior Poisson    regularity    applied   to
    \eqref{eq:global_distributional_laplacian_bound}
   gives
     \begin{equation}\label{eq:local_second_order_sobolev_regularity}
    u\in
      W^{2,\psi}_{\mathrm{loc}}(\C^{n})
      \qquad\text{for
  every     }1<\psi<\infty.
       \end{equation}
 Choose $\mathpart{\psi>n({n}-1)}$.   By
      \eqref{eq:local_second_order_sobolev_regularity},      the
    local regularity    result
   of    B\l    ocki--Dinew
     \cite[Introduction, (2)]{BD}
   applies to  this plurisubharmonic
  Bedford--Taylor  solution with smooth
   positive density   one
      and gives $u\in   C^\infty(\C^{n})$. Since $\mathpart{{u}}$
   is       smooth,
      \eqref{eq:unit_density_bedford_taylor_measure}
    and \eqref{eq:global_distributional_laplacian_bound} hold  classically
    everywhere:
   \[
     \det    D^{2}_\C
     u=1,\qquad 0\le\Delta_\R   u\le4\Lambda.
  \]

      \subsection{Rescale to make the real Hessian constant}\label{sec:blowdown}
%
  %
%
 %  -------------------------------------------------------------------------
%
 % Apply    the uniform Hölder estimate  to two fixed
%
%  points after blowdown. The
    % resulting negative power of the   scale forces their real
%
 %  Hessians to  agree.

   For       $\mathpart{R\ge1}$,
      set
%
 %
%
%
% The factor multiplying the potential exactly cancels the dilation of its second derivatives.
   \begin{equation}\label{eq:quadratic_blowdown_rescaling}
   u_{R}(z)=R^{-2}u(Rz).
     \end{equation}
   Under the  rescaling \eqref{eq:quadratic_blowdown_rescaling},
   the   determinant
    equation \eqref{eq:unit_density_bedford_taylor_measure}       and
     the Laplacian bound
    \eqref{eq:global_distributional_laplacian_bound} remain    valid
   for    $u_{R}$.      Together  with
        \eqref{eq:global_growth_of_normalized_and_original_potentials},
   they  give
        \[
      \det          D^2_\C
   u_{R}=1,\qquad
   0\le\Delta_\R
  u_R\le4\Lambda,\qquad
     \|u_R\|_{L^%
 \infty(B_2)}\le5(M+m).
     \]
    Wang's    interior estimate
 \cite[Theorem 1.1]{Wang}, after  a
   fixed dilation   from
  ${B}_{2}$ to $B_{1}$, therefore gives
    $\mathpart{\alpha\in(0,1)}$  and      $\mathpart{C>0}$,
   independent  of $\mathpart{R}$, such that
%
% Uniformity in the dilation parameter forces the original Hessian oscillation to vanish.
   \begin{equation}\label{eq:uniform_real_hessian_holder_bound_on_blowdowns}
      [D^{2}_\R
    u_{R}]_{C^\alpha(B_1)}\le    C.
   \end{equation}
  This estimate controls the full
   real Hessian, so    it also
  controls the    pluriharmonic part    of the
      potential.

     Fix    $\mathpart{x,y\in\R^{2n}}$. For
     sufficiently     large $\mathpart{R}$, both
     $\mathpart{x/R}$    and
  $\mathpart{y/{R}}$ belong     to $B_{1}$. Since
   $D^{2}_\R   u_R(\xi)=D^2_\R   u(R\xi)$,
     \eqref{eq:uniform_real_hessian_holder_bound_on_blowdowns}       gives
     \[
      |D^2_\R
     u(x)-D^{2}_\R       u(y)|
   \le  C
     R^{-\alpha}|x-y|^\alpha\longrightarrow0.
      \]
   Thus    $D^{2}_\R u=H_0$ is  constant,
    and   integration gives
     \[
    u(\xi)=\tfrac12\xi^{T}H_0\xi%
  +b\cdot\xi+d.
 \]
      Let $P_{x}$ be    orthogonal    projection onto  the
     first  $\mathpart{n}$ real    coordinates.
  Evaluate    \eqref{eq:real_coordinate_quadratic_lower_bound}
    on  $\mathpart{\xi=se}$, divide  by  $s^{2}$ and let
      $\mathpart{s\to\infty}$. This    gives
 $e^TH_0e\ge 2\delta|P_xe|^{2}$  for  each unit
   vector $\mathpart{e}$. The  upper
    bound in \eqref{eq:global_growth_of_normalized_and_original_potentials}   gives
      $H_{0}\le2(M+m)\Id_{2n}$.     The
   complex Hessian has determinant one by
    \eqref{eq:unit_density_bedford_taylor_measure}.    This   proves
  Theorem~\ref{thm:main}.
   \appendix

       \section{Quadratic upper growth from a lower bound}\label{app:height}
%
 %
%
 %    =========================================================================
%
   % Appendix   A    |
%
% Deriving the    quadratic upper bound
%
% Previous note:      % Lemma A uses a height  estimate on one
     % Lemma A.1 uses a height  estimate on one
%
%   fixed ball. Its rescaling proof
% Previous note:       %  precedes   the contact-set argument for Lemma A.1.
      %  precedes   the contact-set argument for Lemma A.2.
%
 %  The density may vanish.
 % Equation index (printed number
   % and source    label)
      %     A.1       eq:condenser_capacity_from_sublevel_masses
%
%   A.2     eq:sublevelset_bound
%
%   A.3    eq:logarithmic_sublevel_capacity_bound
 %     A.4    eq:angular_energy_lower_bound
%
%     =========================================================================

  Write   $\mathpart{A(w)=(w_{i\bar
   j})}$,    $\omega_{w}=\sqrt{-1}\,\partial\bar\partial       w$,
   and $\beta=\omega_{|{z}|^{2}}$.      Thus $\omega_w^{n}=(\det
      A(w))\beta^{n}$  for       smooth $\mathpart{w}$,
    and $\beta^{n}=2^{n}  n!\,dV$    for Euclidean
 volume     $\mathpart{dV}$.
%
     %

%      \begin{growthlemma}[Quadratic  upper
     \begin{lemma}[Quadratic  upper
   growth]\label{lem:lower-upper}
   Fix $\mathpart{n\ge2}$, $\mathpart{m>0}$ and     $\mathpart{{F}>0}$. For       every
     $\mathpart{H\ge0}$ there are      constants
%
%    $\mathpart{A=A(n,m,F)\ge1}$,
   $\mathpart{b_{\mathrm h}=b_{\mathrm h}(n,m,F)\ge1}$,
     $L_{\mathrm     {h}}=L_{\mathrm h}(n,m,F,H)\ge1$,    and
    $\mathpart{M_*=M_*(n,m,F,H)}$ with    the    following
  property. If $L\ge  L_{\mathrm {h}}$,
 let $\mathpart{u}$ be  locally  bounded
%
%   and plurisubharmonic  on $\mathpart{B_{2AL}}$, with
  and plurisubharmonic  on $\mathpart{B_{2b_{\mathrm h}L}}$, with
       \[
      \omega_{u}^n=f\beta^n,\qquad
  0\le  f\le F\quad\text{a.e.},\qquad
    u(z)\ge      m|z|^2,\qquad
     u(0)\le    H .
     \]
  Here   $\mathpart{f}$ is  measurable   and the equation
  holds    in the Bedford--Taylor
   sense.      Then
     $u(z)\le  M_*(1+|z|^{2})$ for $\mathpart{|z|<2L}$.

      In  particular, if $\mathpart{{u}}$  is   entire with the same
      density bound  and
     $u(z)\ge m|z|^{2}-b$, then   $u(z)\le   C(1+|z|^{2})$ on
   $\C^{{n}}$ for   a finite
  constant       depending
  only    on $\mathpart{n,m,F,b,u(0)}$.
%
       %
%
%   \end{growthlemma}
  \end{lemma}
 
   The proof    reduces to     the
    following  estimate on    a fixed ball.   The
     linear-contact and condenser argument    for
   smooth    potentials   is   adapted
   from \cite{ContactNote}.  We  include    its  proof      and the
   passage     to weak
      potentials
      with  measurable   density.

  \begin{lemma}[Height from
    a    quadratic   lower bound]\label{lem:height}
     Fix       $\mathpart{{n}\ge2}$,   $\mathpart{m>0}$,
     and $\mathpart{F>0}$. There   are $\mathpart{D>2}$   and  $\mathpart{0<\kappa,\eta<1}$,
   depending    only on $\mathpart{n,m,F}$, with the following   property.
  Let  $\mathpart{u}$   be locally
    bounded and plurisubharmonic  on $\mathpart{B_{2{D}}}$ and
   suppose   that
     \[
    u(z)\ge  m|z|^{2}\quad(z\in B_{2D}),\qquad  0\le
     f\le      F\quad\text{a.e.    on }B_{2D}.
        \]
   Assume that the Bedford--Taylor measure  $\omega_{u}^{n}$
  equals  ${f}\beta^{n}$ on     $\mathpart{B_{2D}}$. If  $\mathpart{u(0)\le\eta}$,
 then $\mathpart{\sup_{B_\kappa}u\le2}$. The density may  be
 measurable   and   may vanish.
     \end{lemma}
%
%
    
   %     -------------------------------------------------------------------------
%
% The rescaling   divides   u by
%
 %  the square of the radius. This   keeps the
%
 % quadratic growth coefficient and   the upper
%
% bound on   the density    fixed.
%
  \begin{proof}[Proof     of Lemma~\ref{lem:lower-upper}
  from     Lemma~\ref{lem:height}]
   Take $\mathpart{D,\kappa,\eta}$      from  Lemma~\ref{lem:height} for
    $\mathpart{n,m,{F}}$,   and   set
      \[
%
%       A=4D/\kappa,\qquad
      b_{\mathrm h}=4D/\kappa,\qquad
      L_{\mathrm      {h}}=\max\{1,\sqrt{H/\eta}\},\qquad
     M_*=\max\{m,2L_{\mathrm
    h}^2,8/\kappa^2\}.
    \]
 For $\mathpart{|z|<2L}$, choose ${R}=\max\{L_{\mathrm
     {h}},2|z|/\kappa\}$.
%
%   Since    $\mathpart{DR\le    AL}$,    the function
  Since    $\mathpart{DR\le    b_{\mathrm h}L}$,    the function
   $\mathpart{V(\zeta)=R^{-2}u(R\zeta)}$   is defined
   on $\mathpart{B_{2D}}$.   It
   satisfies       $V\ge m|\zeta|^{2}$, $\mathpart{V(0)\le\eta}$  and
     $\omega_V^{n}=f(R\cdot)\beta^n$.      Also
   $\mathpart{|z|/R\le\kappa/2}$. Hence
     \[
     u(z)\le
      2R^2
     =\max\{2L_{\mathrm
     h}^2,8|{z}|^2/\kappa^2\}
     \le
      M_*(1+|z|^{2}).
   \]
    For   an    entire
     solution    set $\mathpart{U=u+{b}}$ and $\mathpart{{H}=U(0)\ge0}$.
  Then $U(z)\ge   m|z|^{2}$,
  and the    same $\mathpart{{M}_*}$ works    at every radius.
   Letting      $\mathpart{L\to\infty}$  and   subtracting $\mathpart{b}$
   proves    the assertion.
      \end{proof}
  We first record the capacity inequality    used in the proof.
  Let  $\mathpart{W}$ be     smooth
   and  strictly  plurisubharmonic,
     and  suppose $\mathpart{\{W\le {b}\}}$ is  compactly
      contained in  its domain.    Write
     $\Omega_{t}=\{W<t\}$      and
    $M(t)=\int_{\Omega_t}\omega_{W}^n$. For    regular
   values $\mathpart{a<b}$     with
   $\mathpart{K=\{W\le a\}}$    nonempty, let $\mathpart{h}$   be the    relative
      extremal    function of   $\mathpart{K}$
  in $\Omega_{{b}}$. It is $\mathpart{-1}$ on $\mathpart{K}$, is      zero
     on the outer  boundary, and      has
    zero Monge--Ampère measure  off   $\mathpart{K}$.    It equals zero
    on  components    of
     $\Omega_{b}$
 disjoint      from   $\mathpart{K}$.   Then
    \begin{equation}%
    \label{eq:condenser_capacity_from_sublevel_masses}
     \int_{\Omega_b}(-h)\omega_h^{n}
      =\operatorname{Cap}(K,\Omega_{b})
     \le\left(\int_a^b
      M(t)^{-1/n}\,dt\right)^{-n}.
     \end{equation}
   Here capacity      uses the
   normalization $\omega_{h}=\sqrt{-1}\partial\bar\partial   h$.
   We  prove \eqref{eq:condenser_capacity_from_sublevel_masses}
    at the     end of    the appendix.

   \begin{proof}[Proof of      Lemma~\ref{lem:height} for
     smooth      potentials]
  Decrease      $\mathpart{m}$
    to $\mathpart{\min({m},1)}$ and    take $\mathpart{D>3+\sqrt{5/m}}$.
    For the    smooth argument let
   $\mathpart{V}$   be strictly plurisubharmonic     on $B_{D}$ with
     $V\ge    m|z|^{2}$ and  $\det
   A(V)\le  F_{1}$, where $F_{1}=(F^{1/n}+1)^{n}$.
     Fix     $\mathpart{0<r<1/4}$
   and  suppose $V(0)\le  mr^{2}/4$.

 %  -------------------------------------------------------------------------
%
% Minimizers of affine   tilts lie  in the low
%
 %   sublevel. Varying the   slope
%
 %  fills a ball    in the image
%
 %  of the  gradient, and the area formula gives the
%
  % lower    bound
%
 %  on   contact volume.
%
      \smallskip\noindent\emph{A
    low  set of      positive volume.}
     For    every    real    covector
  $\mathpart{q}$ with $\mathpart{|q|\le   mr/4}$, minimize   $\mathpart{V(x)-q\cdot x}$
     on $\overline {B}_{r}\subset\mathbb R^{2n}$.
   All   minimizers lie in $B_{r}$:
    on the   boundary the minimized   function is   at
  least    $3mr^{2}/4>V(0)$.
     Their  union $E_{r}$ is compact and
      satisfies
    \[
       V(x)\le   V(0)+q\cdot x\le
    r^{2},\qquad
     D_{\mathbb R}^2V(x)\ge0,\qquad
    \nabla   V(x)=q .
   \]
   At a contact point put $P=D^{2}_{\mathbb R}V$   and let
    $\mathpart{J}$ be       multiplication
  by $\mathpart{\sqrt{-1}}$ on $\mathbb    {R}^{2n}$.  The
  matrix   $({P}+J^{T}PJ)/2$      represents
   the    Hermitian part    of $\mathpart{P}$ and
 has real  determinant   $4^n(\det A(V))^{2}$.
 Concavity   of $\mathpart{\log\det}$
     gives    $\det  P\le\det((P+J^{T}PJ)/2)$, first   for
     $\mathpart{P>0}$    and then   for
  $\mathpart{P\ge0}$  by adding $\mathpart{\epsilon I}$ and taking the    limit.
    The area       formula for   $\mathpart{\nabla
   V}$   now    gives
%
% The area formula converts gradient-image volume into a lower bound for contact-set volume.
     \begin{equation}%
     \label{eq:sublevelset_bound}
      |E_{r}|\ge c r^{2n},\qquad E_r\subset
    B_{r}\cap\{V\le   r^2\},
    \qquad
       c=|B_1|(m/8)^{2n}F_1^{-2}.
       \end{equation}
%
    %

      %    -------------------------------------------------------------------------
   % A high point forces oscillation along    complex circles. The
%
     %   low    contact    set
%
% provides the    volume
%
%  used in   the angular energy lower  bound.
%
        \smallskip\noindent\emph{A
 high    point    costs energy.}
 Choose   regular values  $r^{2}<a<2r^2$   and $\mathpart{4<b<5}$.
   For $\mathpart{W=V}$ all    sublevels
      up  to   $\mathpart{{b}}$   are    compactly
  contained in $B_{D}$.  Since
    $\Omega_{t}\subset  B_{\sqrt{t/m}}$,
      we      have $M(t)\le C_{n,m,F_{1}}t^n$.
    Consequently
       \begin{equation}%
      \label{eq:logarithmic_sublevel_capacity_bound}
       e(h):=\int_{\Omega_{b}}(-h)\omega_{h}^n\le
      C_{n,m,F_1}
    [\log(b/a)]^{-n},\qquad      h\ge
     (V-b)/(b-a).
  \end{equation}
   Taking    the  maximum of
 $\mathpart{(V-{b})/(b-a)}$  and $\mathpart{-1}$ gives an admissible
   competitor for the    relative
     extremal       function.

  If $\mathpart{V(p)=2}$ for $\mathpart{|p|\le   r/4}$, this competitor gives
  $\mathpart{h(p)>-3/4}$.    Put
   $\mathpart{w=(-h-3/4)_+}$ and  $\mathpart{g=1-4w}$   in $\Omega_{b}$ and    extend $\mathpart{{g}}$
     by   $\mathpart{1}$   outside.
  On  $E_{r}$ we have    $\mathpart{g=0}$.
  Every phase circle    about   $\mathpart{{p}}$ of    radius    at most
    $\mathpart{2r}$  contains    a   point    where
  $\mathpart{g=1}$.    This  holds if the circle   meets    the
   complement  of    $\Omega_{b}$.
   Otherwise its complex-line     disk lies in   $\Omega_{b}$
   by    the  maximum
      principle for    $\mathpart{V}$, and  the   submean inequality for $\mathpart{{h}}$
    gives   the
     assertion.

     Choose  $\mathpart{0<\theta<1}$ with
    $|B_{1}|\theta^{2n}\le c/2$.    Then
    $E=E_{r}\setminus   B_{\theta
   r}(p)$ has volume      at  least $\mathpart{cr^{2n}/2}$ and
   lies    in    $\mathpart{A=\{\theta
      r\le|{z}-p|\le2r\}}$.         Write
     $\mathpart{\Theta
      g(z)=\left.\frac{d}{ds}g(p+e^{\sqrt{-1}s}(z-p))\right|_{s=0}}$.
 The condenser   is  Lipschitz on    the
      closed    ring     by   Guan's    theorem
    \cite{Guan02}.  The  constant extension
     on $\mathpart{K}$    is Lipschitz  as well. On    a
     circle through $\mathpart{{E}}$, the    values $\mathpart{0}$   and $\mathpart{1}$
   of       $\mathpart{{g}}$ and Cauchy--Schwarz    give
   $\int_{0}^{2\pi}|\Theta
  g|^2\,ds\ge(2\pi)^{-1}$.
  Integration    in $\mathpart{z\in {E}}$    and rotation invariance on    $\mathpart{{A}}$
    therefore  yield
    $\int_{A}|\Theta
  g|^2\,dV\ge |E|/(4\pi^2)$. Set
   $\gamma=\sqrt{-1}\,\partial\bar\partial\log|z-p|^{2}$.
     At
     $\mathpart{z-{p}=(s,0,\ldots,0)}$,    $\mathpart{\gamma}$ has  transverse
    eigenvalues   $\mathpart{s^{-2}}$
      and one radial zero. The  remaining
     gradient    factor controls
   $|\Theta g|^{2}/(4s^2)$. By
 unitary invariance,
% The angular derivative detects the oscillation along circles centered at the high point.
       \begin{equation}\label{eq:angular_energy_lower_bound}
       \int_A     \sqrt{-1}\,\partial
   g\wedge\bar\partial g\wedge\gamma^{n-1}
     \ge  c_{n}\int_A   \frac{|\Theta
      g|^{2}}{|z-p|^{2n}}\,dV
      \ge
   c_{n,m,F_1}>0     .
  \end{equation}

     % -------------------------------------------------------------------------
  % Compare this upper bound    with the
%
 %  preceding angular  lower bound, using the
%
 % capacity estimate for
% the nested sublevels.
    \smallskip\noindent\emph{An     upper
      estimate for the same  energy.}
  Take  $\mathpart{R_*=e(\sqrt{5/m}+1)}$
  and $S=1+\log(R_*^2/(\theta   r)^{2})$.
 On    $\Omega_{b}$
     define
   \[
     \tau=\max\{\log(|z-p|^2/R_*^{2}),-S\},\qquad
   v=\max\{\tau,2Sh\}.
      \]
   Then $\mathpart{v}$    is   nonpositive
    and    plurisubharmonic, agrees   with   $\mathpart{2Sh}$ near the
      outer   boundary,   and
   satisfies $\mathpart{-v\le 2S(-h)}$. On  every    component of
  $\Omega_{b}$  disjoint  from $\mathpart{K}$,
     we   have $\mathpart{h=0}$ and   hence $\mathpart{v=0}$.
  On    the part
  of   $\mathpart{A}$ supporting the   gradient
     of $\mathpart{g}$   we have   $\mathpart{h\le-3/4}$,   so $\mathpart{2Sh<-S\le\tau}$. Also
      the    logarithm on
     $\mathpart{{A}}$ is     larger than    $\mathpart{-S}$.
   Hence      $\omega_{v}=\gamma$ there.

   For a    product $\mathpart{T}$  of $\mathpart{n-1}$ factors
   chosen    from  $\omega_h,\omega_{v}$,
     Stokes' formula     and
      the compact support    of $\mathpart{w}$    give
     \[
     \int_{\Omega_b}\sqrt{-1}\,\partial
   w\wedge\bar\partial w\wedge T
      =\int_{\Omega_{b}}
 w\,\omega_h\wedge  T
       \le\int_{\Omega_b}(-h)\omega_h\wedge
    T       .
    \]
 Indeed, first replace $\mathpart{(s-3/4)_+}$ by smooth   increasing
   Lipschitz  cutoffs
  $\mathpart{\chi(s)}$ vanishing   for $\mathpart{s\le3/4}$,
  with    derivatives    increasing to
   $\mathpart{\mathbf1_{\{s>3/4\}}}$.  Integrating    by      parts    with
    a primitive  of
    $\chi'(-{h})^{2}$  gives    the
    weight  $\int_0^{-h}\chi'(s)^{2}{d}s$     on $\omega_{{h}}$.
   This converges     uniformly to
     $\mathpart{w}$. The bounded-potential    integration by
    parts is    justified by    Bedford--Taylor
   approximation   \cite{BT}.
     Integration  by parts   for  bounded    zero-boundary potentials
    of finite    mass
       \cite[Section~2]{CegrellEnergy}
  gives
      \[
     \int_{\Omega_{b}}(-h)\omega_v\wedge
   T
     =\int_{\Omega_b}(-v)\omega_h\wedge
      T
    \le     2S\int_{\Omega_b}(-h)\omega_h\wedge
  T    .
       \]
  All      mixed masses are finite since    $\mathpart{h}$ has    finite
 condenser    mass and
 $\mathpart{v=2Sh}$ in      a    boundary
      collar.  Applying this inequality $\mathpart{n-1}$ times      gives
      \[
     \int_A\sqrt{-1}\,\partial
   g\wedge\bar\partial    g\wedge\gamma^{n-1}
       \le
   16(2S)^{n-1}e(h)
     \le
  \frac{C_{n,m,F_{1}}}{\log(1/r)}    .
  \]
   Choose   $\mathpart{{r}}$ small
    enough      to contradict \eqref{eq:angular_energy_lower_bound}.
    There can    then  be  no point of height       $\mathpart{2}$ in
    $\mathpart{B_{r/4}}$.  Continuity    along
  segments       from   $\mathpart{0}$  gives
     $\mathpart{{V}<2}$    there.
      \end{proof}

%    -------------------------------------------------------------------------
  % This is the weak  approximation step
   %
%
    %   in    the height lemma. The    density is
%
 %
%
    % only  measurable    here,
%
    %   so the comparison    must pass to the limit at the
%
% level      of
%
%   measures.
%
    \begin{proof}[Passage    to locally
   bounded   Bedford--Taylor      solutions]
   Take nonnegative smooth    radial mollifiers $\chi_{j}$ of
   mass   one with   radii
     decreasing to zero.  On a  neighborhood  of
   $\mathpart{\overline  B_{3D/2}}$ let
    \[
      U_{j}=u*\chi_j+j^{-1}(1+|z|^2),\qquad
        g_j=f^{1/n}*%
        \chi_j+j^{-1}.
   \]
  Discard finitely many  indices    so the    convolutions
   use only $\mathpart{B_{2D}}$.
   For every    positive  Hermitian
   $\mathpart{H}$ of    determinant    one, the mixed
  Monge--Ampère    inequality
    \cite{Dinew}      gives
   $\mathpart{n^{-1}\operatorname{tr}(H^{-1}{A}(u))\ge    f^{1/n}}$
   in       distributions, by
      testing     against    the    quadratic
     form  with    complex   Hessian    $\mathpart{H}$.
    After    convolution,
  the  inequality
   $\mathpart{n^{-1}\operatorname{tr}(H^{-1})\ge1}$    and
   the    identity
     \[
      \mathpart{(\det
      P)^{1/n}}=\inf_{H>0,\ \det   H=1}
     n^{-1}\operatorname{tr}(H^{-1}P)\qquad(P\ge0)
    \]
    give $\det  A(U_j)\ge    g_{j}^{n}$.
 The radial    submean property gives $U_{j}\ge u$
 and   $U_{j}\downarrow       u$ pointwise,   taking
   dilations of  a fixed    mollifier.
  On $B_{D}$,    solve   the smooth  strictly
  positive Dirichlet  equation
   $\det     A(v_{j})=g_j^n$,
     $v_j|_{\partial       B_{D}}=U_j$    \cite{CKNS}.
    Comparison       gives   $v_{j}\ge
  U_{j}\ge    u$. Paste    $v_{j}$ inside $B_{D}$ to   $U_{j}$
     outside,   obtaining
    a  plurisubharmonic    $V_{{j}}$ on $\mathpart{B_{3D/2}}$.
  The   psh pasting    lemma applies    because
     $v_{j}\ge    U_j$ inside   and    their
      boundary values   agree.    On  either open side of the  gluing
    sphere      its    measure
   dominates      $g_{j}^n\beta^n$.
     Positivity and   the zero    Lebesgue measure of the
   sphere give the same     inequality on
    all of $\mathpart{B_{3D/2}}$.
    We   claim $v_{j}(0)\to u(0)$.   The functions
    $V_{j}$  are locally  uniformly bounded
  above  and below, by   comparison    with $\mathpart{u}$ and the
  maximum    principle   on    $B_{{D}}$.
   Take any    subsequence converging in
  ${L}^{1}_{\rm loc}$ to       a plurisubharmonic
    $\mathpart{v}$, using      plurisubharmonic compactness.
   Then $\mathpart{v\ge {u}}$ and $\mathpart{v=u}$  on   the
 annulus     outside $\overline {B}_{D}$. These statements hold
    for   the canonical
    upper   semicontinuous
   representatives by  the submean characterization.
    For each $\mathpart{H}$ as above,  the    mixed inequality for $V_{j}$
    passes   to the
     distributional
       limit
   $\mathpart{n^{-1}\operatorname{tr}(H^{-1}A(v))\ge
    f^{1/n}}$,  since
  $g_{j}\to  f^{1/{n}}$ in    $L^{1}_{\rm loc}$. Convolve this
  family    of  linear
   inequalities with a
   radial mollifier $\mathpart{\chi_\epsilon}$ and  take its    pointwise
     infimum    over    $\mathpart{{H}}$.
   This    gives
     \[
    \det    A(v*\chi_\epsilon)\ge(f^{1/n}*\chi_\epsilon)^{n}
  .
   \]
  The  left-hand Monge--Ampère measures  converge
   by   decreasing
    Bedford--Taylor
    continuity   \cite{BT}. The right side converges       in
    ${L}^{1}_{\rm  loc}$    to $\mathpart{f}$  by boundedness. Thus
   $\omega_v^{n}\ge\omega_{u}^n$.
  Comparison  on $\mathpart{B_{5D/4}}$, whose
   boundary    has   a collar on which $\mathpart{v=u}$,
    gives $\mathpart{v\le    u}$ \cite{EGZ}.
   Every $L^{1}_{\rm  loc}$ limit is consequently $\mathpart{u}$.

 For  $\mathpart{0<s<D}$, subaveraging   $v_{j}(0)$ over   $B_{s}$
  and passing to the limit
    bounds its limsup by    the average      of $\mathpart{{u}}$  over $B_{s}$. Let
    $\mathpart{s\downarrow0}$       and
   use  $v_{j}(0)\ge {u}(0)$ to  prove   the claim. The
    functions $v_{j}$    satisfy
     $v_j\ge    m|z|^{2}$ and $0<\det A(v_j)\le F_{1}$. With the   radius
    from the   smooth
    proof    take $\eta=mr^{2}/8$
     and   $\mathpart{\kappa=r/4}$, recalling  that  $\mathpart{m}$ was      decreased
  to at most    one.  If $\mathpart{u(0)\le\eta}$, then
   $v_{j}(0)\le  mr^{2}/4$ eventually.
     The  smooth  estimate and $u\le v_{j}$ give
 the asserted bound.
    \end{proof}
   \begin{proof}[Proof      of
   the condenser inequality    \eqref{eq:condenser_capacity_from_sublevel_masses}]
    For a       regular $\mathpart{t\in(a,b]}$, let    $D_{t}$ be the
   union of the components of
    $\Omega_{t}$  meeting
      $\mathpart{K}$.   On    every  remaining component the condenser
    potential    is zero.    On the  ring
 $N_{t}=D_{t}\setminus       K$,  solve
     \[
     \omega_{P}^n=\epsilon\rho_t,\qquad
     \rho_t=(t-a)^{-n}\omega_{W}^n,\qquad
    P|_{\partial    K}=0,\quad  P|_{\partial
    D_t}=1,\qquad    0<\epsilon<1   .
   \]
  Guan's ring  theorem    \cite{Guan02}
   applies  componentwise.    The plates have
    smooth strongly   pseudoconvex boundaries,
    and $K\cap    D_{t}$ is    holomorphically
    convex   in    $D_{t}$ because it is a sublevel
   of the psh function    $\mathpart{W}$ on    a
   Stein   domain. The
    strict subsolution is $\mathpart{(W-a)/(t-a)}$.  Guan's    theorem
    gives     smooth regularized
   solutions      and $C^{1}$  convergence to the    homogeneous
    solution    as    $\mathpart{\epsilon\downarrow0}$.  On
  a closed interval  of regular values
   the   solutions  for      fixed   $\mathpart{\epsilon}$ depend
     differentiably    on $\mathpart{{t}}$: pull
  the rings     back by the level-set
  flow and     apply   the implicit  function   theorem
  to the elliptic  linearized  Dirichlet operator.
   Comparison    and the   maximum
     principle give $\mathpart{0\le
   {P}\le1}$.
    For     the flux
   calculation use    the real   one-form
   $\mathpart{\alpha_\phi=(\sqrt{-1}/2)(\bar\partial\phi-\partial\phi)}$,
      so
  $\mathpart{{d}\alpha_\phi=\omega_\phi}$.
 On the     outer boundary write
     $\mathpart{\bar\partial
   P=q\bar\partial W}$, $\mathpart{q\ge0}$, and set
    \[
      d\nu=\alpha_W\wedge\omega_W^{n-1}\big|_{\partial
    D_t},\qquad
        G_\epsilon(t)=\int_{\partial
   D_{t}}q^n\,d\nu,\qquad
      C_\epsilon(t)=G_\epsilon(t)+n\epsilon\int_{N_t}P\rho_t
   .
      \]
 The boundary   form is
     positive  with the outer orientation.
    Its    integral  is
   $\int_{D_{t}}\omega_W^n\le    M(t)$ by Stokes.
    The corresponding    outer   flux of  $\mathpart{P}$
      equals $\mathpart{G_\epsilon}$:   pull    back
   $\alpha_{P}=q\alpha_W$ to the
      boundary, differentiate   tangentially, and
   wedge      with the first
    factor.   In the    limit
      $G_\epsilon(t)\to
   C(t)=\operatorname{Cap}(K,\Omega_{t})$    by $C^{1}$      convergence
     and the outer  flux    formula    for
      the homogeneous    condenser  \cite{Guan02}.
      The  constant   extension   across $\mathpart{K}$ retains the plate
    mass.  The  homogeneous
    solution on the   ring is at   least $\mathpart{{P}}$ by comparison,
   with    equality on  the
  outer  boundary,
 so $\mathpart{G_\epsilon(t)\ge    C(t)}$.

   We    give the flux variation to fix the   coefficient
  of      the  boundary term.
   Let   $\mathpart{{Y}}$ be an   outer-boundary velocity with
      $\mathpart{YW=1}$, let    a dot denote
     differentiation       at
    fixed spatial     coordinates,  and   put $T=\omega_{P}^{n-1}$.
 Then    $\mathpart{\dot P=0}$ on the inner    boundary and $\mathpart{\dot    P=-q}$ on
 the  outer one.
 Differentiating  the    outer flux, using the  transport    formula
  and tangential
     Stokes,
     gives
   \[
     G_\epsilon'=     n\int_{\partial
   D_{t}}\alpha_{\dot P}\wedge    T
             +\epsilon\int_{\partial
  D_t}\iota_Y\rho_t     .
     \]
     Indeed the  fixed-boundary     derivative
    contributes $\mathpart{n}$ identical      terms     after
    tangential integration      by parts, and
  the       transport contribution is  the
     contraction    of
    $\omega_P^{n}$ with $\mathpart{{Y}}$.   Green's     identity    between $\mathpart{{P}}$    and
   $\mathpart{\dot     P}$, together  with
  $n\omega_{\dot   P}\wedge T=\epsilon\dot\rho_{t}$,   yields
     \[
         G_\epsilon'=-n\int_{\partial
   D_t}q^{n+1}d\nu
         +\epsilon\int_{N_{t}}P\dot\rho_t-n\epsilon\int_{N_t}\dot
       P\rho_t
       +\epsilon\int_{\partial
     D_t}\iota_{Y}\rho_t       .
        \]
  When  $n\epsilon\int_{N_{t}}P\rho_{t}$ is    differentiated,
    the terms     involving
   $\mathpart{\dot  P}$ cancel.   The    terms
   left besides   the first boundary    integral  are
 $\mathpart{O(\epsilon)}$   uniformly   on closed intervals
     of regular    values since
  $\mathpart{0\le       P\le1}$  and $\rho_t,\dot\rho_{t},Y$ are
     bounded        there. Hölder now gives
  \[
    \mathpart{C_%
        \epsilon'}\le
      -nG_\epsilon^{1+1/n}M(t)^{-1/n}+O(\epsilon).
     \]
   On   a  regular interval
  $[s_1,s_{2}]$ we   have
   $G_\epsilon(t)\ge      C(t)\ge
      C(s_{2})>0$       and
     $\mathpart{C_\epsilon-G_\epsilon=O(\epsilon)}$
    uniformly.
   Differentiate $\mathpart{C_\epsilon^{-1/n}}$,   integrate,
      and   let $\mathpart{{\epsilon}\downarrow0}$.
    It    follows
     that
    \[
        C(s_2)^{-1/n}%
     -C(s_1)^{-1/n}
      \ge     \int_{s_1}^{s_{2}}
    M(t)^{-1/n}\,dt   .
   \]
   The function $\mathpart{C(t)^{-1/n}}$ is    nondecreasing
  by capacity monotonicity.
      Its   increment   across $\mathpart{(a,b)}$ therefore dominates the sum of
    its   increments
   on disjoint    regular
  intervals. Sard's theorem   gives  full measure   to  their
    union,  so the last    integral extends over
    $\mathpart{(a,{b})}$. Discarding       the nonnegative
  initial   inverse capacity
    proves    \eqref{eq:condenser_capacity_from_sublevel_masses}.
   \end{proof}
      \section{An alternative geometric proof for smooth solutions}
   \label{sec:geometric-appendix}
%
     %    =========================================================================
%
% Appendix   B |
%
% Geometric proof for smooth  entire solutions
%
     % The Green representation and    the
%   disc averages give global metric  bounds.
%
    %    The   cutoff calculation
%
 %     then forces the Calabi energy to vanish.
      % Equation  index (printed number    and
%
%   source label)
%   B.1     eq:metric_laplacians_of_potential_and_weight
   %   B.2      eq:green_operator_barrier_inequality
%
 %      B.3    eq:dirichlet_green_representation
%
%       B.4    eq:green_measure_total_mass_identity
%
    %     B.5    eq:green_measure_quadratic_moment_identity
%
    %      B.6      eq:boundary_measure_vanishing_moment
%
 %   B.7          eq:limiting_probability_mass_and_moment
%
 %   B.8     eq:metric_bounds_on_probability_support
   %       B.9    eq:limiting_green_representation
%
  %    B.10     eq:translated_potential_trace_identity
%
   %   B.11     eq:translated_potential_defect_minimum
%
   %   B.12     eq:translated_metric_trace_tail_bound
%
 %   B.13    eq:complex_disc_green_mean_identity
%
 %    B.14   eq:global_metric_equivalence
%
    %   B.15   eq:calabi_energy_and_trace_inequalities
   %       B.16   eq:calabi_quotient_differential_inequality
%
%   B.17   eq:calabi_cutoff_gradient_laplacian_bounds
 %   B.18       eq:calabi_quotient_cutoff_maximum_inequality
%
  %
%
%    =========================================================================

 %
    This  appendix gives an      alternative proof
   of     the Liouville  theorem
     when the       entire
     solution    is      smooth. The    aim is to obtain a   global
      complex-Hessian bound    directly
     from    the quadratic growth,   using   a
      probability    measure   built from a
    Green    function and supported
   where        the
     metric is controlled.
    A Calabi estimate then      makes
  the  metric constant, and a  harmonic
  estimate    recovers
      the full potential.    We give  the geometric argument
    in    full
    below.
  The Bedford--Taylor    case    still
      requires    the   regularity   established
  in
 Sections~\ref{sec:approximation}--\ref{sec:regularity};   the
  argument below can also be used once   that
   regularity   is   known.

     \begin{theorem}[Smooth
  quadratic-growth   rigidity]
   \label{thm:appendix-smooth}
    Let    $\mathpart{{n}\ge2}$ and   let
    $w\in      C^\infty(\C^{n})$ be strictly plurisubharmonic    and  suppose  that
     \[
      \det
       D^2_{\C}w=1,\qquad
     w(z)\ge\delta\sum_{i=1}^n|x_{i}|^2,\qquad
   x_i=\operatorname{Re}z_i,\quad
     z\in\C^{n},\quad\delta>0.
  \]
  Then  $D^{2}_{\C}w$       is
  constant and $\mathpart{w}$  is   a real quadratic polynomial.
      Its  homogeneous    quadratic part is   at  least
     $\delta\sum_i|x_{i}|^2$.
    \end{theorem}
   Set  $\mathpart{m=\delta/2}$ and $\mathpart{u=w-q+m}$, where
  $\mathpart{q}$ is defined      in
     \eqref{eq:pluriharmonic_quadratic_growth_normalization}.
 Applying    \eqref{eq:normalized_two_sided_growth_on_ball}
  on expanding    balls
       gives
     $\mathpart{m\rho\le u\le M\rho}$ on $\C^{n}$
   for      a  finite    $\mathpart{M}$.
   Write   $g=D^{2}_{\C}u$        and
   let    $g_{0}=\Id_{n}$ be the  flat    Hermitian
    metric. All comparisons at     different
    points   use the   fixed flat
 coordinates,    or  equivalently parallel
    transport      for the flat
   connection   $\mathpart{D}$.  In
      this appendix
     \[
    \Delta_{g}\varphi=\tr(g^{-1}D^{2}_{\C}\varphi),\qquad
     S=\tr(g^{-1}g_0).
    \]
     Thus   $\Delta_{{g}}$ is      the complex
    metric     Laplacian,  whereas
 $\mathpart{\Delta_\R}$ retains its     earlier    meaning.
      Constants     denoted by    $\mathpart{{C}}$
   may    change  from
     line to line. The    constants    in the     final    uniform
    metric     comparison
  depend     only on $\mathpart{n,m}$.
       \subsection{Construct a probability measure on a controlled region}
       \label{sec:appendix-measure}
%
 % -------------------------------------------------------------------------

  %
%
     We construct     a probability    measure
      $\mathpart{{\mu}}$,    supported  where $\mathpart{{S}}$
   is bounded, and derive   the Green
   representation
      \eqref{eq:limiting_green_representation}.
     For  smooth    $\mathpart{{\varphi}}$   satisfying
    $\mathpart{|\varphi|\le
   C\rho}$ and    $\Delta_{g}\varphi\ge0$, this     formula
  expresses    $\mathpart{\varphi(x)}$  as its   average
  minus a nonnegative   integral.
  The  quadratic lower    bound
    supplies    a  positive barrier, and    the
    upper    bound will    later
 justify applying  the formula to translates
    of $\mathpart{u}$.    We  first work on   bounded  domains and
   show that   the    boundary
  contribution    vanishes
    for   functions of quadratic growth.

      The  determinant condition   in
   Theorem~\ref{thm:appendix-smooth}  and
   the      arithmetic--geometric mean inequality give $\mathpart{S\ge  n}$.
     Direct  differentiation   gives
     \begin{equation}\label{eq:metric_laplacians_of_potential_and_weight}
   \Delta_g  u=n,\qquad
   \Delta_g\rho=S,\qquad
      |\partial\rho|_g^{2}\le
  S\rho.
  \end{equation}
  Set    $\mathpart{\ell=\max\{n,4n/m\}}$,  and
   choose    a smooth function
      $\mathpart{\chi:[n,\infty)\to[0,1]}$
     that is  one on $\mathpart{[{n},\ell]}$   and
      zero    on
     $\mathpart{[2\ell,\infty)}$.   Define
        \[
  V=u-\frac    {m}2|{z}|^2,\qquad
    k=\frac{2\chi(S)}{m\rho}.
    \]
    The normalized   lower bound     $\mathpart{u\ge
      m\rho}$ gives $\mathpart{V\ge  m\rho/2}$. The  determinant      condition in that theorem and the
 arithmetic--geometric
    mean  inequality  give $\mathpart{S\ge   n}$.
     If    $\mathpart{S\le\ell}$,
   then    $\mathpart{\chi(S)=1}$,  so $\mathpart{kV\ge1}$ and   $\mathpart{n/S\le1}$.
    If  $\mathpart{S\ge\ell}$,    then
   $\mathpart{n/S\le      m/4}$ and   $\mathpart{kV\ge0}$.
   Since $\Delta_{g}V=n-mS/2$, these  two
    cases      imply
   \begin{equation}%
  \label{eq:green_operator_barrier_inequality}
       (-\Delta_{g}+kS)V\ge\frac
   m4S.
      \end{equation}
   Fix $x\in\C^{n}$.   For $\mathpart{R}$
   large   enough  that   $x\in\Omega_{R}
       =\{\rho<R^2\}$,
    let $G_{R}(x,y)$ be    the   positive Dirichlet
    Green     function of
      $-\Delta_{g}+kS$, normalized      by
   $(-\Delta_{g,y}+kS(y))G_{R}(x,y)=\delta_x$
  with   respect    to
   $d\mathrm{vol}_{{g}}(y)$. Define
  \[
   d\nu_R(y)=S(y)G_R(x,y)\,d\mathrm{vol}_{g}(y),
    \]
   and    let $d\beta_{R}$
   be the corresponding positive   conormal
    boundary     measure.    These measures
    exist    on each    bounded domain
  because  the coefficients
   are smooth and elliptic there,    and
       $\mathpart{kS%
      \ge0}$.
    The    metric Laplacian
   is self-adjoint    with    respect to its
   volume    measure.  Green's    formula
     therefore  reads
      \begin{equation}%
     \label{eq:dirichlet_green_representation}
   \varphi(x)%
    =
      \int_{\Omega_{R}}\left(k\varphi-\frac{\Delta_g\varphi}{S}\right)d\nu_R
     +\int_{\partial\Omega_R}\varphi%
    \,d\beta_R
       \end{equation}
   for    smooth    $\mathpart{\varphi}$ on
     $\overline{\Omega_{R}}$.
    Applying it
     to the  positive  function   $\mathpart{V}$ and using
        \eqref{eq:green_operator_barrier_inequality}
   yields
     \[
   \nu_{R}(\C^{n})\le\frac{4V(x)}m.
   \]
   Here the measures    on
     $\Omega_{R}$ are extended by   zero.
   Domain    monotonicity
    of positive Dirichlet Green      functions
    shows  that $\nu_{R}$ increases
   to a   finite    measure   $\mathpart{\nu}$.
      
    Testing    \eqref{eq:dirichlet_green_representation}    with
    $\mathpart{1}$ and      with  $\mathpart{\rho}$, using
  $\Delta_{g}\rho=S$
    from \eqref{eq:metric_laplacians_of_potential_and_weight}     and
 $\rho=R^{2}$     on
   $\partial\Omega_{{R}}$,    gives
     \begin{align}
 \int k\,d\nu_R+\int
     d\beta_R&=1,
    \label{eq:green_measure_total_mass_identity}%
      \\
     \int   k\rho\,d\nu_R+R^2\int
     d\beta_R
   &=\rho(x)+\nu_R(\C^{n}).
    \label{eq:green_measure_quadratic_moment_identity}
    \end{align}
  To control    the
     second        boundary term, choose
     $\mathpart{\vartheta\in
 C^\infty([0,\infty))}$,      with
   $\mathpart{0\le\vartheta\le1}$, equal to  zero
    on       $\mathpart{[0,1/4]}$
  and to one
     on $\mathpart{[3/4,\infty)}$.      Put
   $\vartheta_R=R^{2}\vartheta(\rho/R^2)$.
    For     large
 $\mathpart{R}$,  $\vartheta_{R}(x)=0$ and
    $\vartheta_{R}=R^2$
   on   $\partial\Omega_{R}$.
  By     \eqref{eq:metric_laplacians_of_potential_and_weight}
   and the chain    rule,
   \[
     \frac{\Delta_g%
   \vartheta_R}{S}
   =\vartheta'(%
    \rho/R^2)
      +\frac{\vartheta''(%
    \rho/R^2)}{R^2}
        \frac{|\partial\rho|_g^2}{S}
     \le      C\,\mathbf
   {1}_{\{\rho\ge      R^2/4\}}
    \quad\hbox{on
 }\Omega_R.
    \]
  Inserting       this    function
     in   \eqref{eq:dirichlet_green_representation}, using
 ${k}\vartheta_{R}\ge0$ and
    the      domain-monotonicity comparison
   $\nu_{R}\le\nu$,
  gives
  %
%
% Finiteness of the limiting measure makes its mass beyond the expanding radius vanish.
    \begin{equation}\label{eq:boundary_measure_vanishing_moment}
   0\le   R^2\int
  d\beta_R
  \le   C\nu(\{\rho\ge
      R^{2}/4\})\longrightarrow0.
    \end{equation}
      The      last limit   follows   from
  finiteness     of  $\mathpart{\nu}$.
   Define $\mathpart{d\mu=k\,d\nu}$.     Both
     $\nu=\nu_{{x}}$    and $\mu=\mu_{x}$ depend
   on the   fixed    evaluation   point $\mathpart{x}$.   Monotone
 convergence in
    \eqref{eq:green_measure_total_mass_identity}--\eqref{eq:green_measure_quadratic_moment_identity},  together
    with
     \eqref{eq:boundary_measure_vanishing_moment},
   gives
      \begin{equation}\label{eq:limiting_probability_mass_and_moment}
    \mu(%
      \C^{n})=1,\qquad
   \int\rho\,d\mu=\rho(x)+\nu(\C^{n})<\infty.
   \end{equation}
  Since    the   cutoff defining $\mathpart{k}$ vanishes when
  $\mathpart{S\ge2\ell}$,
     $\mathpart{d\mu=k\,d\nu}$
   is supported in $\mathpart{\{S\le2\ell\}}$.
  There    every    eigenvalue   of    $\mathpart{g}$
 is at  least $\mathpart{(2\ell)^{-1}}$,
      and   their     product
    is   one. Hence
  \begin{equation}%
     \label{eq:metric_bounds_on_probability_support}
     (2\ell)^{-1}g_0\le
    {g}\le(2\ell)^{n-1}g_{0}
    \quad\hbox{on
     }\supp\mu.
     \end{equation}

      We will   use the   limiting
   Green formula   for    a smooth real
     function    $\mathpart{\varphi}$
   satisfying $\mathpart{|\varphi|\le       C\rho}$ and
 $\Delta_{g}\varphi\ge0$.
    Its  boundary   integral   tends to zero
   by \eqref{eq:boundary_measure_vanishing_moment},    and     its
    $\mathpart{\mu}$-integral   is
   absolutely convergent
   by    \eqref{eq:limiting_probability_mass_and_moment}.
  Domain    monotonicity
  of    the   Dirichlet Green    functions    makes the
  nonnegative    integrals
     of $\Delta_{g}\varphi/S$  increase   with $\mathpart{R}$. Rearranging
 \eqref{eq:dirichlet_green_representation} shows that
     their limit is finite   and gives
     \begin{equation}\label{eq:limiting_green_representation}
    \varphi(x)=\int\varphi\,d\mu
   -\int\frac{\Delta_{g}%
  \varphi}{S}\,d\nu.
    \end{equation}
   \subsection{Use translations and disc averages to bound the metric}
      \label{sec:appendix-metric}
%
       %
%
     %   -------------------------------------------------------------------------

   %
%
     We now use   the probability measure  to
     compare   the metric    at
  $\mathpart{x}$    with   its   average over the  controlled region.
     Translating
   the   potential preserves its
   determinant, so  the matrix
    arithmetic--geometric  mean   inequality
   produces    an   averaging
    inequality.
    A uniform   tail estimate    then  permits its
   second-order     limit
   without assuming    a    global derivative bound.
    For a  flat    displacement
     ${h}\in\C^{n}$,     put
     \[
     \mathpart{J(h)}=\int
    u(y+h)\,d\mu(y).
     \]
   On      bounded    sets of $\mathpart{h}$, the
    bounds    for  the normalized potential $\mathpart{u}$
 and  the determinant
     equation  give
  $\mathpart{0<u(y+h)\le
   C\rho(y)}$  and
   $\Delta_{g}[u(\,\cdot+h)]\ge0$.
     Formula
        \eqref{eq:limiting_green_representation}
    therefore     gives
  \begin{equation}\label{eq:translated_potential_trace_identity}
     \mathpart{J(h)-u(x+h)}=
       \int\tr\bigl(g(y)^{-1}g(y+h)\bigr)\frac{d%
    \nu(y)}{S(y)}.
      \end{equation}
   Both   sides are
     finite; the   preceding bound  $\mathpart{u(y+h)\le   C\rho(y)}$
   and the    finite    moment    in
      \eqref{eq:limiting_probability_mass_and_moment}  make     the
     nonnegative right-hand   side uniformly bounded    for   bounded
 $\mathpart{h}$. Since   $\mathpart{\det g(y+h)=\det       g(y)=1}$, the  matrix
   arithmetic--geometric mean    inequality
  makes its trace integrand
   at   least $\mathpart{n}$, with equality   for
      $\mathpart{h=0}$. Subtracting \eqref{eq:translated_potential_trace_identity} at    $\mathpart{{h}=0}$ from its   value
   at         $\mathpart{h}$
      gives
      \begin{equation}\label{eq:translated_potential_defect_minimum}
    \mathpart{J(h)-u(x+h)}\ge
 J(0)-u(x).
  \end{equation}

    We    next justify   the
      required   second-order  averaging.
  On   $\mathpart{\supp{\mu}}$,  \eqref{eq:metric_bounds_on_probability_support}
    and         $\mathpart{S\le2\ell}$
      imply
      \[
    \tr\bigl(g_0^{-1}g(y%
     +h)\bigr)
 \le(2\ell)^{n}
   \frac{\tr\bigl(g(y)^{%
      -1}{g}(y+h)\bigr)}{S(y)}.
      \]
  Using    $\mathpart{d\mu=k\,d\nu}$,
   $\mathpart{k\le2/(m\rho)}$,    and
    \eqref{eq:translated_potential_trace_identity},     we
     obtain
      \begin{equation}%
      \label{eq:translated_metric_trace_tail_bound}
     \int_{\{\rho>R^{2}\}}\tr%
     \bigl(g_0^{-1}g(y+h)\bigr)\,d\mu(y)
   \le\frac{C}%
     {R^2},
     \end{equation}
   uniformly     for bounded $\mathpart{h}$. Here $\mathpart{C}$  may depend on
     $\mathpart{n,m,M}$,       the
     fixed point $\mathpart{{x}}$,      and    the
  chosen    bound for   $\mathpart{|{h}|}$, but  is independent
  of $\mathpart{{R}}$ and  of   $\mathpart{h}$
    within that     bound. The same argument  with
 $\mathpart{\rho\ge1}$   proves finiteness on   the
   whole space.
    On  compact   sets
   of $\mathpart{{y}}$, smoothness gives       continuity in $\mathpart{h}$.
   The  uniform
    tail    estimate      thus proves   continuity of
   $\mathpart{\int  g(y+{h})\,d\mu(y)}$,  entry   by
   entry. Positivity    bounds
  the size  of   every entry  by
  the    trace.

 Fix  a constant complex
 vector $\mathpart{v}$. The planar    Green
    formula on      the
    complex   disc   gives
    \begin{equation}\label{eq:complex_disc_green_mean_identity}
     \frac1{2\pi}\int_0^{2\pi}u(y+re^{\sqrt{-1}\theta}v)\,d\theta-u(y)
     =\frac{2}\pi\int_{|\zeta|<r}
  \log\frac      r{|\zeta|}\,g(y+\zeta
   v)[v]\,dA(\zeta).
     \end{equation}
   The  scalar   kernel on
    the    right  has   total mass $r^{2}$.
   Average    \eqref{eq:translated_potential_defect_minimum}
    over   $\mathpart{h=re^{\sqrt{-1}\theta}v}$
    and          use  \eqref{eq:complex_disc_green_mean_identity}
    at   $\mathpart{y}$ and   at $\mathpart{{x}}$.
    Tonelli's theorem    applies
    to the    nonnegative disc
      integrands. Their   integrals
  are    finite    by
  \eqref{eq:translated_metric_trace_tail_bound}.
  Consequently,
     \[
       0\le\frac{2}\pi\int_{|\zeta|<r}\log\frac
     r{|\zeta|}
  \left(\int    g(y+\zeta
 v)[v]\,d\mu(y)-g(x+\zeta    v)[v]\right)
    dA(\zeta).
     \]
  Divide      by $r^{2}$  and    let
        $\mathpart{r\downarrow0}$.
     Using   the continuity
     obtained from \eqref{eq:translated_metric_trace_tail_bound} and       the
    mass ${r}^{2}$ of    the
  kernel in \eqref{eq:complex_disc_green_mean_identity}, we  obtain
      \[
    \mathpart{g(x)[v]}\le\int
       g(y)[v]\,d\mu(y).
     \]
   Since $\mathpart{v}$ is  arbitrary,
 \eqref{eq:limiting_probability_mass_and_moment}
    and    \eqref{eq:metric_bounds_on_probability_support}
 imply
    \[
        g({x})\le\int
 g\,d\mu\le(2\ell)^{n-1}g_{0}.
     \]
   The    point    $\mathpart{x}$
   was    arbitrary. The    determinant   condition   in
      Theorem~\ref{thm:appendix-smooth} now  gives
   a   lower   eigenvalue
  bound  as
      well.    Thus
     \begin{equation}%
      \label{eq:global_metric_equivalence}
   C^{-1}g_0\le     g\le
      Cg_{0}\qquad\hbox{on      }\C^n.
    \end{equation}
 In  particular, $\mathpart{g}$ is complete.
      No   completeness
    assumption was used
      to construct   $\mathpart{\mu}$.
   \subsection{Apply the Calabi estimate to make the metric constant}
    \label{sec:appendix-calabi}
    %
%
   %     -------------------------------------------------------------------------
%
     % In this calculation P     is the
%
 %   quotient involving the Calabi energy. The
 % letter q    is  reserved for the pluriharmonic
%
  % normalization    used to  recover
%
   %   the
      %     potential.
%
  %

  The aim here is    to
  eliminate  the    derivatives of the       metric.
    The Calabi identities control     a
    nonnegative     third-derivative
 quantity   and    the
    bounded trace   of     $\mathpart{g}$. We   divide    that quantity
   by
     ${a}-\tr(g_{0}^{-1}g)$,   with
   $a>1+\sup_{\C^{n}}\tr(g_0^{-1}g)$.
    The      resulting
      quotient satisfies a   differential inequality  with
   a positive quadratic   term.    A cutoff
    on      expanding Euclidean balls
 then     shows that
     the    quotient vanishes.

     Let $\mathpart{\nabla}$ be the
   K\"ahler   connection of   $\mathpart{g}$, and let
    $\mathpart{B=\nabla-{D}}$ be   its     holomorphic
    connection    difference:
     \[
    B^k_{ij}=\sum_\ell(g^{-1})_{\ell
     k}\,D_{i} g_{j\bar\ell}.
  \]
      Set     $E=|B|_{g}^2$
   and  $T=\tr(g_{0}^{-1}g)$.
    The     determinant
     equation   in   Theorem~\ref{thm:appendix-smooth}
        gives
     $\mathpart{\operatorname{Ric}(g)=0}$,  and
  \eqref{eq:global_metric_equivalence}
   bounds    $\mathpart{T}$.
   The       flat-background
     Calabi identities      yield
      \begin{equation}\label{eq:calabi_energy_and_trace_inequalities}
    \Delta_gE\ge0,
     \qquad
    \Delta_{g}T\ge
 cE,\qquad
     |\partial
   T|_g^2\le   CE,
   \end{equation}
  where $\mathpart{c>0}$. For    completeness,
    flatness     of  $\mathpart{D}$     gives
   $\bar\nabla_bB^{k}_{ij}=-R^k{}_{i\bar    b
      j}$.
      The    contracted Bianchi
   identity expresses the    contracted  derivative
 of   this   curvature as a    derivative    of
      the Ricci      tensor; it vanishes
   because  the  determinant equation     in
    Theorem~\ref{thm:appendix-smooth}
   gives      $\mathpart{\operatorname{Ric}({g})=0}$.
     Commuting     contracted covariant
 derivatives acts  on each tensor
     index by  this    zero   Ricci       endomorphism.   Differentiating
   $|B|_g^{2}$ leaves
    nonnegative     covariant-derivative   squares,
   which    gives
  the first  inequality in
   \eqref{eq:calabi_energy_and_trace_inequalities}; see  also
     \cite[Section 2.4]{TWY}.
      Differentiating $\mathpart{\log\det g=0}$ once   in each
   complex   direction
    expresses    $\Delta_{g}T$
  as the nonnegative  squared norm   of
      $\mathpart{{D}g}$,   with    its derivative
     index contracted   by $g_{0}^{-1}$.
   The   norm $\mathpart{E}$ uses   $\mathpart{g^{-1}}$   in
   that     index. Uniform    equivalence
   in  \eqref{eq:global_metric_equivalence} gives
  the     second   inequality.
    The     third
  follows   from
  $D_iT=\tr(g_{0}^{-1}D_i    g)$ and
  Cauchy--Schwarz.

   Choose  $\mathpart{a>1+\sup    T}$, and
   put $\mathpart{P={E}/(a-T)}$.
  The   denominator    is bounded above       and
 below  by positive constants.
     Applying  the    product
   rule      to    $\mathpart{E=(a-T)P}$ and   using
  \eqref{eq:calabi_energy_and_trace_inequalities}
      gives
      \begin{equation}\label{eq:calabi_quotient_differential_inequality}
    \Delta_g     P\ge
     c_1P^2+
       \frac{2}{a-T}\operatorname{Re}\langle\partial
    T,\partial        P\rangle_{g},
      \qquad
      \frac{|\partial
      T|_{g}}{a-T}\le C\sqrt       P
     \end{equation}
 for    a constant
      $c_{1}>0$.
  Fix $x\in\C^{n}$    and,   on  the Euclidean
    ball     $B_{{R}}(x)$,    set
    $\phi=1-|{z}-x|^{2}/R^2$.    Uniform
  equivalence  in
    \eqref{eq:global_metric_equivalence}
     gives
%
% Each differentiation of the cutoff contributes one inverse power of its radius.
      \begin{equation}\label{eq:calabi_cutoff_gradient_laplacian_bounds}
   |\partial\phi|_{g}^2\le
    C/R^2,\qquad
     \Delta_g\phi\ge-C/R^{2}.
  \end{equation}
 The function
  $\phi^{2}P$ vanishes   at the boundary.
  At a positive interior
     maximum,
    $\mathpart{\partial
     P=-2{P}\,\partial\phi/\phi}$.
     The   product     rule and
   \eqref{eq:calabi_quotient_differential_inequality}    therefore   give
     \begin{align}
      0\ge\Delta_g(%
    \phi^2P)
     &\ge
   c_1\phi^{2}P^2
    -\frac{4\phi
     P}{a-T}
      \operatorname{Re}\langle\partial
   T,\partial\phi\rangle_g
    +2\phi
   P\,\Delta_g\phi-6P|\partial\phi|_g^2
     \nonumber\\
       &\ge\frac{c_1}{2}\phi^2P^2-\frac{C}{R^2}P.
    \label{eq:calabi_quotient_cutoff_maximum_inequality}
     \end{align}
    In    the   last step, the    drift term
    has    size at most
       ${C}\phi
   P^{3/2}|\partial\phi|_{g}$.
    Young's    inequality absorbs half the
     positive    quadratic
   term, and   \eqref{eq:calabi_cutoff_gradient_laplacian_bounds}
   bounds the    remainder.
   Dividing \eqref{eq:calabi_quotient_cutoff_maximum_inequality}
    by    $\mathpart{P>0}$   proves
   $\max_{B_R(x)}\phi^2P\le
     C/R^{2}$.
  If   the maximum is
 zero, the    same    bound is immediate.
 Since $\mathpart{\phi(x)=1}$, we    have
    ${P}(x)\le      C/R^{2}$.
   Letting    $\mathpart{{R}\to\infty}$  gives
    $\mathpart{P(x)=0}$.
    Thus     $\mathpart{B=0}$
    everywhere,        and $\mathpart{Dg=0}$.
      \subsection{Recover the quadratic potential}
       \label{sec:appendix-potential}
%
   %  -------------------------------------------------------------------------

   A    constant complex Hessian     leaves a
      pluriharmonic   part
  undetermined.   This       final
 step uses  quadratic growth   to show
     that    this    remaining
    part  is   also a polynomial of   degree at
    most   two,  completing the
 smooth    Liouville theorem.

   Let $\mathpart{{H}}$ be  the   constant   matrix
   $\mathpart{g}$, and   set $\mathpart{{v}=u-H[z]}$.
    Then $D^{2}_{\C}v=0$,   so $\mathpart{v}$
   is pluriharmonic and
 $\mathpart{\Delta_\R v=0}$.    The      growth    bounds obtained
   at    the beginning  of this   appendix give
     $|v(z)|\le     C(1+|z|^{2})$. For
  a   fixed  point  $\mathpart{x}$,
     the   harmonic   interior   derivative
      estimate gives
  \[
      |D^{3}_\R
    v(x)|
  \le\frac{C_n}{R^3}%
      \sup_{B_R(x)}|v|
     \le\frac{C_{x}(1+R^2)}%
     {R^3}\longrightarrow0.
     \]
   Consequently  $D^{3}_\R   v=0$ everywhere, and
    $\mathpart{v}$,   hence   $\mathpart{u}$,
    is     a real polynomial of  degree at
    most   two.
  Adding      $\mathpart{q-m}$
   gives    the       original   quadratic polynomial $\mathpart{{w}}$.
     For   each    $\mathpart{z}$, its homogeneous quadratic   part is
      $\lim_{s\to\infty}w(sz)/s^{2}$.
  The   assumed   lower   bound gives       the asserted
  estimate for    this limit.
     \section{A smooth counterexample under quadratic upper growth}
        \label{sec:smooth-counterexamples}
%
  %    =========================================================================
%
% Appendix    C |
%
    %  Smooth counterexample    with quadratic  upper growth
%
 % The potential is expressed using
   %  an  implicitly defined positive parameter.
%
 % Its  determinant  is one,
%
  %   but an eigenvalue  on   the real axis is unbounded.
%
% Equation index (printed
 % number  and source label)
%
%   C.1    eq:counterexample_determinant_and_growth
      %   C.2     eq:counterexample_implicit_parameters
     %   C.3      eq:counterexample_potential_formula
%
  %   C.4      eq:counterexample_parameter_derivatives
 %      C.5    eq:counterexample_potential_first_derivatives
%
  %    C.6    eq:counterexample_potential_second_derivatives
%
%   C.7       eq:counterexample_complex_hessian_matrix
%
   %      C.8    eq:counterexample_determinant_calculation
%
 %   C.9      eq:counterexample_implicit_parameter_bound
%
%   C.10     eq:counterexample_quadratic_upper_bound
%
 %    C.11     eq:counterexample_axis_hessian
%
  %   =========================================================================

    Li and Wu  \cite[Remark~1.2]{LiWu}
     prove that a   smooth   entire solution
    of     $\sigma_2(D^{2}_{\R}v)=1$
   with $|v(x)|\le   C(1+|x|^{2})$   is
 a    quadratic
       polynomial.
   We    construct  a smooth    positive counterexample
   to    the   corresponding
    statement   for    $\det
     D^{2}_{\C}v=1$ on    $\C^{2}$.
   Its extension to the complex
    $\sigma_{k}$ equation  is  explained
      in
       Remark~\ref{rem:smooth-counterexample}.
   Smoothness   is the advantage of this
  example     over
       \eqref{eq:coordinate_product_counterexample}%
      .
The name Taub--NUT comes from the Lorentzian solutions of
Taub \cite{TaubSpacetime} and Newman, Tamburino and Unti \cite{NewmanTamburinoUnti}.
   The   construction comes  from the
    pullback  of LeBrun's Taub--NUT
    potential with  parameter
 $m=\tfrac{1}2$ \cite[p.~302]{LeBrunTaubNUT}
     by   the
  holomorphic     map
     $({z}_1,z_2)%
    \mapsto(e^{z_{1}},z_2e^{-z_1})$.
   We verify    all the   required   properties
    directly   below.

      \begin{proposition}[Smooth
      counterexample]
   \label{prop:smooth-counterexample}
     There exists      a nonquadratic strictly
 plurisubharmonic   function
    $v\in    C^\infty(\C^{2})$
     satisfying
  \begin{equation}%
   \label{eq:counterexample_determinant_and_growth}
  \det
     D^{2}_{\C}v=1,
     \qquad
   0<v(z)\le8(1+|z|^2).
     \end{equation}
  Its complex     Hessian
   is   unbounded.
    \end{proposition}
    \begin{proof}
  \emph{Construction and
     smoothness.}
 Write    $x=\operatorname{Re}z_{1}$
     and
  $q=|z_{2}|^2$, and   define
    $\mathpart{s>0}$ and $\mathpart{t\ge0}$    by
%
  %
%
% For fixed q, the left-hand side increases strictly with the positive parameter s.
    \begin{equation}\label{eq:counterexample_implicit_parameters}
   \log    s+s-\frac
     {q}s=2x,
      \qquad     t=\frac
        {q}s.
  \end{equation}
     For      fixed
    $\mathpart{q\ge0}$,   the left-hand    side    of
 \eqref{eq:counterexample_implicit_parameters}
 increases from    $\mathpart{-\infty}$ to $\mathpart{+\infty}$,
    and its derivative with    respect
  to $\mathpart{{s}}$  is
      \[
     \frac{1}s+1+\frac
      q{s^2}=\frac{1+s+t}{s}>0.
        \]
     Consequently,      \eqref{eq:counterexample_implicit_parameters}
  has a  unique    positive
 solution $\mathpart{s}$,  which is   smooth   in  $\mathpart{(x,q)}$ by
     the implicit    function
  theorem. This argument
   applies also  at    $\mathpart{q=0}$    by extending $\mathpart{q}$
      locally   to   real
    values. We  therefore    obtain a positive smooth
   function   on all    of
     $\C^{2}$,  including  $\{z_{2}=0\}$, by  setting
       \begin{equation}\label{eq:counterexample_potential_formula}
      v=s+t+%
     \frac{1}2(s^2+t^2),
    \qquad
      K=1+s+t.
   \end{equation}
Here $s$ and $t$ are the squares of LeBrun's two radial parameters.
At $(\zeta_1,\zeta_2)=(e^{z_1},z_2e^{-z_1})$, his relations give
$s e^{s-t}=|\zeta_1|^2=e^{2x}$ and $st=|\zeta_1\zeta_2|^2=q$.
Taking the logarithm of the first identity and substituting $t=q/s$
gives \eqref{eq:counterexample_implicit_parameters}.
    \emph{The complex   Monge--Amp\`ere
   equation.}
       Differentiating
   \eqref{eq:counterexample_implicit_parameters}  gives
    \begin{equation}\label{eq:counterexample_parameter_derivatives}
  s_x=\frac{2s}{K},\qquad
    s_q=\frac{1}K,
  \qquad
   t_x=-\frac{2t}{K},\qquad
    t_q=%
     \frac{1+s}{sK}.
     \end{equation}
      Using
 \eqref{eq:counterexample_parameter_derivatives} in
   \eqref{eq:counterexample_potential_formula},
   we      obtain
    \begin{equation}\label{eq:counterexample_potential_first_derivatives}
      v_x=2(s-t),\qquad
 v_q=1+\frac{1}s.
   \end{equation}
     Differentiating
  \eqref{eq:counterexample_potential_first_derivatives}  and    using
   \eqref{eq:counterexample_parameter_derivatives}
  once more     gives
    \begin{equation}%
    \label{eq:counterexample_potential_second_derivatives}
      v_{xx}=\frac{4(s+t)}K,\qquad
   v_{xq}=-\frac2{sK},\qquad
   v_{qq}=-\frac1{s^{2}K}.
      \end{equation}
      By     \eqref{eq:counterexample_implicit_parameters}--\eqref{eq:counterexample_potential_formula},
  the      potential   $\mathpart{v}$
   depends on $\mathpart{z}$ only   through    ${x}=\operatorname{Re}z_{1}$
     and $q=|z_2|^{2}$.   Thus
    \eqref{eq:counterexample_potential_first_derivatives}--\eqref{eq:counterexample_potential_second_derivatives}
    give
        \begin{equation}\label{eq:counterexample_complex_hessian_matrix}
  D^{2}_{\C}v
   =\begin{pmatrix}
        \tfrac14v_{xx}&%
   \tfrac12z_2v_{xq}\\[2pt]
    \tfrac12\bar
    z_2v_{xq}&v_q+qv_{qq}
     \end{pmatrix}
     =\begin{pmatrix}
  \dfrac{s+t}{K}&%
  -\dfrac{z_2}{sK}\\[6pt]
   -\dfrac{\bar
    z_2}{sK}&1+\dfrac{1+s}{sK}
   \end{pmatrix}%
     .
  \end{equation}
      By
     \eqref{eq:counterexample_implicit_parameters}--\eqref{eq:counterexample_potential_formula},
     $\mathpart{q=st}$
    and   $\mathpart{(s+t)(1+s)-t=sK}$.
   Therefore      \eqref{eq:counterexample_complex_hessian_matrix}
    yields
       \begin{equation}%
   \label{eq:counterexample_determinant_calculation}
     \det
      D^2_{\C}v
    =\frac{s+t}{K}
     +\frac{(s+t)(1+s)-t}{sK^{2}}
   =\frac{s+t+1}{K}=1.
   \end{equation}
 The first diagonal
   entry  in \eqref{eq:counterexample_complex_hessian_matrix}
     is positive. Together   with
    \eqref{eq:counterexample_determinant_calculation},
   this proves  that $D^{2}_{\C}v$
   is    positive   definite.

   \emph{Quadratic    upper
    growth.}
      When   $\mathpart{s\ge1}$,
  \eqref{eq:counterexample_implicit_parameters}  gives
    \[
   s^2=q+2xs-s\log
   s
    \le         q+2|x|s\le
   q+2x^{2}+\tfrac12s^2.
   \]
     Combining   this with $s^{2}<1$       when    $\mathpart{s<1}$,
   we    obtain
      \begin{equation}%
        \label{eq:counterexample_implicit_parameter_bound}
      s^2\le1+2q+4x^{2}.
    \end{equation}
    Since      $\mathpart{\log s\le s}$,
  \eqref{eq:counterexample_implicit_parameters}     gives
   $\mathpart{t\le2s+2|x|}$.    Multiplying  by $\mathpart{t}$ and using
 $\mathpart{st=q}$ gives
       \[
       t^2\le2q%
    +2|x|t\le2q+2x^2+\tfrac12t^{2},
      \qquad\text{hence}\qquad
    t^2\le4q+4x^2.
   \]
    This        bound,
   \eqref{eq:counterexample_implicit_parameter_bound},      and
      $s+t\le1+\tfrac{1}2(s^2+t^2)$
      show    from
     \eqref{eq:counterexample_potential_formula}
       that
     \begin{equation}\label{eq:counterexample_quadratic_upper_bound}
   0<v\le1+s^2+t^2\le2+6q+8x^2
    \le8(1%
   +|z|^{2}).
    \end{equation}
   \emph{Nonquadraticity.}
 On $\{z_{2}=0\}$,
  \eqref{eq:counterexample_implicit_parameters}    and
   \eqref{eq:counterexample_complex_hessian_matrix}
   become
%
% On the axis the diagonal entries are reciprocal, although one grows without bound.
     \begin{equation}%
     \label{eq:counterexample_axis_hessian}
       \log
    s+s=2x,\qquad
   D^2_{%
       \C}v=
   \begin{pmatrix}
     \dfrac{s}%
     {1+s}&0\\[6pt]
   0&1+\dfrac{1}s
    \end{pmatrix}%
      .
      \end{equation}
   As    $\mathpart{x\to-\infty}$,    the
    first relation in
   \eqref{eq:counterexample_axis_hessian}
   forces  $\mathpart{s\to0}$.
     The  second   relation
 then  gives $\mathpart{{v}_{2\bar2}\to+\infty}$.
    Thus $\mathpart{v}$ has   an   unbounded complex   Hessian and
      is   not   quadratic.
        \end{proof}

    Along   $z_{1}=x\in\R$,
     $z_{2}=0$,         equations
     \eqref{eq:counterexample_implicit_parameters}--\eqref{eq:counterexample_potential_formula}
   give
     $\mathpart{t=0}$, $\mathpart{\log s+s=2x}$, and
   $v=s+s^{2}/2$. Thus    $\mathpart{{v}(x,0)\to0}$
 as $\mathpart{x\to-\infty}$, whereas the
   lower     bound   in
   Theorem~\ref{thm:main} would  give
    $v(x,0)\ge\delta    x^{2}$.
   It is also   not real convex. Indeed,
   writing      $a=\operatorname{Re}z_{2}$
     and   taking  $z_{2}$  real,    use
 $\mathpart{q=st}$    from \eqref{eq:counterexample_implicit_parameters}
    in
    \eqref{eq:counterexample_potential_first_derivatives}--\eqref{eq:counterexample_potential_second_derivatives}
      to
    obtain
     \[
      v_{aa}=2v_q%
      +4qv_{qq}
      =2+\frac2s-\frac{4t}{sK}
     =-\frac{2}{13}\quad\text{when
    }s=\frac12,\ t=5.
     \]
    Such      a  point exists
 by \eqref{eq:counterexample_implicit_parameters},    with
    $z_{2}=\sqrt{5/2}$     and
      $\mathpart{x=(\log(1/2)+1/2-5)/2}$.
    Thus this   example does not    conflict  with
   rigidity   results
    that    assume
     real convexity.
  %
%
   %
  
 %    =========================================================================
%
%   References
   %
%
 %    =========================================================================
%
 %
%
\section*{Acknowledgments}
The authors acknowledge the assistance of ChatGPT in this research.

   \medskip
     \noindent\begin{minipage}{\linewidth}
      \small%
     \raggedright
       \textsc{Jingchen
    Hu},    Sichuan University\\
  \href{mailto:jingchenhoo@gmail.com}{\nolinkurl{jingchenhoo@gmail.com}},
    \href{mailto:jingchenhu@scu.edu.cn}{\nolinkurl{jingchenhu@scu.edu.cn}}

   \smallskip
 \textsc{Li    Sheng}, Sichuan
   University\\
      \href{mailto:lsheng@scu.edu.cn}{\nolinkurl{lsheng@scu.edu.cn}}
    \end{minipage}
     \end{document}